\documentclass[english]{amsart}
\usepackage{preamble}
\usepackage{todonotes}
\title[The $\epsilon$lectrical Hopf algebra and categorified Fock space]{The quantum $\epsilon$lectrical Hopf algebra and categorification of Fock space}

\author{Jonas Nehme}
\address{J.N.: Max Planck Institute for Mathematics, Vivatgasse 7, 53115 Bonn, Germany, \texttt{nehme@mpim-bonn.mpg.de}}
\author{Catharina Stroppel}
\address{C.S.: Department of Mathematics University of Bonn, Endenicher Allee 60, 53115 Bonn, Germany, \texttt{stroppel@math.uni-bonn.de}}


\begin{document}
\maketitle 
\begin{abstract}
	In this article we introduce a generalization of the Khovanov--Lauda Rouquier algebras, the \emph{electric KLR algebras}. 
	These are superalgebras which connect to super Brauer algebras in the same way as ordinary KLR-algebras of type $A$ connect to symmetric group algebras. As super Brauer algebras are in Schur--Weyl duality with the periplectic  Lie superalgebras, the new algebras describe morphisms between refined translation functors  for this least understood family of  classical Lie superalgebras with reductive even part. 
	The electric KLR algebras are different from quiver Hecke superalgebras introduced by Kang--Kashiwara--Tsuchioka and do not categorify quantum groups. We show that they categorify a quantum version of a type $A$ electric Lie algebra. 
	The electrical  Lie algebras arose from the study of electrical networks. They recently appeared in the mathematical literature as Lie algebras of a new kind of electric Lie groups introduced by Lam and Pylyavskyy. 
	We give a definition of a quantum electric algebra and realise it as a coideal subalgebra in some quantum group. 
	We finally prove several categorification theorems: most prominently we use cyclotomic quotients of electric KLR algebras to categorify higher level Fock spaces. 
\end{abstract}

\setcounter{tocdepth}{1}
\tableofcontents

\section*{Introduction}
\subsection*{Background}
Many diagrammatic algebras arise from the representation theory of Lie algebras.
The most prominent example being Schur--Weyl duality for the group algebra of the symmetric group $\mathfrak{S}_d$ (which is an example of a diagram algebra) and $\lie{gl}(n)$.
Both act in the obvious way on $W^{\otimes d}$ for an $n$-dimensional vector space $W$ and centralize each other.
\begin{equation}\label{SW}
	\lie{gl}(n)\act W^{\otimes d}\tca \bbC\mathfrak{S}_d
\end{equation}
There are many more examples, for instance if we replace $W^{\otimes d}$ by mixed tensor powers $W^{\otimes d}\otimes (W^*)^{\otimes d'}$, we have to replace the right-hand side by the walled Brauer algebra $\mathrm{wBr}_{d,d'}$.
Moreover, for $\lie{o}(r)$ and $\lie{sp}(r)$, there exists a similar version with the Brauer algebra on the right-hand side.
All the above-mentioned examples share that $W^{\otimes d}$ (resp.~$W^{\otimes d}\otimes (W^*)^{\otimes d'}$) are semisimple and decompose as a bimodule into irreducible modules.

There also exists many similar statements for non-semisimple examples, see e.g.~\cite[\S3]{AST} for some arising in a classical setting.
An important rich source of further examples, see \cite{Sergeev}, comes from replacing Lie algebras by Lie superalgebras, e.g.~we have the in general non-semisimple generalization of \cref{SW} from \cite{BS12}, 
\begin{equation*}
	\lie{gl}(m|n)\act W^{\otimes d}\otimes (W^*)^{\otimes d'}\tca\mathrm{wBr}_{d,d'}.
\end{equation*}
We refer to \cite{Sergeev}, \cite{BS12}, \cite{DLZ18}, \cite{ES16b} for more examples and details in case of classical Lie (super)algebras and to \cite{JungKang} and  \cite{Moon}, \cite{KT}, \cite{Cou18}, \cite{women2} for the so-called queer and periplectic Lie superalgebras respectively. 
Another non-semisimple situation is provided by Higher Schur--Weyl duality
\begin{equation*}
	\lie{gl}(n)\act M(\lambda)\otimes W^{\otimes d}\tca A_d,
\end{equation*}
where $M(\lambda)$ is a (possibly parabolic) Verma module of highest weight $\lambda$ and $\bbC\mathfrak{S}_d$ gets replaced by (a cyclotomic quotient of) the degenerate affine Hecke algebra $A_d$, see \cite{ArakawaSuzuki}, \cite{BK08} with its orthogonal and symplectic versions \cite{ESVW}, \cite{RuiSong}. 
These examples generalise to the corresponding super versions as well, even to the queer and periplectic Lie superalgebras,  \cite{HillSussan}, \cite{women2}.

The connection to category $\mathcal{O}$ for the classical general linear, symplectic and orthogonal Lie algebras has a huge advantage. 
Namely, in \cite{BGS96}, it was shown, using the geometry of mixed perverse sheaves, that category $\mathcal{O}$ for a semisimple complex Lie algebra admits a Koszul grading. 
And this grading then induces a grading on the centralizing algebra appearing in higher Schur--Weyl duality. 
For the queer and periplectic Lie algebras such a connection to geometry is missing (and at least in the periplectic case also not to expect). 
\subsection*{Modern perspective}
A more recent way to obtain this grading is by starting on the opposite side of the duality: a grading is constructed directly on the involved algebra $A_d$ or at least on a completed version thereof, which induces then a grading on the endomorphism algebra of the involved Lie theoretic object. 
In practice, this is often done by refining the endofunctor $\_\otimes W$ by considering its generalised eigenspaces with respect to the action of a Casimir operator (or equivalently a commutative subalgebra generated by Jucys--Murphy type elements in the analog of $A_d$). 
The refined endofunctors often satisfy interesting Serre type relations. 
This lead to categorifications of Lie algebra or quantum group actions, see e.g.~\cite{Webster}, \cite{ICM} for some examples.  

In the case of $\lie{gl}(n)$, the graded version of the degenerate affine Hecke algebra $A_d$ is (strictly speaking after some completion) given by (a completion of) the KLR-algebra $R_d$, which was introduced by Khovanov--Lauda and Rouquier in \cite{KL09}, \cite{Rouquier}. 
The precise isomorphism theorems can be found in  \cite{Rouquier}, \cite{BK09}, \cite{KKK}.

A main reason for introducing $R_d$ in \cite{KL09}, \cite{Rouquier} was to categorify the positive half of the quantum group for $\lie{gl}_\infty$, i.e.~by realizing it as the Grothendieck group of $\cC=\bigoplus_{d\geq 0}R_d\proj$, the categories of finitely generated projective $R_d$-modules. 
By passing to cyclotomic quotients of $R_d$, one obtains categories on which $\cC$ acts. They categorify the action of $\cUq(\lie{n}_+)$ on the Fock space $\Fock$ of semiinfinite wedges or alternatively of partitions as defined in \cite{LTFock}, \cite{KMS}   \cite{ArikiLectures}, as well as higher level versions. We refer to  \cite{Mathascycl} for an overview.

On the other hand, a similar categorification result was obtained in \cite{BSsuper} in different terms. By replacing $\cC'$ by  finite dimensional representations of $\lie{gl}(m|n)$, the authors categorified  the tensor product $\wedge^m V^\circledast\otimes \wedge^n V$ of exterior powers of the (dual) natural representation of the quantum group for $\lie{gl}_\infty$, see also \cite{Brundanold}.

\subsection*{Goal of the paper}
We want to develop here a categorification story, similar to both, the KLR construction  \cite{KL09}, \cite{Rouquier} and the categorification from  \cite{BSsuper} using categories of representations of Lie super algebras, but now for the periplectic Lie superalgebra $\lie{p}(n)$  instead of $\lie{gl}(m|n)$. 
The replacement of the quantum group of $\lie{gl}_\infty$ will be defined by the relations amongst the refined endomorphisms for $\lie{p}(n)$.

Recall that simple Lie superalgebras were classified by Kac in \cite{Kac77}. To include $\lie{gl}(m|n)$  we prefer to work instead with quasi-reductive Lie superalgebras i.e.~with reductive even parts, \cite{Sergquasireductive}. The classical quasi-reductive Lie superalgebras  are then controlled by  four infinite families. These are $\lie{gl}(m|n)$ and $\lie{osp}(r|2n)$ which are direct super analogues of $\lie{gl}(m)$ and $\lie{o}(r)$ (resp.~$\lie{sp}(2n)$), the queer Lie superalgebra $\lie{q}(n)$ and the periplectic Lie superalgebra $\lie{p}(n)$. We take a closer look at the latter and develop a similar story as mentioned above. This completes,  with \cite{NehmeDiss}, the treatment of all these classical quasi-reductive  Lie superalgebras. For $\lie{osp}(r|2n)$ this was done in  \cite{ESVW}, \cite{ES17}, \cite{ES21}, and for $\lie{q}(n)$ it is achieved in \cite{NehmeDiss}. 

\subsection*{The electric KLR algebras}
The periplectic Lie superalgebra is the Lie superalgebra preserving an odd non-degenerate bilinear form $\beta$ on an $(n|n)$-dimensional vector superspace $W$.

In this setting, the counterpart for Schur--Weyl duality is provided by the \emph{super Brauer category $\sBr$}, \cite{Moon}, \cite{KT}, \cite{Cou18}, \cite{women2}. 

The super Brauer category $\sBr$ is the $\bbC$-linear strict monoidal supercategory generated by one object $\ast$ and odd morphisms 
$\flat=\begin{tikzpicture}[myscale=0.6, line width=\fw]
\draw (0,0)..controls +(0.2,0.8) and +(-0.2,0.8)..(1,0);
\end{tikzpicture}$ and $\flat^\ast=\begin{tikzpicture}[myscale=0.6, line width=\fw]
	\draw (0,0)..controls +(0.2,-0.8) and +(-0.2,-0.8)..(1,0);
\end{tikzpicture}$ as well as the even morphism $s=\begin{tikzpicture}[myscale=0.6, line width=\fw]
	\draw (0,0)--(1,1) (1,0)--(0,1);
\end{tikzpicture}$ subject to the relations
\begin{center}
	%\begin{tabularx}{\textwidth}{l>{\centering\arraybackslash}Xl>{\centering\arraybackslash}X}
	\refstepcounter{svw}(\thesvw)\label{vwinverse}
	$\begin{tikzpicture}[myscale=0.5, line width=\fw]
		\draw (0,0)--(1,1)--(0,2);
		\draw (1,0)--(0,1)--(1,2);
	\end{tikzpicture}=
	\begin{tikzpicture}[myscale=0.5, line width=\fw]
		\draw (0,0)--(0,2);
		\draw (1,0)--(1,2);
	\end{tikzpicture}\quad\quad\quad$\refstepcounter{svw}(\thesvw)\label{vwbraid}
	$\begin{tikzpicture}[myscale=0.5, line width=\fw]
		\draw (0,0)--(1,1)--(2,2)--(2,3);
		\draw (1,0)--(0,1)--(0,2)--(1,3);
		\draw (2,0)--(2,1)--(1,2)--(0,3);
	\end{tikzpicture}=
	\begin{tikzpicture}[myscale=0.5, line width=\fw]
		\draw (0,0)--(0,1)--(1,2)--(2,3);
		\draw (1,0)--(2,1)--(2,2)--(1,3);
		\draw (2,0)--(1,1)--(0,2)--(0,3);
	\end{tikzpicture}\quad\quad\quad$
	\refstepcounter{svw}(\thesvw)\label{vwuntwist}
	$\begin{tikzpicture}[myscale=0.5, line width=\fw]
		\draw (2,2)--(1,1)..controls +(0.2,-0.8) and +(-0.2,-0.8)..(2,1)--(1,2);
	\end{tikzpicture}= 
	-\begin{tikzpicture}[myscale=0.5, line width=\fw]
		\draw (0,2)--(0,1)..controls +(0.2,-0.8) and +(-0.2,-0.8)..(1,1)--(1,2);
	\end{tikzpicture}$\\
	\refstepcounter{svw}(\thesvw)\label{vwtangle}
	$\begin{tikzpicture}[myscale=0.5, line width=\fw]
		\draw (0,2)--(1,1)..controls +(0.2,-0.8) and +(-0.2,-0.8)..(2,1)--(2,2) (1,2)--(0,1)--(0,0.5);
	\end{tikzpicture}= 
	\begin{tikzpicture}[myscale=0.5, line width=\fw]
		\draw (0,2)--(0,1)..controls +(0.2,-0.8) and +(-0.2,-0.8)..(1,1)--(2,2) (1,2)--(2,1)--(2,0.5);
	\end{tikzpicture}\quad\quad\quad$
	\refstepcounter{svw}(\thesvw)\label{vwsnake}
	$-\begin{tikzpicture}[myscale=0.5, line width=\fw]
		\draw (0,0)--(0,1)..controls +(0.2,0.8) and +(-0.2,0.8)..(1,1)..controls +(0.2,-0.8) and +(-0.2,-0.8)..(2,1)--(2,2);
	\end{tikzpicture}= \begin{tikzpicture}[myscale=0.6, line width=\fw]
		\draw (0,0)--(0,2);
	\end{tikzpicture}=
	\begin{tikzpicture}[myscale=0.5, line width=\fw]
		\draw (0,2)--(0,1)..controls +(0.2,-0.8) and +(-0.2,-0.8)..(1,1)..controls +(0.2,0.8) and +(-0.2,0.8)..(2,1)--(2,0);
	\end{tikzpicture}$
	%\end{tabularx}
\end{center}

The bilinear form $\beta$ provides an isomorphism $W\cong\Pi W^*$ and gives rise to an odd superadjunction $(\_\otimes W, \_\otimes W)$.
From \cref{vwtangle} and \cref{vwuntwist} using \cref{vwsnake} it is easy to deduce the following additional relations (keeping in mind that $\flat$ and $\flat^\ast$ are odd).
\begin{center}
	\begin{tabularx}{\textwidth}{l>{\centering\arraybackslash}Xl>{\centering\arraybackslash}X}
		\refstepcounter{equation}(\theequation)\label{vwtangletwo}&
		$\begin{tikzpicture}[myscale=0.6, yscale=-1, line width=\fw]
			\draw (0,2)--(1,1)..controls +(0.2,-0.8) and +(-0.2,-0.8)..(2,1)--(2,2) (1,2)--(0,1)--(0,0.5);
		\end{tikzpicture}= 
		\begin{tikzpicture}[myscale=0.6, yscale=-1, line width=\fw]
			\draw (0,2)--(0,1)..controls +(0.2,-0.8) and +(-0.2,-0.8)..(1,1)--(2,2) (1,2)--(2,1)--(2,0.5);
		\end{tikzpicture}$&
		\refstepcounter{equation}(\theequation)\label{vwuntwisttwo}&
		$\begin{tikzpicture}[myscale=0.6, yscale=-1, line width=\fw]
			\draw (2,2)--(1,1)..controls +(0.2,-0.8) and +(-0.2,-0.8)..(2,1)--(1,2);
		\end{tikzpicture}= 
		\begin{tikzpicture}[myscale=0.6,yscale=-1, line width=\fw]
			\draw (0,2)--(0,1)..controls +(0.2,-0.8) and +(-0.2,-0.8)..(1,1)--(1,2);
		\end{tikzpicture}$
	\end{tabularx}
\end{center}

Denoting by $\Fund$ the category with objects $W^{\otimes d}$ and \emph{all} morphisms (not necessarily degree preserving), there exists a full monoidal functor 
\begin{equation}\label{schurweyl}\tag{SW}
	\FunctorsBrtopn\colon\sBr\to\Fund,
\end{equation}
see e.g.~\cite{Cou18}*{Theorem 8.3.1}, \cite{KT}. Recently, also a web calculus was developed for tensor products of symmetric and exterior powers of $W$ in \cite{DKM}.

The degenerate affine version of $\sBr$ is the \emph{super $VW$-category} introduced in \cite{women2}. It is the $\bbC$-linear strict monoidal supercategory $\sVW$ generated by a single object $\ast$ and morphisms 
$\flat=\begin{tikzpicture}[myscale=0.5, line width=\fw]
\draw (0,0)..controls +(0.2,0.8) and +(-0.2,0.8)..(1,0);
\end{tikzpicture}$, $\flat^\ast=\begin{tikzpicture}[myscale=0.6, line width=\fw]
	\draw (0,0)..controls +(0.2,-0.8) and +(-0.2,-0.8)..(1,0);
\end{tikzpicture}$, $s=\begin{tikzpicture}[myscale=0.5, line width=\fw]
	\draw (0,0)--(1,1) (1,0)--(0,1);
\end{tikzpicture}$ as above and an additional even morphism $y=\begin{tikzpicture}[myscale=0.5, line width=\fw]
	\draw (0,0) -- (0,1) node[midway] (a) {};
	\fill (a) circle(\dws);
\end{tikzpicture}$ subject to the relations \cref{vwinverse,vwbraid,vwtangle,vwuntwist,vwbraid} together with two additional relations:
\begin{center}
	\begin{tabularx}{\textwidth}{l>{\centering\arraybackslash}Xl>{\centering\arraybackslash}X}
		\refstepcounter{equation}(\theequation)\label{vwdotcap}&
		$\begin{tikzpicture}[myscale=0.6, line width=\fw]
			\draw (0,0)..controls +(0.2,0.8) and +(-0.2,0.8)..(1,0) node[pos=0.9] (a) {};
			\fill (a) circle(\dws);
		\end{tikzpicture}= \begin{tikzpicture}[myscale=0.5, line width=\fw]
			\draw (0,0)..controls +(0.2,0.8) and +(-0.2,0.8)..(1,0) node[pos=0.1] (a) {};
			\fill (a) circle(\dws);
		\end{tikzpicture}+ \begin{tikzpicture}[myscale=0.5, line width=\fw]
			\draw (0,0)..controls +(0.2,0.8) and +(-0.2,0.8)..(1,0);
		\end{tikzpicture}$&
		\refstepcounter{equation}(\theequation)\label{vwdotcrossing}&
		$\begin{tikzpicture}[myscale=0.5, line width=\fw]
			\draw (0,0)--(0,2) (1,0)--(1,2) node[midway] (a) {};
			\fill (a) circle(\dws);
		\end{tikzpicture}= 
		\begin{tikzpicture}[myscale=0.5, line width=\fw]
			\draw (0,0)--(1,1)--(0,2) (1,0)--(0,1)--(1,2);
			\fill (0,1) circle(\dws);
		\end{tikzpicture}+ 
		\begin{tikzpicture}[myscale=0.5, line width=\fw]
			\draw (0,0)--(1,2) (1,0)--(0,2);
		\end{tikzpicture}+ 
		\begin{tikzpicture}[myscale=0.5, line width=\fw]
			\draw (0,0)..controls +(0.2,0.8) and +(-0.2,0.8)..(1,0) (0,2)..controls +(0.2,-0.8) and +(-0.2,-0.8)..(1,2);
		\end{tikzpicture}$
	\end{tabularx}
\end{center}
The relations \cref{vwdotcap} and \cref{vwsnake} imply then also the following relation 
\begin{center}
	\begin{tabularx}{\textwidth}{l>{\centering\arraybackslash}Xl>{\centering\arraybackslash}X}
		\refstepcounter{equation}(\theequation)\label{vwdotcup}&$
		\begin{tikzpicture}[myscale=0.5,yscale=-1, line width=\fw]
			\draw (0,0)..controls +(0.2,0.8) and +(-0.2,0.8)..(1,0) node[pos=0.9] (a) {};
			\fill (a) circle(\dws);
		\end{tikzpicture}= \begin{tikzpicture}[myscale=0.5,yscale=-1, line width=\fw]
			\draw (0,0)..controls +(0.2,0.8) and +(-0.2,0.8)..(1,0) node[pos=0.1] (a) {};
			\fill (a) circle(\dws);
		\end{tikzpicture} -\begin{tikzpicture}[myscale=0.5,yscale=-1, line width=\fw]
			\draw (0,0)..controls +(0.2,0.8) and +(-0.2,0.8)..(1,0);
		\end{tikzpicture}$.&\makebox[\widthof{(9)}]{}&
	\end{tabularx}
\end{center}
In \cite{women2}, it was shown that there is a monoidal functor $\Psi_n\colon\sVW\to\End(\lie{p}(n)\mmod)$ from $\sVW$ to the category of endofunctors of $\lie{p}(n)\mmod$ (the monoidal structure given by composition of endofunctors).

We will introduce in \cref{defKLRcorrect} the \emph{electric KLR category} as a new analogue of the KLR category  from  \cite{KL09}, \cite{Rouquier} and in \cref{locallyunital} an analogue of the KLR algebra, that is a graded version $\sR$ of $\sVW$.
\subsection*{Cyclotomic quotients}
As in the KLR setting above we also define and study cyclotomic quotients: Given a polynomial $p(x)=\sum_{i=0}^\level \delta_ix^i$ with $\delta_i\in\bbR$  the cyclotomic quotient $\sVWcyc$ is the quotient of $\sVW$ by the right tensor ideal generated by 
\begin{equation}\label{vwcycrelation}
	\sum_{i=0}^n a_i\cdot\begin{tikzpicture}[line width=\fw, myscale=0.7]
		\draw (0,0)--(0,1) node[midway] (a) {};
		\fill (a) circle(\dws);
		\node[anchor=west] at (a) {$i$};
	\end{tikzpicture}.
\end{equation}
Let $\mathrm{Kar}(\sVWcyc)$ be the Karoubian envelope of $\sVWcyc$. We then show for any charge vector $\chargevec$ as in \cref{defgeneric} an isomorphism theorem for level $\ell$ cyclotomic quotients:
\begin{mainthm}[\cref{isoncvw}]\label{isoncvwintro}
	There is  a morphism of algebras $\Phi\colon\sRcyc\to\mathrm{Kar}(\sVWcyc)$. 
	It is an isomorphism after passing to the additive Karoubian envelopes. 
\end{mainthm}
The proof involves some further development of the combinatorics of multi-up-down-tableaux from \cite{Cou18}.  The main byproduct is the construction of a triangular basis which turns the algebras into graded based quasi-hereditary algebras, \cref{sRcycisquher}. 
Then we can ask what the categories of projective modules categorify. 
\subsection*{Quantum electric algebras}  
We determine the relations in the split Grothendieck group of the projective modules for the electric KLR algebra in \cref{relationsinH} and prove the Categorification \cref{categorification} which, sloppily formulated,  is: 
\begin{mainthm}
	The analogue of the positive half of the quantum group for $\lie{gl}_\infty$ from the KLR setting is, in the electric KLR setting,  the quantum electric algebra $\el$. 
\end{mainthm}
Electric Lie algebras arose from the study of electrical networks and the modelling of their behaviour as pioneered by \cite{kennelly99}.  
Lam and Pylyavskyy introduced in \cite{LamPy} what they call \emph{electrical Lie groups}.  
These Lie groups, or more precisely their nonnegative parts, act on the space of planar electrical networks via combinatorial operations which were considered in  \cite{CIM}. 
The corresponding electrical Lie algebras are obtained by deforming the Serre relations of a semisimple Lie algebra in a way suggested by the star-triangle transformation of electrical networks, \cite{kennelly99}.  
So far these Lie groups and Lie algebras were mostly studied by physicists, see e.g.~\cite{Caiphysics}, but recently also appeared in the mathematical literature, \cite{LamPy}, \cite{su2014}, \cite{BeGaGo2024}, \cite{Terrence}, \cite{Lam} to name just a few. 

Quite unexpectedly,  the type A electrical Lie group is isomorphic to the symplectic group, \cite{LamPy}.  
This observation is the starting point for a connection to the periplectic Lie superalgebras. 
As was observed by Serganova, based on \cite{women1}, the refined translation functors $\Theta_i$ for the periplectic Lie algebras satisfy the defining relations of a symplectic Lie algebra of infinite rank. 
So far however, any attempt of quantising these relations to encode a grading/filtrations on the category failed.  

The electric KLR algebras we introduced can be graded more or less in only two ways, see \cref{uniquegrad}. 
This allows us to introduce, see also \cref{remVassily}, a quantum version $\el$ of the electric Lie algebra, see \cref{defel}, depending on a sign $\epsilon$ reflecting the two possible choices of gradings. 

We study basic properties of this quantum version.  
The existence of an  obvious filtration, \cref{grel}, indicates that also categorified passing to an associated graded, the ordinary KLR algebras and the electric KLR algebras become isomorphic after forgetting the grading. This will be crucial in the proof of \cref{categorification}. 

The interesting new feature is however this unusual grading and the slightly deformed Serre relations reminiscent of quantum symmetric pairs. 
Indeed, we construct a \emph{quantum electric Hopf algebra} $\Uqg$  in \cref{DefelectricHopf} and show the following:  

\begin{mainthm}[see \cref{elcoideal}]\label{thmC}
	The quantum electric algebra $\el$ is a coideal subalgebra of the quantum electric Hopf algebra $\Uqg$. 
\end{mainthm}

\subsection*{Fock spaces}
We then define an analogue of the natural representation $V$ and its dual  $V^\circledast$ for $\Uqg$ and its restriction to $\el$ and introduce Fock spaces $\Fock$ and dual Fock spaces $\dFock$. These  constructions are surprisingly involved. It is straight-forward to define exterior products of $V$,  and $V^\circledast$ for $\Uqg$ and $\el$, but these exterior products have not a well-defined limit to semi-infinite wedge compatible with the action.  
By working with a mix  of two different comultiplications we finally obtain:  

\begin{mainthm}[see \cref{Fockaselrep}]
	There is an electric  Fock space $\Fock$, i.e.~the space of semiinfinite wedges  can be equipped with an action of $\el$.
\end{mainthm}
The generators of $\el$ act (up to powers of $q$) via the usual combinatorics of adding and removing boxes of partitions. Note however that the generators of the coideal subalgebra in \Cref{thmC}  are sums of creation and annihilation operators.

\subsection*{Categorification result}
We finally prove several categorification results. In analogy to the ordinary KLR situation we show for the cyclotomic quotients given by a charge vector $\chargevec$ as in \cref{defgeneric} the following result: 
\begin{mainthm}[see \cref{higherleftactioncat}]
	The categories $\sR^\level(\chargevec)$ of projective modules over the cyclotomic quotient of level $\level$ categorify the level $\level$ Fock space $\Fockzero_{\chargevec,\level}$. 
	The action of the electric algebra is given by an action of the electric KLR category.
\end{mainthm} 
By passing to right modules we obtain a categorification of the dual Fock space $\dFockzero_{\chargevec,\level}$, see \cref{higherrightactioncat} together with a pairing between them. We categorify several involutions, including a bar involution, which might be interesting on their own. 

\subsubsection*{Acknowledgments} 
We thank Vera Serganova for guiding us through the world of periplectic Lie superalgebras and Jon Brundan for discussions in Bonn. C.S. is deeply grateful to  Inna Entova-Aizenbud for many hours of calculations and exchange of ideas as follow-up to \cite{women2}. Although they did not result in a paper, the project here might not have been started without them.  We thank Azat Gainutdinov for sharing pieces of related work in progress.

The results of this article are part of the first author's PhD thesis. The authors are supported by the Max-Planck Institute for Mathematics (IMPRS Moduli Spaces) and the Gottfried Wilhelm Leibniz Prize of the German Research Foundation.

\subsection*{Some preliminaries and notation}
We denote by $\mathfrak{S}_n$ the symmetric group of order $n!$, generated by the simple transpositions $s_1\coloneqq(1,2), \ldots s_{n-1}\coloneqq (n-1,n)$. 
\begin{nota}\label{defR}
	We denote by $\bbR$ a subset of a commutative ring with unit $1$ (there is no harm to take for $\bbR$ the real numbers) such that for any $r\in\bbR$, $r+m1\in\bbR$ for any $m\in\mathbb{Z}$. For $i,j\in\bbR$ we say $i-j\in\bbZ$ if $i-j\in\bbZ1$.
\end{nota}
\begin{defi}\label{defsubseq}
	A \emph{standard subsequence} of $\bm{i}=(i_1,\ldots,i_m)\in\bbR^m$ is some $\bm{j}=(i_j,i_{j+1},\ldots, i_{j+n-1})\in\bbR^n$ obtained from $\bm{i}$ by taking a connected sequence of entries.
	By an \emph{admissible permutation} of $\bm{i}$ we mean a permutation of the entries which involves only simple transpositions that swap entries $a,b$ with $a\not=b\pm1$.
	By a \emph{subsequence} of $\bm{i}$ we mean any standard subsequence of an admissible permutation of $\bm{i}$.
	Moreover, $\bm{i}$ is \emph{braid avoiding} if $(a,a\pm1,a)$ is not a subsequence of $\bm{i}$.
\end{defi}
\section{Combinatorics of multi-up-down-tableaux}\label{tabacomb}
\subsection{Partitions and residues}\label{prelim}
Throughout this article we fix a \emph{charge} $\charge\in\bbR$. 
A \emph{partition} $\lambda$ is a sequence $\lambda_1\geq\lambda_2\geq\cdots\geq\cdots$ of weakly decreasing non-negative integers.
The \emph{length} of $\lambda$ is the maximal $\ell$ such that $\lambda_\ell> 0$.
We call $\abs{\lambda}\coloneqq\sum_{i=1}^\ell\lambda_i$ the \emph{size} of $\lambda$. 
We will not distinguish between $\lambda$ and the finite sequence $\lambda_1\geq\lambda_2\geq\dots\geq \lambda_\ell$.
We also identify $\lambda$ with its \emph{Young diagram} built from $\abs{\lambda}$ boxes with $\lambda_i$ boxes (left-adjusted) in row $i$.

For every box $\bbox=(r,c)$ in the Young diagram of $\lambda$, specified by its row $r$ and its column $c$, we define its \emph{charged content} as $\cont(\bbox)\coloneqq \charge+r-c$.

We denote by $\Add(\lambda)$ and by $\Rem(\lambda)$ the set of boxes of $\lambda$ that can be added to respectively removed from $\lambda$
such that the result is again a Young diagram.
These sets refine to the union of the sets $\Add_i(\lambda)\coloneqq\{\bbox\in\Add(\lambda)\mid \cont(\bbox)=i\}$ respectively $\Rem_i(\lambda)\coloneqq\{\bbox\in\Rem(\lambda)\mid \cont(\bbox)=i\}$ with $i$ running through $\charge+\mathbb{Z}$.

If $\mu$ can be obtained from $\lambda$ by adding a box we write $\lambda\rightarrow\mu$ or $\lambda\xrightarrow{\bbox}\mu$ encoding additionally the box $\bbox$ which was added.
We also write in this case $\lambda\xrightarrow{-\bbox}\mu$, i.e.~$\mu$ is obtained by removing $\bbox$ from $\lambda$.
We moreover use the notation $\lambda\oplus\bbox$ for $\mu$ and $\mu\ominus\bbox$ for $\lambda$.
The abbreviation $\lambda\leftrightarrow\mu$ means $\lambda\rightarrow\mu$ or $\mu\rightarrow\lambda$.

Next, we extend the notion of charged contents to treat box addition and box removal $\lambda\xrightarrow{\pmbbox}\mu$ for $\pmbbox=\pm\bbox=\pm(r,c)$ in parallel.
We introduce two different extensions, the \emph{residue} $\res(\pmbbox)$ and the \emph{dual residue} $\res^\circledast(\pmbbox)$ of $\pmbbox$ as follows: 
\begin{equation}\label{defres}
	\begin{aligned}
		\res(\pmbbox)=\res(\lambda\xrightarrow{\pmbbox}\mu) &\coloneqq \begin{cases}
			\delta+c-r&\text{if $\pmbbox=(r,c)$,}\\
			\delta+c-r+1&\text{if $\pmbbox=-(r,c)$,}
		\end{cases}\\
		\res^\circledast(\pmbbox)=\res^\circledast(\lambda\xrightarrow{\pmbbox}\mu)&\coloneqq \begin{cases}
			\delta+c-r&\text{if $\pmbbox=(r,c)$,}\\
			\delta+c-r-1&\text{if $\pmbbox=-(r,c)$.}
		\end{cases}
	\end{aligned}
\end{equation}	
%	\begin{rem}
% We mimic here the usual terminology from maps in the situation $\lambda\xrightarrow{\pmbbox}\mu$ by considering the box $\bbox=(r,s)$ in relation to the "source" partition $\lambda$ or in relation to the "target" partition $\mu$.
%	\end{rem}
An \emph{up-down-tableau} of length $k$ is a sequence $(\ta{t}_0, \ta{t}_1, \dots, \ta{t}_k)$ of partitions such that $|\ta{t}_0|=0$ and $\ta{t}_i\leftrightarrow\ta{t}_{i+1}$.
The \emph{shape $\mathrm{Shape}(\ta{t})$ of $\ta{t}$} is $\ta{t}_k$.
To each up-down-tableau we can associate the two residue sequences $\bm{i}\coloneqq\bm{i}_\ta{t}\coloneqq(\res(\pmbbox_1), \dots, \res(\pmbbox_k))$ and $\bm{i}^\circledast\coloneqq\bm{i}^\circledast_\ta{t}\coloneqq(\res^\circledast(\pmbbox_1), \dots, \res^\circledast(\pmbbox_k))$, where $\ta{t}_i\xrightarrow{\pmbbox_i}\ta{t}_{i+1}$.
If $\ta{t}_{i+1}=\ta{t}_i\oplus\bbox_i$, then $\pmbbox_i=\bbox_i$ and $\res(\pmbbox_i)=\cont(\bbox_i)=\res^\circledast(\pmbbox_i)$.
Thus, we recover the charged contents.
\subsection{Combinatorics of multipartitions}\label{multipar}

We now consider multi-partitions and multi-up-down-tableaux.
These are straightforward generalizations obtained by replacing every partition by a tuple of partitions.
Namely, an \emph{$\level$-multi-partition} is an $\level$-tuple $\lambda=(\lambda^1, \dots,\lambda^\level)$ of partitions $\lambda^i$.
We identify $\lambda$ with the corresponding tuple of Young diagrams and call $\abs{\lambda}=\sum_{i=1}^n\abs{\lambda^i}$ the \emph{size} of $\lambda$.
The set of all $\level$-multi-partitions is denoted $\Par^\level$.
We identify $\Par^1$ with the set $\Par$ of partitions.

Every box $\bbox=(r,c,k)$ in the Young diagram of $\lambda\in\Par^\level$ has now a third coordinate $k$ that indexes the component $\lambda^k$, $1\leq k\leq \level$ containing $\bbox$.

To distinguish the components of a multi-partition, we use a \emph{charge sequence} $\chargevec(\infty)\in\bbR^{\mathbb{N}}$.
It determines a charge vector $\chargevec\coloneqq\chargevec(\level)\coloneqq(\delta_1, \dots, \delta_\level)\in\bbR^\level$ for any fixed $\level\in\mathbb{N}$.
For $\lambda\in\Par^\level$ we define the \emph{charged content} of a box $\bbox=(r,c,k)$ as $\cont(\bbox)\coloneqq\delta_k+r-c$.
We denote by $\Add(\lambda)$ and $\Rem(\lambda)$ the set of boxes that can be added to respectively removed from $\lambda$.
As for partitions, these sets are the union of the sets $\Add_i(\lambda)$ (and of $\Rem_i(\lambda)$) of addable (respectively removable) boxes of charged content $i\in\bbR$. 
\begin{rem}\label{meaninggeneric}
	Note that if $\delta_i-\delta_j\notin\bbZ1\subseteq\bbR$ for all $i\neq j$, then the charged content of a box $(r,c,k)$ in $\lambda\in\Par^\level$ uniquely determines this component.
\end{rem}

We again use the arrow notation $\lambda\xrightarrow{\pmbbox}\mu$ if $\mu$ can be obtained from $\lambda$ by adding or removing a box $\bbox$.
If $\pmbbox=\bbox=(r,c,k)$, we have $\lambda^k\xrightarrow{(r,c)}\mu^k$ and $\lambda^i=\mu^i$ for $i\neq k$.
We also extend the notion of (dual) residues involving boxes $\bbox=(r,c,k)$ in $\lambda\in\Par^\level$:
\begin{equation}\label{deflres}
	\begin{aligned}
		\res(\pmbbox)&\coloneqq \begin{cases}
			\delta_k + c-r&\text{if $\pmbbox=(r,c,k)$,}\\
			\delta_k + c-r+1&\text{if $\pmbbox=-(r,c,k)$,}
		\end{cases}\\
		\res^\circledast(\pmbbox)&\coloneqq \begin{cases}
			\delta_k + c-r&\text{if $\pmbbox=(r,c,k)$,}\\
			\delta_k + c-r-1&\text{if $\pmbbox=-(r,c,k)$.}
		\end{cases}
	\end{aligned}
\end{equation}
Obviously $\level=1$, $\delta_1=\delta$ recovers the case \cref{defres} of partitions.
\begin{nota}
	An \emph{$\level$-multi-up-down-tableau} $\ta{t}$ of length $m$ is a sequence $(\ta{t}_0, \ta{t}_1, \dots, \ta{t}_m)$ of $\level$-multi-partitions $\ta{t}_i$ such that $\ta{t}_0$ has size $|\ta{t}_0|=0$ and $\ta{t}_i\leftrightarrow\ta{t}_{i+1}$.
	We call $\ta{t}_m$ the \emph{shape $\shape(\ta{t})$ of $\ta{t}$}.
	By $\restr{\ta{t}}{n}$ for $n<m$ we denote the $\level$-multi-up-down-tableau $(\ta{t}_0, \ta{t}_1, \dots, \ta{t}_n)$ of length $n$ which is the restriction to the first $n+1$ multi-partitions.
\end{nota}
We can draw an $\level$-multi-up-down-tableau by drawing the tuple of Young diagrams of the partitions and arrows between consecutive $\level$-multi-partitions.
Observe, that any $\level$-multi-up-down-tableau $\ta{t}$ necessarily has $\ta{t}_0=(\emptyset,\dots,\emptyset)$.

As above, we associate to $\ta{t}$ two residue sequences $\bm{i}=(\res(\pmbbox_1), \dots, \res(\pmbbox_m))$ and $\bm{i}^\circledast=(\res^\circledast(\pmbbox_1), \dots, \res^\circledast(\pmbbox_m))$ if $\ta{t}_i\xrightarrow{\pmbbox_i}\ta{t}_{i+1}$.
\begin{nota}
	Denote by $\Tudn_m(\lambda)$ the set of all $\level$-multi-up-down-tableaux of shape $\lambda$ and length $m$, and by $\Tudn$ (resp.~$\Tudn_m$, $\Tudn(\lambda)$) the set of $\level$-multi-up-down-tableaux (of fixed length $m$ and of fixed shape $\lambda$).
	For each $\level$-multi-partition $\lambda$ there exists the \emph{canonical up-down-tableaux $\ta{t}^\lambda$ of shape $\lambda$} which is obtained by first adding the boxes for $\lambda^\level$ row by row, then the boxes of $\lambda^{\level-1}$ row by row, and so on.
\end{nota}
\begin{ex}Here is an example of $\la\in\Par^2$ its $\ta{t}^\la$ and the charged contents:
	\small 
	
	\ytableausetup{smalltableaux, boxframe=normal, boxsize=0.4cm, aligntableaux = center}
$\la=\left(\ydiagram{3,2}\;,\;\ydiagram{2,2}\right)\quad
	\ta{t}^\la=\left(\begin{ytableau}
		5&6&7\\
		8&9
	\end{ytableau}
	\;,\;
	\begin{ytableau}
		1&2\\3&4
	\end{ytableau}
	\right)\quad\quad
	\ytableausetup{smalltableaux, boxframe=normal, boxsize=0.6cm,aligntableaux = center}
	\left(\begin{ytableau}
		\scriptstyle{\charge_1} & {\scriptstyle{\charge_1} +1} &\scriptstyle{\charge_1+2} \\
		\scriptstyle{\charge_1-1} & \scriptstyle{\charge_1} \\
	\end{ytableau}\;,\;
	\begin{ytableau}
		\scriptstyle{\charge_2} &\scriptstyle{\charge_2+1} \\
		\scriptstyle{\charge_2-1} & \scriptstyle{\charge_2} 
	\end{ytableau}\right).$
\end{ex}

\begin{defi}\label{parorder}
	Define a partial ordering on $\Par^\level$ by setting $\lambda>\mu$ if $\abs{\lambda}<\abs{\mu}$.
\end{defi}	
\begin{nota} \label{defgeneric}
	For the remainder of the article we fix a charge sequence $\chargevec^\infty\in\bbR^ \mathbb{N}$ (with charge vectors $\chargevec (\level)$) which is \emph{generic} that is 
$\delta_i-\delta_j\notin\bbZ1\subseteq\bbR$ for all $i\neq j$.
\end{nota}	

\section{The quantum \texorpdfstring{$\epsilon$lectrical}{electrical} algebras and their Fock spaces}\label{quantumelectrical}
We fix now on as ground field $\ground$, the field of rational functions over $\mathbb{Q}$ in a variable $q$ with its $\mathbb{Q}$-algebra involution $\overline{\phantom{q}}$ given by $q\mapsto \overline{q}\coloneqq q^{-1}$.
The quantum integer $[m]$ for $m\in \mathbb{Z}$ is the polynomial 
$[m]=\frac{q^m-q^{-m}}{q-q^{1}}=q^{m-1}+q^{m-3}\cdots +q^{1-m}\in\ground$.
\subsection{The quantum $\epsilon$lectrical algebras \texorpdfstring{$\el$}{}}
We next define a main player, the quantum $\epsilon$lectrical (or short $q$-$\epsilon$lectrical) algebras.
\begin{defi}\label{defel}
	Let $\epsilon\in\{\pm1\}$.
	We define the corresponding \emph{$q$-$\epsilon$lectrical algebra} $\el$ to be the algebra generated by $\cE_i$, for $i\in\bbZ$, subject to the relations
	\begin{gather}
		\cE_i\cE_j=q^{b_{ij}}\cE_j\cE_i\quad\text{if }\abs{i-j}>1\tag{$\epsilon l$-1},\label{electricijdistant}\\
		q^3\cE_i^2\cE_{i+1}-[2]\cE_i\cE_{i+1}\cE_i+q^{-3}\cE_{i+1}\cE_i^2 = -q^{\epsilon}[2]\cE_i\tag{$\epsilon l$-2},\label{electricipone}\\
		q^{-3}\cE_i^2\cE_{i-1}-[2]\cE_i\cE_{i-1}\cE_i+q^3\cE_{i-1}\cE_i^2 = -q^{\epsilon}[2]\cE_i\tag{$\epsilon l$-3},\label{electricimone}
	\end{gather}
	where $b_{ij}=\begin{cases}
		-2 &\text{if }j=i,i+1,\\
		4\cdot\sgn(j-i)(-1)^{j-i}&\text{otherwise.}
	\end{cases}$
\end{defi}
\begin{rem}\label{shiftinv}
	The $b_{i,j}$ are \emph{shift invariant}, i.e.~$b_{i,j}=b_{i+1,j+1}$ and also $b_{i-1,j}=b_{j,i}$.
	Moreover, we have $b_{j,i}=-b_{i,j}$ if $|i-j|>1$.
\end{rem}
\begin{rem}\label{remVassily}
The $q$-$\epsilon$lectrical algebras $\el$ should be seen as an analogue of a special example of an electric Lie algebra as defined e.g.~in \cite{BeGaGo2024}.
	In informal discussions with Azat Gainutdinov and Vassily Gorbounov we were informed that they are also working on quantum versions.
	Our example should arise as a special example of their construction. 
\end{rem}
Next, we define the bar involution and shift isomorphism connecting $\el$ and $\elinv$.
\begin{lem}[Bar involution]\label{barinvdef}
	There exists a unique $q$-antilinear isomorphism \[\overline{\phantom{U}}\colon \el\to \elinv,\quad \overline{\cE_i}=\cE_i\]
	of $\mathbb{Q}$-algebras.
	Here, $q$-antilinear means $\overline{f\cE}=\overline{f}\cdot\overline{\cE}$ for any $f\in\ground$, $\cE\in\el$.
\end{lem}
\begin{proof}
	Clearly, it suffices to show that the assignments are compatible with the defining relations of $\el$, since then the statements follow from the definitions.
	
	For $1\leq i,j\leq n$ with $|i-j|>1$ (such that the expressions make sense) we have 
	\begin{gather*}
		\overline{\cE_i}\overline{\cE_j} = \cE_i\cE_j= q^{-b_{ij}}\cE_{j}\cE_{i} = q^{-b_{ij}}\overline{\cE_j}\overline{\cE_i} = \overline{q^{b_{ij}}\cE_j\cE_i},\\
		\begin{aligned}[t]
			\overline{q^{3}\cE_i^2\cE_{i+1}-[2]\cE_i\cE_{i+1}\cE_i + q^{-3}\cE_{i+1}\cE_i^2} 
			= 	&q^{-3}\cE_{i}^2\cE_{i+1} -[2]\cE_{i}\cE_{i+1}\cE_{i} + q^{3}\cE_{i+1}\cE_{i}^2\\
			=	&-q^{-\epsilon}[2]\cE_{i}=\overline{-q^{\epsilon}[2]\cE_i},
		\end{aligned}\\
		\begin{aligned}[t]
			\overline{q^{-3}\cE_i^2\cE_{i-1}-[2]\cE_i\cE_{i-1}\cE_i + q^{3}\cE_{i-1}\cE_i^2} 
			= 	&q^{3}\cE_{i}^2\cE_{i-1} -[2]\cE_{i}\cE_{i-1}\cE_{i} + q^{-3}\cE_{i-1}\cE_{i}^2\\
			=	&-q^{-\epsilon}[2]\cE_{i}=\overline{-q^{\epsilon}[2]\cE_i}.
		\end{aligned}
	\end{gather*}
	Thus, we obtain a well-defined antilinear algebra homomorphism $\el\to \elinv$.
\end{proof}
\begin{lem}[Shift isomorphism]\label{sigmadef}
	There exists a unique $\ground$-algebra anti-isomorphism \[\sigma\colon \elinv\to \el,\quad \sigma(\cE_{i})=q^{-\epsilon}\cE_{i+1}.\]
\end{lem}
\begin{proof} Since the $\cE_{i}$ are algebra generators of $\elinv$, there is at most one such anti-homomorphism which is then also an isomorphism, since $\sigma'\colon \el\to \elinv$, $\cE_{i}\mapsto q^{\epsilon}\cE_{i-1}$ provides an inverse to $\sigma$.
	We however need to verify well-definedness, that is the compatibility with the defining relations of $\el$ and $\elinv$. 
	
	To see \cref{electricijdistant} we calculate for $\abs{i-j}>1$ using \cref{shiftinv},
	\begin{equation*}
		\sigma(\cE_{i}\cE_{j}) = q^{-2\epsilon}\cE_{j+1}\cE_{i+1} = q^{-2\epsilon+b_{j+1,i+1}}\cE_{i+1}\cE_{j+1}=\sigma(q^{-b_{i,j}}\cE_{j}\cE_{i})=q^{b_{i,j}}\sigma(\cE_{j}\cE_{i}).	
	\end{equation*}

	To see \cref{electricipone} let $j=i+1$ and calculate
	\begin{align*}
		\sigma(q^{-3}\cE_i^2\cE_{j}-[2]\cE_i\cE_{j}\cE_i+q^{3}\cE_{j}\cE_i^2) 
		= &\;q^{-3\epsilon}(q^{-3}\cE_{j+1}\cE_{i+1}^2-[2]\cE_{i+1}\cE_{j+1}\cE_{i+1}+q^{3}\cE_{i+1}^2\cE_{j+1})\\ 
		= & -q^{-2\epsilon}[2]\cE_{i+1}= \sigma(-q^{-\epsilon}[2]\cE_i).
	\end{align*}
	To see \cref{electricimone} let $j=i-1$ and calculate
	\begin{align*}
		\sigma(q^{3}\cE_i^2\cE_{j}-[2]\cE_i\cE_{j}\cE_i+q^{-3}\cE_{j}\cE_i^2) 
		=&\; q^{-3\epsilon}(q^{3}\cE_{j+1}\cE_{i+1}^2-[2]\cE_{i+1}\cE_{j+1}\cE_{i+1}+q^{-3}\cE_{i+1}^2\cE_{j+1})\\ 
		= & -q^{-2\epsilon}[2]\cE_{i+1}= \sigma(-q^{-\epsilon}[2]\cE_i).
	\end{align*}
	Therefore, the assignments give a well-defined $q$-linear anti-isomorphism $\sigma$.
\end{proof}
\begin{lem}\label{taudef}
	There exists a unique isomorphism of $\mathbb{Q}(q)$-algebras \[\tau\colon \el\to \elinv,\quad \tau(\cE_i)=\cE_{-i}.\]
\end{lem}
\begin{proof}
	This is obvious from the definition. 
\end{proof}
\subsection{The associated graded of the electric algebras}

The algebra $\el$ becomes filtered by putting the monomials $\cE_{i_1}\cdots \cE_{i_k}$ in degree $k$.
We directly obtain: 
\begin{lem}\label{grel}	
	The associated graded algebra $\gr \el$ of $\el$ is the algebra with generators $\cE_i$, $i\in\bbZ$, and relations
	\begin{align}
		\begin{gathered}
			\cE_i\cE_j=q^{b_{ij}}\cE_j\cE_i\quad\text{if }\abs{i-j}>1,\\
			q^3\cE_i^2\cE_{i+1}-[2]\cE_i\cE_{i+1}\cE_i+q^{-3}\cE_{i+1}\cE_i^2 = 0,\\
			q^{-3}\cE_i^2\cE_{i-1}-[2]\cE_i\cE_{i-1}\cE_i+q^3\cE_{i-1}\cE_i^2 = 0.
		\end{gathered}
	\end{align}
\end{lem}
We obtain a basis for $\el$ (resembling a quantum group basis at $q=0$ from \cite{Reineke}):
\begin{cor}
	The subalgebra of $\el$ generated by $\cE_i$, $1\leq i-a\leq n-1$, has basis 
	\[\cE_{a+1}^{m_1}(\cE_{a+2}\cE_{a+1})^{m_2}\cE_{a+2}^{m_3}(\cE_{a+3}\cE_{a+2}\cE_{a+1})^{m_4}(\cE_{a+3}\cE_{a+2})^{m_5}\cE_{a+3}^{m_6}\cdots \cE_{a+n-1}^{m_N}.\] 
	Here, $a\in\bbZ$, $n\in\bbN$ are fixed arbitrarily and $m_i\in\bbN_0$, $N=\frac{(n-1)n}{2}$.
\end{cor}
\begin{proof}For $a=0$, the corresponding polynomials (in the usual generators $E_i$) form a basis of its positive part of the quantum group $U_q(\mathfrak{sl}_n)$, see \cite[Theorem~2]{Ringel}, \cite[8.21]{Jantzenquantum}.
	The result then follows from the definitions and \cref{grel}.
\end{proof}

\subsection{The quantum electric Hopf algebra \texorpdfstring{$\Uqg$}{}}
The goal of this section is to realise the $q$-$\epsilon$lectric algebras as coideal subalgebras of some Hopf algebra which is reminiscent of a quantised universal enveloping algebra.
We define this Hopf algebra using a quantum double construction from a pairing between two Hopf algebras $\Borelp$ and $\Borelp$.
We start by defining the ingredients of the construction.
\begin{defi}
	Consider the free $\bbZ$-module $\lie{h}_\bbZ$ with basis $e_i$, $i\in\bbZ$ and, via pointwise addition, the $\bbZ$-module $X\coloneqq\Hom_{\bbZ}(\lie{h}_\bbZ,\bbZ)$.
	Let $\langle\_,\_\rangle\colon X\otimes_{\bbZ}\lie{h}_\bbZ\to\bbZ$ be the evaluation.
	We write $\alpha_i^\vee\in\lie{h}_\bbZ$ for the element $e_{i+1}-e_i$ and denote by $\eps_i$ the dual element to $e_i$ and set $\alpha_i\coloneqq\eps_{i+1}-\eps_i$.
	In particular, the $\alpha_i$ (and $\alpha^\vee_i$) form the (dual) roots of a root system of type $A_\infty$.
	Furthermore, let $\beta_i\in X$ be defined by 
	\begin{equation}\label{Defbetas}
		\langle \beta_i, e_j\rangle = \begin{cases}
			(-1)^i2 &\text{if }j=i,\\
			(-1)^{i}4&\text{if }(-1)^jj>(-1)^ji,\\
			0&\text{otherwise.}
		\end{cases}
	\end{equation}
	Finally, we define for $i\in\bbZ$, $\gamma_i\in X$ by $\langle \gamma_i, e_j\rangle = -\langle \beta_{i+1}, e_j\rangle$, that is $\gamma_{i}=-\beta_{i+1}$.
\end{defi}\label{deffsupp}
\begin{nota} We denote by $X^\mathrm{fsupp}\subseteq X$ the set of all $\lambda\in X$ with \emph{finite support}, i.e.~$\langle \lambda, e_j\rangle\not= 0$ for only finitely many $j$.
	(Note that $\beta_i,\gamma_i\notin X^\mathrm{fsupp}$, $\alpha_i\in X^\mathrm{fsupp}$.) 
\end{nota}	
\begin{lem}\label{bij}
	For $i,j\in\bbZ$ we have $\langle \beta_i, \alpha_j^\vee\rangle = b_{ji}$, $\langle \gamma_i, \alpha_j^\vee\rangle = -b_{ij}$ and $b_{j,i+1}=b_{i,j}$.
\end{lem}
\begin{proof}This follows by plugging in the definitions.
\end{proof}

\begin{defi}
	Define the algebra $\Borelm$ as the $\bbQ(q)$-algebra generated by $F_i$ for $i\in\bbZ$ and by $K_\lambda$ for $\lambda\in X$, subject to the relations
	
	\begin{minipage}{0.38\textwidth}
		\begin{gather}
			K_\lambda K_\mu = K_{\lambda+\mu},\tag{$1^-$}\label{Borelmkk}\\
			K_0 = 1,\tag{$2^-$}\label{Borelminverse}\\
			K_\lambda F_i = q^{-\langle \lambda, \alpha_i^\vee\rangle}F_i K_\lambda,\tag{$3^-$}\label{Borelmfk}
		\end{gather}	
	\end{minipage}
	\begin{minipage}{0.61\textwidth}
		\begin{gather}
			F_iF_j = q^{b_{ij}}F_jF_i\quad\text{if }\abs{i-j}>1,\tag{$4^-$}\label{Borelmdistant}\\
			q^{3}F_i^2F_{i+1}-[2]F_iF_{i+1}F_i+q^{-3}F_{i+1}F_i^2 = 0,\tag{$5^-$}\label{Borelmserre1}\\
			q^{-3}F_i^2F_{i-1}-[2]F_iF_{i-1}F_i+q^{3}F_{i-1}F_i^2 = 0.\tag{$6^-$}\label{Borelmserre2}
		\end{gather}	
	\end{minipage}
\end{defi}
\begin{lem}\label{lbHopfalgebra}
	The following assignments define (anti-)algebra homomorphism 
	\begin{align*}
		\Delta\colon\Borelm&\to\Borelm\otimes\Borelm&\eps\colon\Borelm&\to\bbQ(q)&S\colon\Borelm&\to\Borelm\\
		F_i&\mapsto F_i\otimes K_{\beta_i} + 1\otimes F_i,&F_i&\mapsto 0,&F_i&\mapsto -F_iK_{-\beta_i},\\
		K_\lambda&\mapsto K_\lambda\otimes K_\lambda,&K_\lambda&\mapsto1,&K_\lambda&\mapsto K_{-\lambda},
	\end{align*}
	which endow $\Borelm$ with the structure of a Hopf algebra.
\end{lem}
\begin{proof}The proof is a standard calculation.
	For details see \Cref{prooflbHopf}. 
\end{proof}
In analogy to the universal enveloping algebra of a finite dimensional simple complex Lie algebra we call $\Borelm$ the \emph{negative Borel part}, and define a \emph{positive Borel part} $\Borelp$:
\begin{defi}
	Define the algebra $\Borelp$ as the $\bbQ(q)$-algebra generated by $E_i$ and $K_\lambda$ for $i\in\bbZ$ and $\lambda\in X$ subject to the relations
	
	\begin{minipage}{0.38\textwidth}
		\begin{gather}
			K_\lambda K_\mu = K_{\lambda+\mu},\tag{$1^+$}\label{Borelpkk}\\
			K_0 = 1,\tag{$2^+$}\label{Borelpinverse}\\
			K_\lambda E_i = q^{\langle \lambda, \alpha_i^\vee\rangle}E_i K_\lambda,\tag{$3^+$}\label{Borelpke}
		\end{gather}	
	\end{minipage}
	\begin{minipage}{0.61\textwidth}
		\begin{gather}
			E_iE_j = q^{b_{ij}}E_jE_i\quad\text{if }\abs{i-j}>1,\tag{$4^+$}\label{Borelpdistant}\\
			q^{3}E_i^2E_{i+1}-[2]E_iE_{i+1}E_i+q^{-3}E_{i+1}E_i^2 = 0,\tag{$5^+$}\label{Borelpserre1}\\
			q^{-3}E_i^2E_{i-1}-[2]E_iE_{i-1}E_i+q^{3}E_{i-1}E_i^2 = 0\tag{$6^+$}\label{Borelpserre2}.
		\end{gather}	
	\end{minipage}
\end{defi}
Not very surprisingly, the positive Borel can also be turned into a Hopf algebra:
\begin{lem}\label{ubHopfalgebra}
	The following assignments define (anti-)algebra homomorphism 
	\begin{align*}
		\Delta\colon\Borelp&\to\Borelp\otimes\Borelp&\eps\colon\Borelp&\to\bbQ(q)&S\colon\Borelp&\to\Borelp\\
		E_i&\mapsto K_{\alpha_i}\otimes E_i + E_i\otimes K_{\alpha_i-\gamma_i}&E_i&\mapsto 0&E_i&\mapsto -K_{-\alpha_i}E_iK_{\gamma_i-\alpha_i}\\
		K_{\lambda}&\mapsto K_{\lambda}\otimes K_{\lambda}&K_{\lambda}&\mapsto1&K_{\lambda}&\mapsto K_{-\lambda}
	\end{align*}
	which endow $\Borelp$ with the structure of a Hopf algebra.
\end{lem}
\begin{proof}The proof is totally analogous to the one of \Cref{lbHopfalgebra}. 
\end{proof}
\begin{rem}
	The slight asymmetry between $\Borelm$ and $\Borelp$ is chosen on purpose and motivated by the categorification results obtained later.
	It encodes the extra data, namely the $b_{i,j}$, appearing in the definition of $\el$ via \cref{bij}.
	A rescaling of $E_i$ by $K_{-\alpha_i}$ would indeed provide formulas similar to those for the $F_i$'s.
\end{rem}
Both, in $\Borelm$ and in $\Borelp$, the $K^\lambda$ for $\lambda\in X$ generate a commutative subalgebra $\Cartan$ which is a Hopf subalgebra.
We call these the \emph{Cartan parts} of $\Borelm$ and $\Borelp$.

\begin{rem} \label{twoCartans} The Cartan parts have basis $K_\lambda$, $\lambda\in X$ and multiplication as in ($1^+$).
	For simplicity, we write $\Borelm$ and $\Borelp$ instead of more suggestively $U^{\geq}_q$ and $U^\leq_q $. 
\end{rem}
We now want to construct from  $\Borelm$ and $\Borelp$ a Hopf algebra via the usual Drinfeld double construction, see e.g.~\cite[IX.4]{Kassel}  for the general concept. 

Fix  now a bilinear pairing 
$(\_,\_)\colon X\times X \to\bbZ$ such that for all $i\in\mathbb{Z}$ we have $(\beta_i, \mu) = -\langle \mu, \alpha_i^\vee\rangle$ and $(\lambda, \gamma_i) = \langle \lambda, \alpha_i^\vee\rangle$.
Note for this that $(\beta_i, \gamma_j)=-\langle \gamma_j, \alpha_i^\vee\rangle=b_{ji}$ is consistent with $(\beta_i, \gamma_i)= \langle \beta_i, \alpha_j^\vee\rangle=b_{ji}$ by \cref{bij}.
\begin{prop}\label{Hopfpairing}
	There exists a unique Hopf pairing $$\langle\_,\_\rangle\colon\Borelm\otimes\Borelp\to\bbQ(q)$$ such that for all $i,j\in\mathbb{Z}$ and $\lambda,\mu\in X$ the following holds:
	\begin{equation*}
		\langle K_\lambda, K_\mu \rangle = q^{(\lambda, \mu)}, \quad\langle F_i, K_\mu\rangle = 0,\quad
		\langle F_i, E_j\rangle = \delta_{ij}\frac{1}{q-q^{-1}}, \quad \langle K_\lambda, E_j\rangle = 0.
	\end{equation*}
\end{prop}
\begin{proof} 
	We need to verify that $\langle\_,\_\rangle$ extends uniquely to a pairing which satisfies the Hopf pairing conditions
	\begin{enumerate}
		\item $\langle a, 1\rangle = \eps(a)$ and $\langle 1, b\rangle = \eps(b)$ for all $a\in\Borelm$ and $b\in\Borelp$.
		\item $\langle aa', b\rangle = \langle a\otimes a', \Delta^{op}(b)\rangle$ for all $a,a'\in\Borelm$ and $b\in\Borelp$.
		\item $\langle a, bb'\rangle = \langle \Delta(a), b\otimes b'\rangle$ for all $a\in\Borelm$ and $b,b'\in\Borelp$.
		\item\label{Hopfpairingantipode} $\langle S(a), b\rangle = \langle a, S^{-1}(b)\rangle$ for all $a\in\Borelm$ and $b\in\Borelp$.
	\end{enumerate}
	The proof is analogous to \cite[1.2]{Lusztigbook}, see also \cite[Prop.~2.9.3, Prop~ 2.9.4]{Xiao} for a summary.
	If one uses (in the notation of the latter) the slightly adjusted functionals $\xi_i(K_\lambda \Theta_i^-)=\frac{q^{(\lambda,\alpha)-\langle\lambda,\check{\alpha}_i\rangle}}{q-q^{-1}}$, $T_\mu(K_\lambda)=q^{(\lambda,\mu)}$, the arguments can be copied.
\end{proof}

This Hopf pairing endows $\Borelm\otimes\Borelp$ with the structure of a Hopf algebra, see e.g.~\cite[\S3.2]{Josephbook} for the construction and \cite[Prop.~2.4]{Xiao} for the explicit formulas:	
\begin{cor}\label{corHopf}
	There is a unique Hopf algebra structure on $\Borelm\otimes\Borelp$ such that $\Borelm$ and $\Borelp$ are Hopf subalgebras via the canonical embeddings and 	
	\begin{enumerate}[label=$(\alph*)$]
		\item the multiplication in Sweedler notation is given by 
		\begin{equation} \label{mult} 
			(a\otimes b)(a'\otimes b') = \sum_{(a'),(b)}\langle S^{-1}(a'_{(1)}), b_{(1)}\rangle a a'_{(2)}\otimes b_{(2)} b'\langle a'_{(3)}, b_{(3)}\rangle, 
		\end{equation}
		\item the comultiplication is given by $\Delta(a\otimes b) = \sum_{(a),(b)}a_{(1)}\otimes b_{(1)}\otimes a_{(2)}\otimes b_{(2)}$ with 
		counit $\eps(a\otimes b) = \eps(a)\eps(b)$, and 
		\item the antipode is $S(a\otimes b) = (1\otimes S(b))(S(a)\otimes 1)$.
	\end{enumerate}
\end{cor}
\begin{rem}\label{Drinfdouble}
	In $\Borelm\otimes\Borelp$, the product $(1\otimes E_i)(F_j\otimes 1) $ is by \cref{mult} equal to 
	\begin{align*}
		& \begin{multlined}[t]{\langle S^{-1}(F_j), E_i\rangle K_{\beta_j}\otimes K_{\alpha_i-\beta_j}\langle K_{\beta_j}, K_{\alpha_i-\beta_j}\rangle} + \langle S^{-1}(F_j), K_{\alpha_i}\rangle K_{\beta_j}\otimes E_i\langle K_{\beta_j}, K_{\alpha_i-\beta_j}\rangle \\
			\begin{aligned}[t]
				+& \langle S^{-1}(F_j), K_{\alpha_i}\rangle K_{\beta_j}\otimes K_{\alpha_i}\langle K_{\beta_j}, E_i\rangle + \langle S^{-1}(1), E_i\rangle F_j\otimes K_{\alpha_i-\gamma_i}\langle K_{\beta_j}, K_{\alpha_i-\gamma_i}\rangle\\
				+& {\langle S^{-1}(1), K_{\alpha_i}\rangle F_j\otimes E_i\langle K_{\beta_j}, K_{\alpha_i-\gamma_i}\rangle} + \langle S^{-1}(1), K_{\alpha_i}\rangle F_j\otimes K_{\alpha_i}\langle K_{\beta_j}, E_i\rangle\\
				+& \langle S^{-1}(1), E_i\rangle 1\otimes K_{\alpha_i-\gamma_i}\langle F_j, K_{\alpha_i-\gamma_i}\rangle + \langle S^{-1}(1), K_{\alpha_i}\rangle 1\otimes E_i\langle F_j, K_{\alpha_i-\gamma_i}\rangle\\
				+& {\langle S^{-1}(1), K_{\alpha_i}\rangle 1\otimes K_{\alpha_i}\langle F_j, E_i\rangle}
			\end{aligned}
		\end{multlined}\\
		&= {q^{\langle\beta_j-\alpha_i,\alpha_j^\vee\rangle}F_j\otimes E_i} +{\delta_{ij}\frac{1}{q-q^{-1}} 1\otimes K_{\alpha_i}}+ {q^{\langle\beta_j-\alpha_i, \alpha_j^\vee\rangle}\langle -K_{-\beta_j}F_j, E_i\rangle K_{\beta_j}\otimes K_{\alpha_i-\beta_j}} ,
	\end{align*}
	since the other summands vanish.
	Now the last summand can be simplified using $\langle -K_{-\beta_j}F_j, E_i\rangle=-\langle K_{-\beta_j}, K_{\alpha_i-\beta_j}\rangle\langle F_j, E_i\rangle=
	-\frac{q^{(-\beta_j,\alpha_i-\gamma_i)} }{q-q^{-1}}=
	-\frac{q^{\langle\-\alpha_i-\gamma_i, \alpha_j^\vee\rangle} }{q-q^{-1}}$.
	%			\begin{align*}
	%			&= -q^{\langle\beta_j-\alpha_i, \alpha_j^\vee\rangle}\langle K_{-\beta_j}, K_{\alpha_i-\beta_j}\rangle\langle F_j, E_i\rangle K_{\beta_j}\otimes K_{\alpha_i-\beta_j}+ q^{\langle\beta_j-\alpha_i,\alpha_j^\vee\rangle}F_j\otimes E_i \\
	%			&\quad\quad+ \delta_{ij}\frac{1}{q-q^{-1}}1\otimes K_{\alpha_i}.
	%		\end{align*}
	
	Altogether, $(1\otimes E_i)(F_j\otimes 1) =q^{\langle\beta_j-\alpha_i,\alpha_j^\vee\rangle}F_j\otimes E_i + \delta_{ij}\frac{1\otimes K_{\alpha_i}-K_{\beta_i}\otimes K_{\alpha_i-\beta_j}}{q-q^{-1}}$.
\end{rem}
In analogy to the universal enveloping algebra of a simple complex Lie algebra we like to identify the Cartan parts $\Cartan$, see \cref{twoCartans}, from the two Borel parts.
\begin{prop}\label{Kbalanced}
	The maps $m$, $\Delta$, $\eps$, $S$ defining the Hopf algebra structure on $\Borelm\otimes\Borelp$ are $\Cartan$-balanced.
	Thus, $\Borelm\otimes_\Cartan\Borelp$ inherits a Hopf algebra structure.
\end{prop}
\begin{proof}
	This is proven in \Cref{proofKbalanced}.
\end{proof}
\begin{defi}\label{DefelectricHopf}
	We call $\Uqg\coloneqq\Borelm\otimes_\Cartan\Borelp$ the \emph{quantum electrical Hopf algebra}.
\end{defi}

\begin{nota} 
	From now on, we will write $F_i$ (respectively $E_i$) for the element $F_i\otimes 1$ (respectively $1\otimes E_i$) in $\Uqg$.
	We also write $K_\lambda$ for $K_\lambda\otimes 1=1\otimes K_\lambda$ in $\Uqg$.
\end{nota}
The quantum $\epsilon$lectical Hopf algebra is very similar to the quantised universal enveloping algebra of $\mathfrak{sl}_\mathbb{Z}$ / type $A_\infty$ (of adjoint type in the sense of \cite[4.5]{Jantzenquantum}):

\begin{cor}\label{qcomm}
	The quantum electrical Hopf algebra $\Uqg$ has as algebra a presentation with generators $E_i$, $F_i$ for $i\in\bbZ$ and $K_{\lambda}$ for $\lambda\in X$ subject to the relations
	\begin{gather}
		K_\lambda K_\mu = K_{\lambda+\mu},\qquad K_0 = 1,\tag{U-1}\label{uqgkk}\\
		K_\lambda F_i = q^{-\langle \lambda, \alpha_i^\vee\rangle}F_i K_\lambda,\qquad K_\lambda E_i = q^{\langle \lambda, \alpha_i^\vee\rangle}E_i K_\lambda\label{uqgefk}\tag{U-2}\\
		[E_i, F_j]_{\beta_{ij}} = \delta_{ij}\frac{K_{\alpha_i}-K_{-\alpha_i}}{q-q^{-1}},\label{uqgef}\tag{U-3}\\
		\begin{aligned}[c]
			&F_iF_j = q^{b_{ij}}F_jF_i \text{ if }\abs{i-j}>1,\\
			&q^{3}F_i^2F_{i+1}-[2]F_iF_{i+1}F_i+q^{-3}F_{i+1}F_i^2 = 0,\\
			&q^{-3}F_i^2F_{i-1}-[2]F_iF_{i-1}F_i+q^{3}F_{i-1}F_i^2 = 0,\\
		\end{aligned}\label{uqgfserre}\tag{U-4}\\
		\begin{aligned}[c]
			&E_iE_j = q^{b_{ij}}E_jE_i\quad\text{if }\abs{i-j}>1,\\
			&q^{3}E_i^2E_{i+1}-[2]E_iE_{i+1}E_i+q^{-3}E_{i+1}E_i^2 = 0,\\
			&q^{-3}E_i^2E_{i-1}-[2]E_iE_{i-1}E_i+q^{3}E_{i-1}E_i^2 = 0,
		\end{aligned}\label{uqgeserre}\tag{U-5}
	\end{gather}
	where $[E_i, F_j]_{\beta_{ij}}$ denotes the $q$-commutator $E_iF_j-q^{\beta_{ij}}F_jE_i$ and $\beta_{ij}=\langle \gamma_i-\alpha_i, \alpha_j^\vee\rangle$.
	Spelling out $\beta_{ij}$ explicitly, we have
	\begin{equation}\label{betaij}
		\beta_{ij} = \begin{cases}
			0&\text{if }i=j,\\
			3(j-i)&\text{if }\abs{i-j}=1,\\
			-4\cdot\sgn(j-i)(-1)^{j-i}&\text{otherwise.}
		\end{cases}
	\end{equation}
	The Hopf algebra structure is given by \cref{lbHopfalgebra} and \cref{ubHopfalgebra}.
\end{cor}
\begin{proof} The relations \cref{uqgkk}--\cref{uqgefk}, \cref{uqgfserre}--\cref{uqgeserre} hold by definition of $\Uqg$ and \cref{uqgef} follows from \cref{Drinfdouble} noting that $\alpha_i+\beta_i-\gamma_i=-\alpha_i$ by \cref{bij} and \cref{Defbetas}.
	The given relations suffice by comparision with $U_q(\mathfrak{sl}_\bbZ)$: one obtains a PBW-type basis for $\Uqg$ from compatible PBW-type bases of $\Borelp$ and $\Borelm$ via \cref{DefelectricHopf}.
\end{proof}
\subsection{Realization of the $q$-$\epsilon$lectric algebra as a coideal}
The quantum electrical Hopf algebra allows treating $\el$ in a more conceptual way as coideal in $\Uqg$:
\begin{thm}[Coideal realisation]\label{elcoideal}
	The $q$-$\epsilon$lectric algebra $\el$ embeds into the quantum electrical Hopf algebra $\Uqg$ as a right coideal via $\cE_i\mapsto F_i+q^{\epsilon-1}E_{i-1}K_{-\alpha_{i-1}}$.
\end{thm}
\begin{proof}
	In \Cref{proofcoideal} we show that the assignment provides a well-defined algebra homomorphism $\operatorname{j}$ which is moreover injective.
	It remains to show that its image $C\coloneqq\operatorname{im}(\operatorname{j})$ is in fact a right coideal.
	We have 
	\begin{align*}
		\Delta(\operatorname{j}(\cE_{i+1})) &= \Delta(F_{i+1}+q^{\epsilon-1}E_{i}K_{-\alpha_{i}})\\ 
		&= F_{i+1}\otimes K_{\beta_{i+1}} + 1\otimes F_{i+1} + q^{\epsilon-1}\otimes E_{i}K_{-\alpha_i} + q^{\epsilon-1}E_{i}K_{-\alpha_i}\otimes K_{-\gamma_i}.
	\end{align*}
	Since $\beta_{i+1}=-\gamma_i$, we obtain 
	\begin{equation*}
		\Delta(\operatorname{j}(\cE_{i+1})) = 	\operatorname{j}(1)\otimes 	\operatorname{j}(\cE_{i+1}) + 	\operatorname{j}(\cE_{i+1})\otimes K_{\beta_{i+1}}\in C\otimes \Uqg.
	\end{equation*}
	This shows that $C$ is a right coideal in $\Uqg$ and finishes the proof.
\end{proof}
\begin{nota}
	From now on we identify $\el$ with its image in $\Uqg$ and thus view it as coideal subalgebra of $\Uqg$ with $\Delta(\cE_{i}) = 1\otimes \cE_{i} + \cE_{i}\otimes K_{\beta_{i}}$.	
\end{nota}

\subsection{(Dual) Natural representation of \texorpdfstring{$\Uqg$}{gl} and their exterior powers}
The Hopf algebra $\Uqg$ is another quantization of the universal enveloping algebra of $\lie{sl}_\bbZ$.
In analogy to $\lie{sl}_\bbZ$ we define a natural representation $V=\bbQ(q)^\bbZ$ of $\Uqg$.
\begin{prop}\label{naturalrepuqg}
	Let $V = \bbQ(q)^\bbZ$ with basis $v_i$, $i\in\bbZ$.
	Then there exists a well-defined right action of $\Uqg$ on $V$ given, for $i,j\in\bbZ$ and $\lambda\in X$, by
	\begin{equation*}
		v_j F_i = \delta_{ij}v_{j+1},\quad
		v_j E_i = \delta_{i+1,j}v_{j-1},\quad
		v_j K_\lambda = q^{\langle\lambda, e_j\rangle} v_j.
	\end{equation*}
\end{prop}
\begin{proof}
	The relation \cref{uqgkk} is immediate, and \cref{uqgefk} follows directly from the definition of $\alpha_j^\vee$.
	For \cref{uqgfserre,uqgeserre} both sides act by $0$, hence these relations are satisfied.
	It remains to check the compatibility with \cref{uqgef}.
	If $i\neq j$, then both sides act by~$0$.
	Otherwise, we have $v_k[E_i, F_i]_{\beta_{ii}} = v_kE_iF_i-v_kF_iE_i = \delta_{k,i+1}v_{i+1} - \delta_{ki}v_i$, which equals 
	$v_k\frac{K_{\alpha_i}-K_{-\alpha_i}}{q-q^{-1}} = v_k\frac{q^{\langle\alpha_i, e_k\rangle}-q^{-\langle\alpha_i, e_k\rangle}}{q-q^{-1}}$, 
	because $\alpha_i=\eps_{i+1}-\eps_i$.
\end{proof}
\begin{defi}
	Let $V^\circledast$ be the (restricted) dual vector space of $V$, i.e.~the vector space with basis $v^{i}\in\Hom_\Bbbk(V, \Bbbk)$, where $v^{i}(v_j)\coloneqq\delta_{ij}$.
	The right $\Uqg$-module structure on $V$ defines a left $\Uqg$-module structure on $V^\circledast$.
	In formulas, it is given by 
	\begin{equation*}
		F_iv^j = \delta_{i+1,j}v^{i},\quad E_iv^j = \delta_{ij}v^{i+1},\quad v^j K_\lambda = q^{\langle\lambda, e_j\rangle} v^j.
	\end{equation*}
	%Later:: Via restriction we obtain left $\el$-modules and by twisting with $\sigma$ we obtain right $\elinv$-modules $V^\circledast$ and $W^{\circledast}$.
\end{defi}
Define a $q$-bilinear pairing $(\_,\_)\colon V^\circledast\otimes V\to\bbQ(q)$ by $(v^i, v_j)=\delta_{ij}$.
\begin{lem}\label{lemmaadjointeasy}
	The bilinear pairing satisfies $(wu,v) = (w,v\sigma(u))$ for all $w\in  V^\circledast$, $v\in V$ and $u\in\elinv$ with $\sigma $ as in \cref{sigmadef}.
\end{lem}
\begin{proof} It is enough to consider $v=v^l, w=v_k, u=\cE_i$ for any $l,k,i$. 
	We compute $(v^l\cE_i,v_k)=(\delta_{(i+2)l}q^{-\epsilon}v^{i+1}+\delta_{il}v^{i+1} ,v_k)=\delta_{(i+1)k}(q^{-\epsilon}\delta_{(i+2)l}+\delta_{il})$. 
	On the other hand $(v^l,v_k\sigma(\cE_i))=q^{-\epsilon}(v^l,v_k\cE_{i+1})=q^{-\epsilon}(\delta_{k(i+1)}(v^l,v_{i+2})+q^{\epsilon}\delta_{k(i+1)}(v^l,v_i))$. Since the latter equals 
$q^{-\epsilon}\delta_{k(i+1)}\delta_{l(i+2)}+\delta_{k(i+1)}\delta_{li}$ the assertion follows.
\end{proof}	
We next define an alternative comultiplication on $\Uqg$:
\begin{defi}\label{newotimes}
	Given $\lambda\in X$ let $\lambda'\in X$ such that $\langle\lambda', e_j\rangle=\langle\lambda, e_{j+1}\rangle$.
	Define the algebra isomorphism 
	\begin{equation*}
		\mathrm{shift}\colon\Uqg\to\Uqg, \quad F_i\mapsto F_{i+1}, E_i\mapsto E_{i+1}, K_\lambda\mapsto K_{\lambda'}.
	\end{equation*}
	Let $\Delta'\coloneqq(\mathrm{shift}\otimes\mathrm{shift})\mathbin{\Delta}\mathrm{shift}^{-1}$ be the induced comultiplication, cf. \cite[7.2]{Jantzenquantum}. 
\end{defi}
\begin{rem}
	We have $\Delta'=(\mathrm{shift}^{-1}\otimes\mathrm{shift}^{-1})\mathbin{\Delta}\mathrm{shift}$ as $\langle\beta_{i-1},e_j\rangle=\langle\beta_{i+1},e_{j+1}\rangle$.
\end{rem}

\begin{nota}		
	\cref{newotimes} defines a second monoidal structure on the category of $\Uqg$-modules, cf. \cite[3.8]{Jantzenquantum}.
	To keep track of the tensor products we use the symbol $\otimes$ for the usual tensor product of vector spaces and $\odot _1$ and $\odot _2$ for the tensor product of $\Uqg$-modules with the action given by $\Delta$ and $\Delta'$ respectively.
	The notation $M\odot N$ means that $\odot$ can be $\odot_1$ or $\odot_2$.	
\end{nota}	

We will use mixed tensor products involving $\odot_1$ and $\odot_2$.
\begin{defi}Given a $\Uqg$-module $M$ and $\dtuple = (l_1, \dots, l_{d-1})\in\{1,2\}^{d-1}$, we define the corresponding \emph{mixed tensor product of $M$} as $M^{\odot \dtuple}\coloneqq M\odot_{l_1}\cdots\odot_{l_{d-1}} M$. 
\end{defi}
\begin{warn} When writing mixed tensor products, we have to be careful with the bracketings, since e.g.~$(M\odot_1 N)\odot_2 P\ncong M\odot_1(N\odot_2 P)$ in general.
	If in the following we suppress the bracketing in mixed tensor products, we will implicitly always assume that the bracketing is left adjusted, e.g.~$M\odot_1 N\odot_2 P\coloneqq (M\odot_1 N)\odot_2 P$.
\end{warn}
Next, we analyze the $\Uqg$-linear endomorphisms of (mixed) tensor powers of $V$ and $V^\circledast$.
In analogy to $U_q(\lie{sl}_\bbZ)$, we expect to find a Hecke algebra action.

Recall the Hecke algebra $\Hecke_d$, which is the $\bbQ(q)$-algebra generated by $H_1$, $\dots$, $H_{d-1}$ subject to the relations
\begin{equation}\label{Heckerel}
	\begin{gathered}
		H_i^2 = 1+(q^{-1}-q)H_i\\
		H_iH_j = H_jH_i\quad\text{if }\abs{i-j}>1,\quad H_iH_{i+1}H_i=H_{i+1}H_iH_{i+1}.
	\end{gathered}
\end{equation}
Given a $\bbQ(q)$-vector space $W$ and a linear endomorphism $\phi$ of $W\otimes W$, define the endomorphisms $\phi_i\coloneqq\operatorname{id}^{\otimes (i-1)}\otimes \phi\otimes \operatorname{id}^{\otimes (d-i-1)}$ of $W^{\otimes d}$ for $i=1,\ldots, d-1$.
Then $\phi$ \emph{satisfies the Hecke relations} if \cref{Heckerel} hold with $H_i$ replaced by $\phi_i$.
\begin{prop}\label{heckeaction}
	Consider the natural right $\Uqg$-module $V$.
	The linear map
	\begin{align*}
		H\colon V\odot V\to V\odot V,\quad
		v_i\odot v_j\mapsto a_{ij}v_j\odot v_i +\delta_{i<j}(q^{-1}-q) v_i\odot v_j,
	\end{align*}
	is $\Uqg$-linear and satisfies the Hecke relations, where for $\odot=\odot_l$ we set 
	\begin{equation*}
		a_{ij}=\begin{cases}
			q^3&\text{if } i\geq j,~ i-l \text{ odd, }j-l\text{ even,}\\
			q^{-1}&\text{if }i\geq j, \text{ otherwise,}\\
			q^{-3}&\text{if }i<j,~i-l \text{ even, }j-l\text{ odd,}\\
			q&\text{if }i<j, \text{ otherwise.}
		\end{cases}	
	\end{equation*}
\end{prop}
\begin{proof}
	This follows by straight-forward calculations, see \Cref{proofHecke}, noting that 
	\begin{equation}\label{rkas}
		a_{ii}=q^{-1}\quad\text{and} \quad a_{ij}a_{ji}=1\quad\text{for any}\quad i\not=j.\qedhere
	\end{equation}
\end{proof}

There is no reason to prefer $V$ to $V^\circledast$.
Analogously to \Cref{heckeaction} we obtain: 

\begin{prop}\label{heckeaction2}
	The linear map
	\begin{align*}
		H^{\circledast}\colon V^{\circledast}\odot V^{\circledast}\to V^{\circledast}\odot V^{\circledast},\quad
		v^i\odot v^j\mapsto a_{ji}v^j\odot v^i +\delta_{i<j}(q^{-1}-q) v^i\odot v^j,
	\end{align*}
	is $\Uqg$-linear and satisfies the Hecke relations, with $a_{ij}$ as in \cref{heckeaction}.
\end{prop}
\begin{rem}\label{heckeaction3}
	As a consequence of \Cref{heckeaction,heckeaction2} we obtain a (left) action of $\Hecke_d$ on any $d$-fold mixed tensor product $V^{\odot\dtuple}$ of $V$ by $\Uqg$-module homomorphisms commuting with the (right) $\Uqg$-action.
	With our implicit bracketing convention we have for instance $v_i\odot_1 v_j\odot_2 v_k\coloneqq(v_i\odot_1 v_j)\odot_2 v_k$.
	Then $F_i$ acts as $F_i\otimes K_{\beta_i}\otimes K_{\beta_{i-1}'}+1\otimes F_i\otimes K_{\beta_{i-1}'}+1\otimes1\otimes F_i$.
	This commutes with the $\Hecke_d$-action.
\end{rem}

\begin{rem}
	One can even check that the Hecke algebra centralizes $\Uqg$ and that we have an isomorphism
	\begin{equation*}
		\Hecke_d\to\End_{\Uqg}(V^{\otimes d}).
	\end{equation*}
	This even works for any mixed tensor product $V^{\odot\dtuple}$.
	
	To see this recall that by quantum Schur--Weyl duality, $\Hecke_d\cong\End_{U_q(\mathfrak{sl}_\bbZ)}(V^{\otimes d})$, where $H_i$ acts in $V\otimes V$ by $v_a\otimes v_b\mapsto v_b\otimes v_a+\delta_{a<b}(q^{-1}-q)v_a\otimes v_b$.
	We now claim that $V\otimes V\cong V\odot V$ as $\Hecke_2$-modules.
	Indeed, an isomorphism as desired is given by
	\begin{equation*}
		v_i\otimes v_j\mapsto\begin{cases}
			v_i\odot v_j&\text{if }i\geq j,\\
			a_{ji}v_i\odot v_j&\text{if }i<j.
		\end{cases}
	\end{equation*}
	This can now easily be extended to arbitrary mixed tensor products.
\end{rem}
The Hecke algebra actions from \Cref{heckeaction,heckeaction2} finally allow defining $q$-wedge products of $V$ and of $V^{\circledast}$.
\begin{defi}\label{wedge}
	Consider a mixed tensor product $V^{\odot\dtuple}$.
	We define the \emph{$q$-wedge product} $\mywedge^{\,\dtuple} V$ to be the subspace of $V^{\odot\dtuple}$ spanned by all elements of the form
	\begin{equation*}
		v_{i_1}\wedge v_{i_2}\wedge\cdots\wedge v_{i_d}\coloneqq \sum_{w\in\mathfrak{S}_d} (-q)^{\ell(w)}H_w(v_{i_1}\otimes v_{i_2}\otimes\cdots\otimes v_{i_d}),
	\end{equation*}
	for $i_1>i_2>\cdots>i_d$.
	If $\dtuple=(2,1,2,1,\dots)$, we just write $\mywedge^d V$ for $\mywedge^{\,\dtuple} V$.
\end{defi}
The goal of the next section is a definition of a Fock space $\Fock$  and its dual $\dFock$ for the electric Lie algebras $\el$. 
We use the $\mywedge^{\,\dtuple} V$ with their $\Uqg$-actions to define Fock spaces for $\el$  following in principle the standard constructions, \cite{LTFock}, as a space of semiinfinite wedges. In detail, the construction is however more involved. We have to make sure that the action of the Cartan part in $\Uqg$ is well-defined. For this the combination of the two monoidal structures  $\odot_1, \odot_2$, i.e.~the choice of $\dtuple=(2,1,2,1,\dots)$ in the definition of the $q$-wedge product will be crucial.

\subsection{The electric Fock space representations \texorpdfstring{$\Fockzero$ and $\dFockzero$}{}} 	
\label{Focksec}
In the following we will consider $V$ and $V^\circledast$ as right $\el$-modules: 
\begin{defi} \label{VandVdual}
	The \emph{natural $\el$-module $V$} is the vector space $V$ with the action restricted from $\Uqg$ to $\el$.
	In formulas, the action is given by $v_j\cE_i=\delta_{ij}(v_{i+1}+q^{\epsilon}v_{i-1})$.
	The \emph{dual natural $\el$-module $V^\circledast$} is the vector space $V^\circledast$ with the action of $\el$ given by the restriction from $\Uqg$ to $\el$ twisted by the shift anti-automorphism $\sigma$ from \cref{sigmadef}.
	In formulas, we have $v^j\cE_i =\delta_{i+2,j}q^{-\epsilon}v^{i+1}+\delta_{ij}v^{i+1}.$ 
\end{defi}	
Indeed, one calculates 
$v^j\cE_i = \sigma(\cE_i)v^j = q^{-\epsilon}\cE_{i+1}\cdot v^j = \delta_{i+2,j}q^{-\epsilon}v^{i+1}+\delta_{ij}v^{i+1}$.\
\begin{defi}\label{defifockspace}
	We define the \emph{Fock space} $\Fockzero$ as the vector space 
	\begin{equation}\label{defFock}
		\Fockzero \coloneqq \varinjlim \mywedge^d V, \quad\text{using the linear maps } \_\wedge v_{-d}\colon\mywedge^d V\to \mywedge^{d+1} V.
	\end{equation}
	
	The Fock space has as basis formal \emph{semiinfinite wedges}
	\[
	v_{i_1}\wedge v_{i_2}\wedge v_{i_3}\wedge\cdots,
	\]
	where $i_j>i_{j+1}$ and $i_j\neq 1-j$ for only finitely many $j\in\bbZ_{>0}$.
\end{defi}
Unfortunately, the action of $\Uqg$ on $q$-wedge products extends only partially to $\Fockzero$: 
\begin{prop}\label{actiononfock}
	Let $\Uqg^{\mathrm{fsupp}}$ be the subalgebra of $\Uqg$ generated by $E_i$, $F_i$ with $i\in\bbZ$ and $K_\lambda$ with $\lambda\in X^{\mathrm{fsupp}}$.
	Then there is a well-defined action of $\Uqg^{\mathrm{fsupp}}$ on $\Fockzero$ induced from the $\Uqg$-action on $q$-wedge products.	
\end{prop}
\begin{proof}
	Consider for any $d$, the map $\_\wedge v_{-d}\colon \mywedge^d V\to \mywedge^{d+1} V$.
	Our (implicit) choice of $\underline{l}$ implies that in the comultiplication of $F_i$ we obtain a $K_{\beta_i}$ in even spots and a $K_{\beta_i'}$ in odd spots.
	However for $i>-d$, we have $v_{-d}K_{\beta_i}=v_{-d}$ if $-d$ is even and have $v_{-d}K_{\beta_i'}=v_{-d}$ if $-d$ is odd.
	Hence, the action of $F_i$ is well-defined.
	Similarly, the action of $E_i$ is well-defined.
	By our assumption on $\lambda$, we have $v_{-d}K_\lambda=v_{-d}$ for $d\gg0$, hence the action is well-defined as well.
\end{proof}
We finally arrive at a well-defined electric Fock space: 
\begin{cor}\label{Fockaselrep}
	The action of $\Uqg^{\mathrm{fsupp}}$ on $\Fockzero$ restricts to a right action of $\el$.
\end{cor}
\begin{proof}
	From the formulas for $\el\subseteq\Uqg$ in \cref{elcoideal} we see that $\el\subseteq\Uqg^{\mathrm{fsupp}}$.
\end{proof}
\begin{defi}\label{dualFockel}
	Similarly to $\Fockzero$, we can define the \emph{dual Fock space} $\dFockzero$ using $V^\circledast$ instead of $V$.
	This has then a basis given by formal semiinfinite wedges
	\[
	v^{i_1}\wedge v^{i_2}\wedge v^{i_3}\wedge\cdots,
	\]
	where $i_j>i_{j+1}$ and $i_j\neq 1-j$ for only finitely many $j\in\bbZ_{>0}$.
	As above, we can identify partitions with the basis vectors of $\dFockzero$ and write $v^\lambda$ for the corresponding basis vector.
	With the same arguments we get an induced (left) action of $\Uqg^{\mathrm{fsupp}}$ on $\dFockzero$ and thus a right action of $\elinv$ via the shift automorphism from \cref{sigmadef}.
\end{defi}
We call $\Fockzero$ the \emph{electric Fock space} and $\dFockzero$ the \emph{dual electric Fock space}.
\begin{cor} \label{cyclic}Both, the Fock space $\Fockzero$ and the dual Fock space $\dFockzero$, are cyclic $\el$-module generated by the vacuum vector $v_\emptyset$.
\end{cor}
One can consider also Fock spaces $\Fock$ depending on a charge $\delta\in\bbR$.  For this let  $V_\delta= \bbQ(q)^\bbZ$ with basis $v_i$, with $i\in\delta+\bbZ$ and let $\Fock$ be the corresponding Fock space defines as before.  Via the identification of vector spaces $V \cong V_\delta$,  $v_i\mapsto v_{\delta+i}$  $V_\delta$ inherits an action of $\el$.   
Similarly, we define $\dFock$, the \emph{dual Fock} space of charge $\delta$.  
In the special case $\delta=0$ we have $\Fockzero_0=\Fockzero$ and $\dFockzero_0=\dFockzero$. The following is straight-forward: 

\begin{prop}All results in \cref{Focksec} hold for  $\Fock$, $\dFock$ instead of $\Fockzero$,  $\dFockzero$.
\end{prop}

\begin{lem}\label{annihi}
	The annihilator of $v_\emptyset\in\Fock$ and of  $v^\emptyset\in\dFock$  is the right ideal generated by $\cE_i$ for $i\neq\delta$ and by the two-sided ideal generated by $\cE_i^2$ for $i\in\bbZ$.
	
\end{lem}
\begin{proof}
	We compute the annihilator $A$ of $v_\emptyset$.
	By definition, we have $v_\emptyset \cE_i=0$ if and only if $i\neq \delta$.
	By \cref{elactonFock}, the element $\cE_i^2$ acts by $0$ on $\Fock$.
	Thus $J\subseteq A$. To show $J=A$ let $u = \cE_{i_1}\cdots \cE_{i_r}$ be a product of generators, that is nonzero in $\el/J$.
	In particular $i_1=0$ and we may assume moreover by \cref{electricipone} and \cref{electricimone} that $u$ is braid-avoiding, that is there is no subsequence in the sense of \cref{defsubseq} of the form $(j,j\pm1,j)$ for some $j$.
	It is straightforward to check that $(i_r, \dots, i_1)$ is the residue sequence of a partition $\lambda$, see \cite{NehmeKhovanovpn}*{Proposition 2.7} for a similar argument.
	By definition of the $\el$-action we have $v_\emptyset\cdot u = c_\lambda v_\lambda + \sum_{\abs{\mu}<\abs{\lambda}}c_\mu v_\mu$ with $c_\lambda\neq 0$.
	
	Conversely, any residue sequence $(i_r, \dots, i_1)$ defining $\lambda$ is braid-avoiding and provides, thanks to \cref{electricijdistant}, up to powers of $q$ the same element $u$ in $A\backslash\el$.
	Therefore, the linear map $J\backslash\el\to\Fock$, $u\mapsto v_\emptyset\cdot u$ is surjective and in fact an isomorphism.
	This shows the claim $J=A$.
	
	The same arguments show that $A$ is also the annihilator of $v^\emptyset\in\dFock$.
\end{proof}

As usual, see e.g.~\cite{ArikiLectures},  we label, depending on the fixed charge $\delta$, the basis vectors (i.e.~the semi-infinite wedges) of $\Fockzero$ by partitions. Namely, to a partition $\lambda$ we assign the basis vector with indices given by $\{\lambda_i+1-i+\delta\}$ and also write $v_\lambda$ for this basis vector. 
Similarly, for $\dFock$ using  $V_\delta^\circledast$  with basis vectors $v^i$,  $i\in\delta+\bbZ$.

\begin{rem} \label{elactonFock} Note that \Cref{VandVdual,wedge} provide explicit formulas for the $\el$-action on $\Fock$ and $\dFock$. 
	Up to powers of $q$, this action is given in the basis of partitions in familiar  terms (cf. e.g.~\cite{ArikiLectures}) using \cref{defres}: 
	\begin{itemize}[leftmargin=20pt]
		\item 
		$\cE_i$ sends a partition $\lambda$ to the linear combination of all partitions $\mu$ where a box of charged content $\delta+i$ was added or a box of charged content $\delta+i-1$ was removed.
		In the language of \cref{prelim} this means $\res(\lambda\to\mu)=\delta+i$.
		\item Similarly, for $\dFockzero$ we have that $\cE_i\in \elinv$ adds boxes of charged content $\delta+i$ and removes boxes of charged content $\delta+i+1$ that means $\res^\circledast(\lambda\to\mu)=\delta+i$.
	\end{itemize}
\end{rem}	

\begin{proof}
	This follows from \cref{Fockaselrep} using \Cref{dualFockel} and \cref{elactonFock}.
\end{proof}
\begin{rem}
	\cref{elactonFock} should justify the notation $\res$ and $\res^\circledast$ by referring to $\Fock$ and $\dFock$.
	The introduction of these two slightly different functions is necessary because $\sigma$ not only scales by a power of $q$, but also shifts the indices.
\end{rem}
\begin{lem}\label{transpose}
	There is a unique isomorphism of vector spaces $\tau\colon\Fock\to\dFock$ satisfying $\tau(v_\emptyset)=v_\emptyset$ and $\tau(v\cE)=\tau(v)\tau(\cE)$  for $v\in\Fock$, $\cE\in\el$. 
\end{lem}
\begin{rem} Up to some $q$-power, $\tau$  transposes the partition, i.e.~$\tau(v_\la)=q^{c(\la)}v_{\la^t}$. 
\end{rem}
\begin{proof}
	This follows directly from \Cref{taudef} and  \Cref{cyclic}, since the annihilator $J$ in \cref{annihi} is $\tau$-invariant.
\end{proof}
\subsection{Pairing and Bar involution on Fock spaces}
From the definition we expect $\dFock$ to be dual to $\Fock$ via the following pairing. 
\begin{defi}
	Define a $q$-bilinear pairing $(\_,\_)\colon\dFock\otimes\Fock\to\bbQ(q)$ by
	\begin{equation*}
		(v^\lambda, v_\mu)=\delta_{\lambda\mu}.
	\end{equation*}
\end{defi}
\begin{lem}\label{lemmaadjoint}
	The bilinear pairing satisfies (with $\sigma $ as in \cref{sigmadef})
	\begin{equation*}
		(wu,v) = (w,v\sigma(u))
	\end{equation*}
	for all $w\in\dFock$, $v\in\Fock$ and $u\in\elinv$.
\end{lem}
\begin{proof}This holds by definition recalling the twist by $\sigma$ in the action.
\end{proof}	
\begin{warn}
	For readers familiar with categorification the shift $\sigma$ appearing in \cref{lemmaadjoint} should be alarming, since we cannot expect that a functor categorifying $\cE_i$ is self-adjoint (even up to grading shifts).
	One should also observe that we do not define a scalar product on $\Fock$, but only a pairing with $\dFock$. 
\end{warn}

We next define a bar involution compatible with the bar involution on $\el$ and $\elinv$.

\begin{prop}\label{barinvfock}
	There exists a \emph{bar involution on $\Fock$}, that is a unique $q$-antilinear isomorphism $\overline{\phantom{U}}\colon\Fock\to\dFock$ satisfying $\overline{v_\emptyset}\coloneqq v^\emptyset$ and $\overline{u.v}=\overline{u}.\overline{v}$.	
\end{prop}
\begin{proof} 
	By \cref{cyclic}, the $\el$-module $\Fock$ is cyclic with generator $v_\emptyset$.
	Therefore, the bar involution on $\Fock$ is unique if it exists.
	If $A\subseteq\el$ is the annihilator of $v_\emptyset$, then $A\backslash \el\to\Fock$, $u\mapsto v_\emptyset\cdot u$ is an isomorphism of (right) $\el$-modules. Now $A=J$ by \cref{annihi}. Since $J$ is obviously preserved under the bar involutions on $\el$ and $\elinv$, the desired (unique) bar involution maps on $\Fock$ and $\dFock$ exist.
\end{proof}

\begin{defi}\label{lFock}
	For a charge vector $\chargevec$ and a level $\level$ we define the \emph{level $\level$ Fock space} 
	$\Fockzero_{\chargevec,\level}=\Fockzero_{\delta_1}\otimes \cdots \otimes\Fockzero_{\delta_\level}$ of charge $\chargevec$. 
	It comes with an obvious $\left({\el}\right)^{\otimes\level}$-action.   
\end{defi}
\cref{elactonFock} generalises to higher levels by identifying the standard basis vectors from $\Fockzero_{\chargevec,\level}$ with $\level$-multipartitions  and then using  the residue functions \cref{deflres}.
\section{The electric KLR-category \texorpdfstring{$\sR$}{}}
The goal of this section is to introduce a new monoidal supercategory, the electric KLR-category, by generators and relations and describe some basic properties.
The morphism spaces assemble into electric KLR-algebras which should be seen as analogues of the KLR algebras from \cite{KL09}, \cite{Rouquier}. 
\begin{nota}\label{defsVec}
	For this section we fix a ground field $\bbk$ and denote by $\sVec$ the symmetric monoidal category of $\bbk$-vector superspaces with (super) degree preserving morphisms.
	For $V\in\sVec$ we denote by $\abs{v}\in\{0,1\}$ the degree of $v\in V$ implicitly assuming $v$ to be homogeneous. 
\end{nota}	
Thus, in $\sVec$, the braiding morphisms are $v\otimes w\mapsto (-1)^{|v||w|}w\otimes v$.
By a \emph{supercategory} we mean a $\sVec$-category, i.e.~a category enriched in $\sVec$, in the sense of \cite{Kelly}.
Moreover, $\sVec$ has a symmetric braiding and we can consider monoidal supercategories.
Morphisms in these satisfy $(f\otimes 1)(1\otimes g)=(-1)^{|f||g|}(1\otimes g)(1\otimes f)$.

\subsection{The definitions}
For basics on monoidal supercategories we refer for instance to \cite{BrEl}, \cite{CEIII}.
We denote by $\mathbf{1}$ the monoidal unit in a given monoidal (super)category.
\begin{defi}\label{DefKLR}
	Let $\mathrm{sR}(\bbZ)$ be the $\bbk$-linear strict monoidal supercategory freely generated on the level of objects 
	by $a\in \bbZ$ and on the level of morphisms by 
	\begin{tabularx}{\textwidth}{lcllcl}
		\text{even generators:}
		&
		\begin{tikzpicture}[line width=\lw, myscale=0.6]
			\draw (0,0) -- (1,1) (1,0)--(0,1);\node[fill=white, anchor=north] at (0,0) {$a$}; \node[fill=white, anchor=north] at (1,0) {$b$};\node[fill=white, anchor=south] at (0,1) {$b$}; \node[fill=white, anchor=south] at (1,1) {$a$};
		\end{tikzpicture} 
		&$\colon a\otimes b\rightarrow b\otimes a$,&\quad&
		\begin{tikzpicture}[line width=\lw, myscale=0.6]
			\draw (0,0)--(0,1) node[midway] (a) {};\fill (a) circle (\dw);\node[fill=white, anchor=north] at (0,0) {$a$};\node[fill=white, anchor=south] at (0,1) {$a$};
		\end{tikzpicture} 
		&$ \colon a\rightarrow a,$\\
		
		\text{odd generators:}
		&
		\begin{tikzpicture}[line width=\lw, myscale=0.6]
			\node at (0,0) (A){$a+1$}; \node at (1,0) (B) {$a$};\draw (A)..controls +(0.2,-0.8) and +(-0.2,-0.8)..(B);
		\end{tikzpicture}
		&$ \colon\mathbf{1}\rightarrow (a+1)\otimes a,$&
		&\begin{tikzpicture}[line width=\lw, myscale=0.6]
			\node at (0,0) (A){$a$}; \node at (1,0) (B) {$a+1$};\draw (A)..controls +(0.2,0.8) and +(-0.2,0.8)..(B);
		\end{tikzpicture}
		&$\colon a\otimes (a+1)\rightarrow\mathbf{1} ,$\\
		&\quad
	\end{tabularx}
	modulo the following (local) relations \cref{ncdotcap}-\cref{ncbraid}: 
	
	\begin{tabularx}{\textwidth}{lXll}
		\refstepcounter{sr}(\thesr)\label{ncdotcap}&$
		\begin{tikzpicture}[line width=\lw, myscale=0.5]
			\node at (0,0) (A) {$a$};
			\node at (1,0) (B) {$a+1$};
			\draw (A)..controls +(0,1) and +(0,1)..(B) node[pos=0.1] (a) {};
			\fill (a) circle(\dw);
		\end{tikzpicture}=\begin{tikzpicture}[line width=\lw, myscale=0.5]
			\node at (0,0) (A) {$a$};
			\node at (1,0) (B) {$a+1$};
			\draw (A)..controls +(0,1) and +(0,1)..(B) node[pos=0.9] (a) {};
			\fill (a) circle(\dw);
		\end{tikzpicture}$&\refstepcounter{sr}(\thesr)\label{ncdotcrossing}&
		$\begin{tikzpicture}[line width=\lw, myscale=0.55]
			\node at (0,0) (A) {$a$};
			\node at (1,0) (B) {$b$};
			\node at (0,1.5) (C) {$b$};
			\node at (1,1.5) (D) {$a$};
			\draw (A.north)--(D.south) node[pos=0.9] (a) {} (B.north)--(C.south);
			\fill (a) circle(\dw);
		\end{tikzpicture}-\begin{tikzpicture}[line width=\lw, myscale=0.55]
			\node at (0,0) (A) {$a$};
			\node at (1,0) (B) {$b$};
			\node at (0,1.5) (C) {$b$};
			\node at (1,1.5) (D) {$a$};
			\draw (A.north)--(D.south) node[pos=0.1] (a) {} (B.north)--(C.south);
			\fill (a) circle(\dw);
		\end{tikzpicture}=\begin{cases}
			\begin{tikzpicture}[line width=\lw, myscale=0.55]
				\node at (0,0) (A) {$a$};
				\node at (1,0) (B) {$a$};
				\node at (0,1.5) (C) {$a$};
				\node at (1,1.5) (D) {$a$};
				\draw (A.north)--(C.south) (B.north)--(D.south);
			\end{tikzpicture}&\text{if $b=a$,}\\
			\begin{tikzpicture}[line width=\lw, myscale=0.55]
				\node at (0,0) (A) {$a$};
				\node at (1,0) (B) {$a+1$};
				\node at (0,1.5) (C) {$a+1$};
				\node at (1,1.5) (D) {$a$};
				\draw (A)..controls +(0.2,0.8) and +(-0.2,0.8)..(B)(C)..controls +(0.2,-0.8) and +(-0.2,-0.8)..(D);
			\end{tikzpicture}&\text{if $b=a+1$,}\\
			0&\text{otherwise,}
		\end{cases}$\\[\baselineskip]
		\refstepcounter{sr}(\thesr)\label{nctangleone}&
		$\begin{tikzpicture}[line width=\lw, myscale=0.5]	
			\node at (0,2.5) (A) {$a$};
			\node at (0,1.5) (B) {$a$};
			\node at (1,1.5) (C) {$b$};
			\node at (2,1.5) (D) {$b+1$};
			\node at (0,0) (E) {$b$};
			\node at (1,0) (F) {$a$};
			\node at (2,0) (G) {$b+1$};
			\draw (A)--(B) (B.south)--(F.north) (C.south)--(E.north) (G)--(D) (C)..controls +(0.2,0.8) and +(-0.2,0.8)..(D);
		\end{tikzpicture}=
		\begin{tikzpicture}[line width=\lw, myscale=0.5]	
			\node at (2,2.5) (A) {$a$};
			\node at (0,1.5) (B) {$b$};
			\node at (1,1.5) (C) {$b+1$};
			\node at (2,1.5) (D) {$a$};
			\node at (0,0) (E) {$b$};
			\node at (1,0) (F) {$a$};
			\node at (2,0) (G) {$b+1$};
			\draw (A)--(D) (D.south)--(F.north) (C.south)--(G.north) (E)--(B) (B)..controls +(0.2,0.8) and +(-0.2,0.8)..(C);
		\end{tikzpicture}$&\refstepcounter{sr}(\thesr)\label{ncsnake}&
		$-\begin{tikzpicture}[line width=\lw, myscale=0.5]
			\node at (0,0) (A){$a$};
			\node at (0,1) (B){$a$};
			\node at (1,1) (C){$a+1$};
			\node at (2,1) (D){$a$};
			\node at (2,2) (E){$a$};
			\draw (A.north)--(B.south) (B)..controls +(0.2,0.8) and +(-0.2,0.8)..(C)(C)..controls +(0.2,-0.8) and +(-0.2,-0.8)..(D) (D)--(E);
		\end{tikzpicture}=\begin{tikzpicture}[line width=\lw, myscale=0.5]
			\node at (0,0) (A) {$a$};	\node at (0,2) (B) {$a$};\draw (A) -- (B);
		\end{tikzpicture}=
		\begin{tikzpicture}[line width=\lw, myscale=0.5]
			\node at (2,0) (A){$a$};
			\node at (2,1) (B){$a$};
			\node at (1,1) (C){$a-1$};
			\node at (0,1) (D){$a$};
			\node at (0,2) (E){$a$};
			\draw (A.north)--(B.south) (C)..controls +(0.2,0.8) and +(-0.2,0.8)..(B)(D)..controls +(0.2,-0.8) and +(-0.2,-0.8)..(C) (D)--(E);
		\end{tikzpicture}$\\[\baselineskip]
		\refstepcounter{sr}(\thesr)\label{ncuntwist}&
		$\begin{tikzpicture}[line width=\lw, myscale=0.5]
			\node at (0,0) (A) {$a+1$};\node at (1,0) (B) {$a$};
			\node at (0,1.5) (C) {$a$};\node at (1,1.5) (D) {$a+1$};
			\draw (A.north)--(D.south) (B.north) --(C.south);
			\draw (C)..controls +(0.2,0.8) and +(-0.2,0.8)..(D);
		\end{tikzpicture}=0$&\refstepcounter{sr}(\thesr)\label{ncinverse}&
		$\begin{tikzpicture}[line width=\lw, myscale=0.5]
			\node at (0,0) (A) {$a$};\node at (1,0) (B) {$b$};
			\node at (0,1.5) (C) {$b$};\node at (1,1.5) (D) {$a$};
			\node at (0,3) (E) {$a$};\node at (1,3) (F) {$b$};
			\draw (A.north)--(D.south) (B.north) --(C.south) (D.north)--(E.south) (C.north)--(F.south);
		\end{tikzpicture}=\begin{cases}
			0&\text{if $a=b$,}\\
			\begin{tikzpicture}[line width=\lw, myscale=0.5]
				\node at (0,0) (A) {$a$};\node at (1,0) (B) {$b$};
				\node at (0,1.5) (C) {$a$};\node at (1,1.5) (D) {$b$};
				\draw (A.north)--(C.south) node[midway] (a) {} (B.north) --(D.south);
				\fill (a) circle(\dw);
			\end{tikzpicture}-
			\begin{tikzpicture}[line width=\lw, myscale=0.5]
				\node at (0,0) (A) {$a$};\node at (1,0) (B) {$b$};
				\node at (0,1.5) (C) {$a$};\node at (1,1.5) (D) {$b$};
				\draw (A.north)--(C.south) (B.north) --(D.south) node[midway] (a) {} ;
				\fill (a) circle(\dw);
			\end{tikzpicture}&\text{if $b=a-1$,}\\
			\begin{tikzpicture}[line width=\lw, myscale=0.5]
				\node at (0,0) (A) {$a$};\node at (1,0) (B) {$b$};
				\node at (0,1.5) (C) {$a$};\node at (1,1.5) (D) {$b$};
				\draw (A.north)--(C.south) (B.north) --(D.south) node[midway] (a) {} ;
				\fill (a) circle(\dw);
			\end{tikzpicture}-\begin{tikzpicture}[line width=\lw, myscale=0.5]
				\node at (0,0) (A) {$a$};\node at (1,0) (B) {$b$};
				\node at (0,1.5) (C) {$a$};\node at (1,1.5) (D) {$b$};
				\draw (A.north)--(C.south) node[midway] (a) {} (B.north) --(D.south);
				\fill (a) circle(\dw);
			\end{tikzpicture}&\text{if $b=a+1$,}\\
			\begin{tikzpicture}[line width=\lw, myscale=0.5]
				\node at (0,0) (A) {$a$};\node at (1,0) (B) {$b$};
				\node at (0,1.5) (C) {$a$};\node at (1,1.5) (D) {$b$};
				\draw (A.north)--(C.south) (B.north) --(D.south);
			\end{tikzpicture} &\text{otherwise,}
		\end{cases}$\\[\baselineskip]
		\refstepcounter{sr}(\thesr)\label{ncbraid}&
		\multicolumn{3}{l}{$\begin{tikzpicture}[line width=\lw, myscale=0.5]
			\node at (0,0) (A) {$a$};
			\node at (1,0) (B) {$b$};
			\node at (2,0) (C) {$c$};
			\node at (0,1.5) (D) {$b$};
			\node at (1,1.5) (E) {$a$};
			\node at (2,1.5) (F) {$c$};
			\node at (0,3) (G) {$b$};
			\node at (1,3) (H) {$c$};
			\node at (2,3) (I) {$a$};
			\node at (0,4.5) (J) {$c$};
			\node at (1,4.5) (K) {$b$};
			\node at (2,4.5) (L) {$a$};
			\draw (A.north)--(E.south) (E.north)--(I.south) (I)--(L)(B.north)--(D.south) (D)--(G) (G.north) -- (K.south)(C)--(F) (F.north)--(H.south)(H.north)--(J.south);
		\end{tikzpicture}-\begin{tikzpicture}[line width=\lw, myscale=0.5]
			\node at (0,0) (A) {$a$};
			\node at (1,0) (B) {$b$};
			\node at (2,0) (C) {$c$};
			\node at (0,1.5) (D) {$a$};
			\node at (1,1.5) (E) {$c$};
			\node at (2,1.5) (F) {$b$};
			\node at (0,3) (G) {$c$};
			\node at (1,3) (H) {$a$};
			\node at (2,3) (I) {$b$};
			\node at (0,4.5) (J) {$c$};
			\node at (1,4.5) (K) {$b$};
			\node at (2,4.5) (L) {$a$};
			\draw (C.north)--(E.south) (E.north)--(G.south) (G)--(J)(B.north)--(F.south) (F)--(I) (I.north) -- (K.south)(A)--(D) (D.north)--(H.south)(H.north)--(L.south);
		\end{tikzpicture}=\begin{cases}
			\begin{tikzpicture}[line width=\lw, myscale=0.5]
				\node at (0,0) (A) {$a$};
				\node at (1,0) (B) {$a+1$};
				\node at (2,0) (C) {$a$};
				\node at (0,1.5) (D) {$a$};
				\node at (1,1.5) (E) {$a+1$};
				\node at (2,1.5) (F) {$a$};
				\draw(A)--(D) (B)--(E) (C)--(F);
			\end{tikzpicture}+
			\begin{tikzpicture}[line width=\lw, myscale=0.5]
				\node at (0,0) (A) {$a$};
				\node at (1,0) (B) {$a+1$};
				\node at (2,0) (C) {$a$};
				\node at (0,1.5) (D) {$a$};
				\node at (1,1.5) (E) {$a+1$};
				\node at (2,1.5) (F) {$a$};
				\draw(A)..controls +(0.2,0.8) and 	+(-0.2,0.8)..(B) (E)..controls +(0.2,-0.8) and 	+(-0.2,-0.8)..(F) (C)--(D);
			\end{tikzpicture}&\text{if $c=a=b-1$,}\\
			-\begin{tikzpicture}[line width=\lw, myscale=0.5]
				\node at (0,0) (A) {$a$};
				\node at (1,0) (B) {$a-1$};
				\node at (2,0) (C) {$a$};
				\node at (0,1.5) (D) {$a$};
				\node at (1,1.5) (E) {$a-1$};
				\node at (2,1.5) (F) {$a$};
				\draw(A)--(D) (B)--(E) (C)--(F);
			\end{tikzpicture}+
			\begin{tikzpicture}[line width=\lw, myscale=0.5]
				\node at (0,0) (A) {$a$};
				\node at (1,0) (B) {$a-1$};
				\node at (2,0) (C) {$a$};
				\node at (0,1.5) (D) {$a$};
				\node at (1,1.5) (E) {$a-1$};
				\node at (2,1.5) (F) {$a$};
				\draw(B)..controls +(0.2,0.8) and 	+(-0.2,0.8)..(C) (D)..controls +(0.2,-0.8) and 	+(-0.2,-0.8)..(E) (A)--(F);
			\end{tikzpicture}&\text{if $c=a=b+1$,}\\
			0&\text{otherwise.}
		\end{cases}$}
	\end{tabularx}
\end{defi}
\begin{rem}For simplicity we work over $\bbk$ and not an arbitrary ground ring or $\bbZ$.
\end{rem}
This definition makes also sense when we replace the set of objects/of labels of the strands by any set $\bbR'$ with an automorphism $(+1)\colon \bbR'\to \bbR'$. 
In particular $\bbR'=\bbR$ as in \cref{defR} works. Objects in $\sRwo(\bbR')$ are then (possibly empty) finite sequences $\mathbf{a}$ of elements in $\bbR'$. We will denote the resulting category $\sRwo(\bbR')$, but mostly work from now on with $\sRwo\coloneqq\sRwo(\bbR)$.

\begin{lem}
	The defining relations \cref{ncdotcap}-\cref{ncbraid} imply the following equalities:\vspace{2mm}\\
	{\normalfont
	\begin{tabularx}{\textwidth}{l>{\centering\arraybackslash}Xlc}
		\refstepcounter{equation}(\theequation)\label{ncdotcup}&$\begin{tikzpicture}[line width=\lw, myscale=0.55]
			\node at (0,0) (A) {$a$};
			\node at (1,0) (B) {$a-1$};
			\draw (A)..controls +(0,-1) and +(0,-1)..(B) node[pos=0.1] (a) {};
			\fill (a) circle(\dw);
		\end{tikzpicture}=\begin{tikzpicture}[line width=\lw, myscale=0.55]
			\node at (0,0) (A) {$a$};
			\node at (1,0) (B) {$a-1$};
			\draw (A)..controls +(0,-1) and +(0,-1)..(B) node[pos=0.9] (a) {};
			\fill (a) circle(\dw);
		\end{tikzpicture}$
		&\refstepcounter{equation}(\theequation)\label{nctangletwo}&
		$\begin{tikzpicture}[line width = \lw, yscale=-1,myscale=0.55]	
			\node at (2,2.5) (A) {$a$};
			\node at (0,1.5) (B) {$b$};
			\node at (1,1.5) (C) {$b-1$};
			\node at (2,1.5) (D) {$a$};
			\node at (0,0) (E) {$b$};
			\node at (1,0) (F) {$a$};
			\node at (2,0) (G) {$b-1$};
			\draw (A)--(D) (D.north)--(F.south) (C.north)--(G.south) (E)--(B) (B)..controls +(0.2,0.8) and +(-0.2,0.8)..(C);
		\end{tikzpicture}
		=
		\begin{tikzpicture}[line width = \lw, yscale=-1,myscale=0.55]	
			\node at (0,2.5) (A) {$a$};
			\node at (0,1.5) (B) {$a$};
			\node at (1,1.5) (C) {$b$};
			\node at (2,1.5) (D) {$b-1$};
			\node at (0,0) (E) {$b$};
			\node at (1,0) (F) {$a$};
			\node at (2,0) (G) {$b-1$};
			\draw (A)--(B) (B.north)--(F.south) (C.north)--(E.south) (G)--(D) (C)..controls +(0.2,0.8) and +(-0.2,0.8)..(D);
		\end{tikzpicture}$\\
		\refstepcounter{equation}(\theequation)\label{ncuntwistii}&
		$\begin{tikzpicture}[line width = \lw,myscale=0.55]
			\node at (0,0) (A) {$a+1$};\node at (1,0) (B) {$a$};
			\node at (0,1.5) (C) {$a$};\node at (1,1.5) (D) {$a+1$};
			\draw (A.north)--(D.south) (B.north) --(C.south);
			\draw (A)..controls +(0.2,-0.8) and +(-0.2,-0.8)..(B);
		\end{tikzpicture}=0$&
		\refstepcounter{equation}(\theequation)\label{ncdotcrossingii}&
		$\begin{tikzpicture}[line width=\lw, myscale=0.5]
			\node at (0,0) (A) {$a$};
			\node at (1,0) (B) {$b$};
			\node at (0,1.5) (C) {$b$};
			\node at (1,1.5) (D) {$a$};
			\draw (A.north)--(D.south) (B.north)--(C.south) node[pos=0.1] (a) {};
			\fill (a) circle(\dw);
		\end{tikzpicture}-\begin{tikzpicture}[line width=\lw, myscale=0.5]
			\node at (0,0) (A) {$a$};
			\node at (1,0) (B) {$b$};
			\node at (0,1.5) (C) {$b$};
			\node at (1,1.5) (D) {$a$};
			\draw (A.north)--(D.south) (B.north)--(C.south) node[pos=0.9] (a) {} ;
			\fill (a) circle(\dw);
		\end{tikzpicture}=\begin{cases}
			\begin{tikzpicture}[line width=\lw, myscale=0.5]
				\node at (0,0) (A) {$a$};
				\node at (1,0) (B) {$a$};
				\node at (0,1.5) (C) {$a$};
				\node at (1,1.5) (D) {$a$};
				\draw (A.north)--(C.south) (B.north)--(D.south);
			\end{tikzpicture}&\text{if $b=a$,}\\
			-\begin{tikzpicture}[line width=\lw, myscale=0.5]
				\node at (0,0) (A) {$a$};
				\node at (1,0) (B) {$a+1$};
				\node at (0,1.5) (C) {$a+1$};
				\node at (1,1.5) (D) {$a$};
				\draw (A)..controls +(0.2,0.8) and +(-0.2,0.8)..(B)(C)..controls +(0.2,-0.8) and +(-0.2,-0.8)..(D);
			\end{tikzpicture}&\text{if $b=a+1$,}\\
			0&\text{otherwise,}
		\end{cases}$
	\end{tabularx}}
\end{lem}
\begin{proof}
	The relations \cref{ncdotcup,nctangletwo} follow from \cref{ncdotcap} respectively \cref{nctangleone} using \cref{ncsnake}.
	After adding a snake \cref{ncsnake} to the left-hand side of \cref{ncuntwistii}, \cref{ncuntwistii} follows with \cref{nctangletwo}, \cref{nctangleone} from \cref{ncuntwist}.
	Finally, \cref{ncdotcrossingii} follows from \cref{ncdotcrossing} by rotation, i.e.~by adding a cup to the bottom and a cap to the top and then applying \cref{nctangleone} and \cref{ncsnake}.
\end{proof}

We denote by $\gosVec$ (and $\gsVec$) the monoidal category of $\bbZ$-graded vector superspaces with supergrading preserving morphisms which preserve (respectively not necessarily preserve) the $\bbZ$-degree.
The braiding morphisms are the flip maps adjusted by signs only with respect to the super grading and not the $\bbZ$-grading.

\begin{prop}\label{grading}
	Let $\epsilon\in\{\pm1\}$.
	Then $\sRwo(\bbR)$, or more generally  $\sRwo(\bbR)$, can be viewed as a monoidal $\gosVec$-category $\sR$ by setting 
	\begin{equation}\label{genshoms}
		\begin{split}
			\operatorname{deg}\left(\begin{tikzpicture}[line width=\lw, myscale=0.6]
				\draw (0,0)--(0,1) node[midway] (a) {} ;
				\fill (a) circle(\dw);
				\node[fill=white, anchor=north] at (0,0) {$a$};\node[fill=white, anchor=south] at (0,1) {$a$};
			\end{tikzpicture}\right)=2, 	\quad
			\operatorname{deg}\left(\begin{tikzpicture}[line width=\lw, myscale=0.6]
				\node at (0,0) (A){$a+1$}; \node at (1,0) (B) {$a$};\draw (A)..controls +(0.2,-0.8) and +(-0.2,-0.8)..(B);
			\end{tikzpicture} \right)=-\epsilon,\quad 
			\operatorname{deg}\left(\begin{tikzpicture}[line width=\lw, myscale=0.6]
				\node at (0,0) (A){$a$}; \node at (1,0) (B) {$a+1$};\draw (A)..controls +(0.2,0.8) and +(-0.2,0.8)..(B);
			\end{tikzpicture} \right)=\epsilon,\\ 
			\operatorname{deg}\left(\begin{tikzpicture}[line width=\lw, myscale=0.6]
				\draw (0,0) -- (1,1) (1,0)--(0,1);\node[fill=white, anchor=north] at (0,0) {$a$}; \node[fill=white, anchor=north] at (1,0) {$b$};\node[fill=white, anchor=south] at (0,1) {$b$}; \node[fill=white, anchor=south] at (1,1) {$a$};
			\end{tikzpicture}\right)=
			\begin{cases}
				-2&\text{if $b=a,a+1$,}\\
				0&\text{if $b-a\notin\bbZ$,}\\
				4\sgn(b-a)(-1)^{b-a}&\text{otherwise.}
			\end{cases}
		\end{split}
	\end{equation}	
\end{prop}		
\begin{proof}It suffices to check that \cref{genshoms} is compatible with \cref{ncdotcap}-\cref{ncbraid}.
\end{proof}
\begin{rem}\label{uniquegrad}
	The (surprisingly difficult) degrees for the crossings are forced on us if we  require for $\sR(\bbZ)$ the dot generator to be of degree $2$, i.e.~compatible with the usual KLR convention. First, by \cref{ncdotcrossing}  the crossings labelled $(a,a)$ must have degree $-2$. Then \cref{nctangleone} forces the crossing labelled $(a,a+1)$ at the bottom to be of degree $-2$. Finally, by \cref{ncinverse}, the crossing labelled $(a+1,a)$ must have degree $4$.  The arguments from \cite[Definition 1.6, Lemma 1.8]{NehmeKhovanovpn}  imply then that the degrees of the other crossings are also forced. The degree for a $(a,a+1)$ cap can be an arbitrary integer $\epsilon_a$ as long as the $(a+1,a)$ cup has degree $-\epsilon_a$.  If we require independent of $a$,  our choices for $\epsilon$ are unique up to an overall positive scaling. 
\end{rem}
\begin{defi} \label{defKLRcorrect}	
	The monoidal $\gosVec$-category $\sR$ is the \emph{electric KLR category}.
\end{defi}

\begin{defi}\label{locallyunital}
	The category $\sR$ can be viewed as a locally unital algebra $\sR$ with set of idempotents labelled by $\bbR$, namely $\sR=\bigoplus_{\bf{a}, \bf{b}\in\bbR}1_{\bf a}\sR1_{\bf b}$, 
	where $1_{\bf a}\sR1_{\bf b}$ is the $\bbZ$-graded vector superspace of all morphisms from ${\bf a}$ to ${\bf b}$ in $\sR$.
	More precisely, $\sR$ is a $\mathbb{Z}$-graded superalgebra (that is an algebra object in $\gosVec$).
	We call this algebra the \emph{electric KLR (super)algebra}.
\end{defi}

For a supercategory $\cC$ we denote by $\cC^{\mathrm{op}}$ its opposite supercategory.
If $\cC$ is moreover monoidal, let $\cC^ {\mathrm{rev}}$ the category $\cC$ with the opposite monoidal structure $a\otimes_{\mathrm{rev}}b=b\otimes a$ on objects and $f\otimes_{\mathrm{rev}}g=(-1)^{\abs{f}\abs{g}}g\otimes f$.
Denote $\cC^{\mathrm{oprev}}=(\cC^{\mathrm{op}})^\mathrm{rev}\cong(\cC^{\mathrm{rev}})^\mathrm{op}$.
\begin{lem}\label{barinvsigmacat}
	There are equivalences of monoidal $\sVec$-categories
	\begin{align*}
		\Sigma\colon\sR^{\mathrm{oprev}}&\to\sR, &{\Tau}\colon\sR^{\mathrm{op}}&\to\sR,\\
		a&\mapsto a+1,&a&\mapsto -a,\\
		\begin{tikzpicture}[line width=\lw, myscale=0.6]
			\draw (0,0) -- (1,1) (1,0)--(0,1);\node[fill=white, anchor=north] at (0,0) {$a$}; \node[fill=white, anchor=north] at (1,0) {$b$};\node[fill=white, anchor=south] at (0,1) {$b$}; \node[fill=white, anchor=south] at (1,1) {$a$};
		\end{tikzpicture}&\mapsto -\begin{tikzpicture}[line width=\lw, myscale=0.6]
			\draw (0,0) -- (1,1) (1,0)--(0,1);\node[fill=white, anchor=north] at (0,0) {$a+1$}; \node[fill=white, anchor=north] at (1,0) {$b+1$};\node[fill=white, anchor=south] at (0,1) {$b+1$}; \node[fill=white, anchor=south] at (1,1) {$a+1$};
		\end{tikzpicture},
		&		\begin{tikzpicture}[line width=\lw, myscale=0.6]
			\draw (0,0) -- (1,1) (1,0)--(0,1);\node[fill=white, anchor=north] at (0,0) {$a$}; \node[fill=white, anchor=north] at (1,0) {$b$};\node[fill=white, anchor=south] at (0,1) {$b$}; \node[fill=white, anchor=south] at (1,1) {$a$};
		\end{tikzpicture}&\mapsto \eta\begin{tikzpicture}[line width=\lw, myscale=0.6]
			\draw (0,0) -- (1,1) (1,0)--(0,1);\node[fill=white, anchor=north] at (0,0) {$-b$}; \node[fill=white, anchor=north] at (1,0) {$-a$};\node[fill=white, anchor=south] at (0,1) {$-a$}; \node[fill=white, anchor=south] at (1,1) {$-b$};
		\end{tikzpicture},
		\\
		\begin{tikzpicture}[line width=\lw, myscale=0.6]
			\node at (0,0) (A){$a$}; \node at (1,0) (B) {$a-1$};\draw (A)..controls +(0.2,-0.8) and +(-0.2,-0.8)..(B);
		\end{tikzpicture}
		&\mapsto-\begin{tikzpicture}[line width=\lw, myscale=0.6]
			\node at (0,0) (A){$a$}; \node at (1,0) (B) {$a+1$};\draw (A)..controls +(0.2,0.8) and +(-0.2,0.8)..(B);
		\end{tikzpicture},
		&\begin{tikzpicture}[line width=\lw, myscale=0.6]
			\node at (0,0) (A){$a+1$}; \node at (1,0) (B) {$a$};\draw (A)..controls +(0.2,-0.8) and +(-0.2,-0.8)..(B);
		\end{tikzpicture}&\mapsto\begin{tikzpicture}[line width=\lw, myscale=0.6]
			\node at (0,0) (A){$-a-1$}; \node at (1,0) (B) {$-a$};\draw (A)..controls +(0.2,0.8) and +(-0.2,0.8)..(B);
		\end{tikzpicture},&\\
		\begin{tikzpicture}[line width=\lw, myscale=0.6]
			\node at (0,0) (A){$a-1$}; \node at (1,0) (B) {$a$};\draw (A)..controls +(0.2,0.8) and +(-0.2,0.8)..(B);
		\end{tikzpicture}&\mapsto\begin{tikzpicture}[line width=\lw, myscale=0.6]
			\node at (0,0) (A){$a+1$}; \node at (1,0) (B) {$a$};\draw (A)..controls +(0.2,-0.8) and +(-0.2,-0.8)..(B);
		\end{tikzpicture},
		&\begin{tikzpicture}[line width=\lw, myscale=0.6]
			\node at (0,0) (A){$a$}; \node at (1,0) (B) {$a+1$};\draw (A)..controls +(0.2,0.8) and +(-0.2,0.8)..(B);
		\end{tikzpicture}&\mapsto\begin{tikzpicture}[line width=\lw, myscale=0.6]
			\node at (0,0) (A){$-a$}; \node at (1,0) (B) {$-a-1$};\draw (A)..controls +(0.2,-0.8) and +(-0.2,-0.8)..(B);
		\end{tikzpicture},
	\end{align*}
	where $a,b\in\bbR$, and $\eta=-1$ if $b\neq a,a+1$ and $\eta=1$ if $b=a+1,a$.
\end{lem}
\begin{proof}
	This is straightforward bearing in mind that $f\circ_{\mathrm{op}} g = (-1)^{\abs{f}\abs{g}}g\circ f$.
\end{proof}

\section{Cyclotomic quotients \texorpdfstring{$\sRcyc$}{}}
The goal of this section is to prove the \emph{Isomorphism Theorem} for cyclotomic quotients $\sRcyc$ of $\sR$ which is a precise formulation of \Cref{isoncvwintro} from the introduction. As an important byproduct of the proof we obtain the \emph{Basis Theorem}.
It establishes the existence of a nice basis of $\sRcyc$ which allows doing highest weight theory. 

\subsection{Definition of cyclotomic quotients}
For a general overview about cyclotomic quotients in the context of (quiver) Hecke algebras we refer to \cite{Mathascycl}. 
\begin{defi}
	Given a natural number $\level$, called the \emph{level}, we define the \emph{cyclotomic quotients} $\sRwocyc$ and $\sRcyc$, of charge $\charge=\charge(\level)$, as the quotients of $\sRwo$ and $\sR$ respectively by the right tensor ideal generated by 
	\begin{equation}\label{cycrelations}
		\begin{tikzpicture}[line width=\lw, myscale=0.6]
			\draw (0,0)--(0,1) node[midway] (a) {};
			\fill (a) circle(\dw);
			\node[anchor = west] at (a) {$n$};
			\node[fill=white, anchor=north] at (0,0) {$a$};\node[fill=white, anchor=south] at (0,1) {$a$};
		\end{tikzpicture},
		\quad \text{where}\quad n=\begin{cases} 1&\text{if $a=\delta_i$, $1\leq i\leq \level$}, \\
			0&\text{otherwise}.
		\end{cases}
	\end{equation}
\end{defi}
\begin{defi}
	The \emph{cyclotomic polynomial of level $\level$} (and charge $\charge$) is defined as
	\begin{equation}\label{defOmega}
		\Omega^\level(x)=\displaystyle\prod_{i=1}^\level(x-\delta_i).	
	\end{equation}
\end{defi}

\begin{defi}\label{defthetai}
	We denote by $\theta_i^k\colon\sRcyc\to\sRcyc$ the endofunctor given by adding a strand labelled $i$ with $k$ dots on the right.
	If $k=0$, we abbreviate $\theta_i\coloneqq\theta_i^0$.
\end{defi}
%\subsection{Cyclotomic quotients in the affine VW-supercategory \texorpdfstring{$\sVW$}{}}
Recall from the introduction the affine VW-supercategory $\sVW$.
\begin{defi}
	The \emph{level $\level$ cyclotomic quotient } $\sVWcyc$ is the cyclotomic quotient of $\sVW$ by the cyclotomic polynomial $\Omega^\level(x)$ of level $\level$ from \cref{defOmega}.
\end{defi}

Given an object, say $*^{\otimes m}$, its endomorphism algebra $\End_{\sVWcyc}(m)$ is a finite dimensional algebra and the $y_j$, (i.e.~identities with a dot on the $j$-th strand) for $1\leq j\leq m$ form a family of pairwise commuting elements.
\begin{nota}
	Denote by $e_\mathbf{i}=e_{i_1, \dots, i_m}$ the idempotents projecting onto the simultaneous generalised $i_j$-eigenspaces for the $y_j$'s, in particular $y_je_\mathbf{i}=e_\mathbf{i}y_j=i_je_\mathbf{i}$.
\end{nota}

\subsection{The Isomorphism Theorem and Cyclotomic Equivalence}	
We finally formulate the \emph{Isomorphism Theorem} from the introduction:
\begin{thm}[Isomorphism Theorem]\label{isoncvw}
	For any level $\level$, the following assignments define a fully faithful functor to the Karoubian envelope $\mathrm{Kar}(\sVWcyc)$ of $\sVWcyc$:
	\begin{flalign}	
		\Phi\colon\sRwocyc&\to\mathrm{Kar}(\sVWcyc), \quad\mathbf{i}=(i_1, \dots, i_m)\mapsto e_{\mathbf{i}},\\
		%e_{\mathbf{i}}\cup e_{\mathbf{i}'}\coloneqq
		\begin{tikzpicture}[line width=\lw, myscale=0.6]
			\draw (-1,0) --(-1,-1)(2,0)--(2,-1);
			\node at (-0.5,-0.5) {$\dots$};
			\node at (1.5,-0.5) {$\dots$};
			\node[fill=white] at (-1,0) {$i_1$};
			\node[fill=white] at (2,0) {$i_m$};
			\node[fill=white, anchor=north] at (-1,-1) {$i_1$};
			\node[fill=white, anchor=north] at (2,-1) {$i_m$};
			\node at (0,0) (A){$i_k$}; 
			\node at (1,0) (B) {$i_{k+1}$};
			\draw (A)..controls +(0.2,-0.8) and +(-0.2,-0.8)..(B);
		\end{tikzpicture}
		&\mapsto e_{\mathbf{i}}\flat^\ast_k e_{\mathbf{i}'},\quad\quad\quad\quad
		%&\mapsto e_{i_1, \dots, \hat{i}_{k}, \hat{i}_{k+1}, \dots, i_m}\flat^\ast_k e_{i_1, \dots, i_m},\\
		%d_ke_{{\mathbf{i}}}\coloneqq
		\begin{tikzpicture}[line width=\lw, myscale=0.6]
			\draw (-1,0) --(-1,1)(2,0)--(2,1);
			\node at (-0.5,0.5) {$\dots$};
			\node at (1.5,0.5) {$\dots$};
			\node[fill=white, anchor=north] at (-1,0) {$i_1$};
			\node[fill=white, anchor=north] at (2,0) {$i_m$};
			\node[fill=white, anchor=south] at (-1,1) {$i_1$};
			\node[fill=white, anchor=south] at (2,1) {$i_m$};
			\draw (0.5,0)--(0.5,1) node[midway] (a) {};
			\fill (a) circle (\dw);
			\node[fill=white, anchor=north] at (0.5,0) {$i_k$};
			\node[fill=white, anchor=south] at (0.5,1) {$i_k$};
		\end{tikzpicture}\mapsto e_{{\mathbf{i}}}(y_k-i_k),\nonumber\\
		%e_{\mathbf{i}'}\cup e_{\mathbf{i}}\coloneqq
		\begin{tikzpicture}[line width=\lw, myscale=0.6]
			\draw (-1,0) --(-1,1)(2,0)--(2,1);
			\node at (-0.5,0.5) {$\dots$};
			\node at (1.5,0.5) {$\dots$};
			\node[fill=white] at (-1,0) {$i_1$};
			\node[fill=white] at (2,0) {$i_m$};
			\node[fill=white, anchor=south] at (-1,1) {$i_1$};
			\node[fill=white, anchor=south] at (2,1) {$i_m$};
			\node at (0,0) (A){$i_k$}; \node at (1,0) (B) {$i_{k+1}$};\draw (A)..controls +(0.2,0.8) and +(-0.2,0.8)..(B);
		\end{tikzpicture}
		&\mapsto e_{{\mathbf{i}'}}\flat_k e_{{\mathbf{i}}},\label{caps}\\
		%&\mapsto e_{i_1, \dots, i_m}\flat_k e_{i_1, \dots, \hat{i}_{k}, \hat{i}_{k+1}, \dots, i_m},\\
		\begin{tikzpicture}[line width=\lw, myscale=0.6]
			\draw (-1,0) --(-1,1)(2,0)--(2,1);
			\node at (-0.5,0.5) {$\dots$};
			\node at (1.5,0.5) {$\dots$};
			\node[fill=white, anchor=north] at (-1,0) {$i_1$};
			\node[fill=white, anchor=north] at (2,0) {$i_m$};
			\node[fill=white, anchor=south] at (-1,1) {$i_1$};
			\node[fill=white, anchor=south] at (2,1) {$i_m$};
			\draw (0,0) -- (1,1) (1,0)--(0,1);
			\node[fill=white, anchor=north] at (0,0) {$i_k$}; 
			\node[fill=white, anchor=north] at (1,0) {$i_{k+1}$};
			\node[fill=white, anchor=south] at (0,1) {$i_{k+1}$}; 
			\node[fill=white, anchor=south] at (1,1) {$i_k$};
		\end{tikzpicture}&\mapsto\begin{cases} e_{{\mathbf{i}}}\eta_{i_{k+1}, i_k}((i_{k+1}-i_k)s_k + 1)&\text{if $i_{k+1}\notin\{i_k, i_k+1\}$},\\0&\text{otherwise,}\end{cases}\label{thebigcross}
	\end{flalign}
	where $\mathbf{i}'=(i_1, \dots, \hat{i}_{k}, \hat{i}_{k+1}, \dots, i_m)$ and $\eta_{b,a}$ is any choice of scalars, such that
	\begin{enumerate}
		\item\label{scalarone} $\eta_{a,b}\eta_{b,a}=\frac{1}{1-(a-b)^2}$ for all $a$, $b\in\bbR$ such that $a-b\notin\{0, \pm1\}$ and
		\item\label{scalartwo} $\eta_{b,a}(b-a)=\eta_{a,b+1}(a-b-1)$ for all $a$, $b\in\bbR$ such that $a\neq b, b+1$.
	\end{enumerate}
\end{thm}
As a consequence of the Isomorphism Theorem in the special case of $\level=1$ and $\charge=0$ we obtain an idempotent version of the \emph{periplectic Brauer algebras}, \cite{Coulembier}.  

For the proof we will introduce elements $\Psi_{\ta{t}}^{\ta{s}}\in\sRcyc$ and show a Basis Theorem. 
\begin{rem}
	Note that the functor is not an equivalence, but it will become an equivalence after additive completion by the \nameref{cycequiv} below. 
\end{rem}

Recall that the Karoubian closure of a category is the idempotent completion of a category. We could also take its additive envelope (which means we allow also finite direct sums of objects and morphisms). In general, taking additive closure and taking Karoubian closure does  not commute, but since we have finite dimensional morphism spaces these procedures in fact  do commute.

\begin{thm}[Cyclotomic equivalence]\label{cycequiv}
	For any level $\level$, the additive closure of $\sRwocyc$ is equivalent as $\sVec$-category to the additive closure of  $\mathrm{Kar}(\sVWcyc)$ of $\sVWcyc$.
\end{thm}
\begin{rem}
	We expect that the Isomorphism Theorem holds for any (not necessarily generic) charge sequence, but our formulation and proof of the Basis Theorem requires the charge to be generic.
\end{rem}
\begin{rem}
	As a consequence of the Isomorphism Theorem we obtain in particular an idempotent version of cyclotomic quotients of the  \emph{periplectic Brauer categories} from \cite{CEIII} or the marked Brauer categories from \cite{KT}. 
\end{rem}
\subsection{The Basis Theorem and applications}
In this section we formulate the Basis Theorem and show some important consequences.

Let $\ta{t}=(\ta{t}_0, \ta{t}_1, \dots, \ta{t}_m) \in\Tud$ with $\shape(\ta{t})=\lambda$.
We start by defining morphisms 
\begin{equation}\label{eqdefPsi}
	\Psi_{\ta{t}}^{\ta{t}^\lambda}\colon \bm{i}_{\ta{t}}\to \bm{i}_{\ta{t}^\lambda\quad}\text{ and }\quad \Psi_{\ta{t}^\lambda}^{\ta{t}}\colon \bm{i}^{\circledast}_{\ta{t}^\lambda}\to\bm{i}^{\circledast}_{\ta{t}} \quad\text{ in } \sRcyc.
\end{equation}
\subsubsection*{Construction of $\Psi^{\ta{t}^\lambda}_{\ta{t}}$ and $\Psi_{\ta{t}^\lambda}^{\ta{t}}$} \hfill\\
\emph{Case 1}: If $|\lambda|=m$ then boxes were only added in $\ta{t}$ and $\bm{i}_{\ta{t}}$ differs from $\bm{i}_{\ta{t}^\lambda}$ by a permutation.
Let $d_{\ta{t}}\in \mathfrak{S}_n$ be the unique such permutation of minimal length.
Pick a reduced expression $d_{\ta{t}}=s_{r_\ell}\cdots s_{r_1}$.
This defines a corresponding composition $\bm{i}_{\ta{t}}\to \bm{i}_{\ta{t}^\lambda}$ of $\ell$ morphisms of the form \cref{thebigcross} (where each simple transposition $s_{k}$ is sent to a diagram where the $k$th and $k+1$th strand cross).
Note that, by construction and by assumption on the charge, the labels, say $a$ and $b$, at these two strands are distant in the sense that they satisfy $a\not\in\{b,a+1,a-1\}$.
But then \cref{ncbraid}, \cref{ncinverse} imply that the construction is independent of the choice of reduced expression.
Thus, we get a well-defined morphism $\Psi_{\ta{t}}^{\ta{t}^\lambda}\colon \bm{i}_{\ta{t}}\to \bm{i}_{\ta{t}^\lambda}$.
Analogously define $\Psi_{\ta{t}^\lambda}^{\ta{t}}\colon \bm{i}^{\circledast}_{\ta{t}^\lambda}\to\bm{i}^{\circledast}_{\ta{t}}$ using the dual residue sequences.
In both constructions $\Psi_{\ta{t}^\lambda}^{\ta{t}^\lambda}$ is the identity on $\bm{i}_{\ta{t}^\lambda}$.

\emph{Case 2}: If $|\lambda|<m$ then consider the minimal $r$ such that $\ta{t}_r$ is obtained from $\ta{t}_{r-1}$ by removing a box.
Denote by $l < k$ the index where this box was added to $\ta{t}$.
Draw a cap from $i_l$ to $i_k$ in $\bm{i}_{\ta{t}}$.
By adding vertical strands at the remaining residues we obtain a diagram representing a morphism from $\bm{i}_{\ta{t}}$ to the subsequence of $\bm{i}_{\ta{t}}$ given by the residues not involved in the cap.
(We leave it to the reader to verify using \cref{nctangleone} that the diagram can be written as a product of elements 
of the forms \cref{thebigcross}, \cref{caps}, and that any such product defines up to sign the same morphism.) 

Repeat this procedure for all boxes that were removed in $\ta{t}$ working with the residue sequence treated by caps already.
This results in a composite morphism from $\bm{i}_{\ta{t}}$ to the subsequence $\bm{i}_{\ta{t}}'$ of $\bm{i}_{\ta{t}}$ where all residues connected with cups are removed.
The length of $\bm{i}_{\ta{t}}'$ equals $|\lambda|$, and we can construct, as in Case 1), a morphism $\bm{i}_{\ta{t}}'\to\bm{i}_{\ta{t}^\lambda}$.
Composing provides a morphism $\bm{i}_{\ta{t}}\to\bm{i}_{\ta{t}^\lambda}$ which is up to an overall sign independent of choices on the way.	

Similarly, we can construct a morphism $\bm{i}_{\ta{t}^\lambda}^\circledast\to \bm{i}_{\ta{t}}^\circledast$ by using cups instead of caps.

The constructed morphisms are only unique up to signs, since caps and cups have odd degree and thus height moves create signs.
To fix this we adjust our construction by height moves so that they satisfy the following \emph{height requirement}: We assume that if two caps (or cups) connect the positions $(k,l)$ and $(k',l')$ with $l<l'$, then $(k,l)$ is lower (resp.~higher) than $(k',l')$.
With this we constructed the desired morphisms \cref{eqdefPsi} in $\sRcyc$.
Recalling that $\bm{i}_{\ta{t}^\lambda}=\bm{i}_{\ta{t}^\lambda}^\circledast$ we can define the following compositions: 

\begin{defi} \label{DefPsi} For $\ta{t}$, $\ta{s}\in\Tud$ with $\shape(\ta{t})=\lambda=\shape(\ta{s})$ define $\Psi_{\ta{t}}^{\ta{s}}=\Psi_{\ta{t}^\lambda}^{\ta{s}}\Psi_{\ta{t}}^{\ta{t}^\lambda}\in\sRcyc$.
	In particular, $\Psi_{\ta{t}}^{\ta{t}}$ is the identity on $\bm{i}_{\ta{t}^\lambda}$.
\end{defi}	

\begin{thm}[Basis Theorem]\label{spanningsetisbasis}
	The set $\cB\coloneqq\{\Psi_{\ta{t}}^{\ta{s}}\mid\ta{t},\ta{s}\in\Tud,  \shape(\ta{t})=\shape(\ta{s})\}$ is a basis, the \emph{updown-tableaux basis}, of $\sRcyc$.
\end{thm}

Before the proof we show some nice properties of the basis elements.
%The $\Psi_{\ta{t}}^{\ta{s}}$ behave well under the translation functors $\theta_i$ and factorize : 
\begin{prop}\label{thetaiaddablebox}
	Let $\ta{t}, \ta{s}\in\Tud(\lambda)$ and $\lambda\xrightarrow{\bbox}\mu$, $\res(\bbox)=i$.
	Then $\theta_i(\Psi_{\ta{t}}^{\ta{s}}) = \Psi_{\ta{t}\smallfrown\mu}^{\ta{s}\smallfrown\mu}$, where $\ta{u}\smallfrown\mu=(\ta{u}_0, \dots, \ta{u}_n, \mu)\in\Tud(\mu)$ for $\ta{u}=(\ta{u}_0, \dots, \ta{u}_n)\in\Tud(\lambda)$.
\end{prop}
\begin{proof}
	Assume first that $\ta{s}=\ta{t}^\lambda$.
	By definition, $\Psi\coloneqq\Psi_{\ta{t}\smallfrown\mu}^{\ta{t}^\lambda\smallfrown\mu}=\Psi_{\ta{t}^\mu}^{\ta{t}^\lambda\smallfrown\mu}\Psi_{{\ta{t}\smallfrown\mu}}^{\ta{t}^\mu}$.
	By assumption, $\bm{i}_{\ta{t}^\lambda\smallfrown\mu}$ and $\bm{i}_{\ta{t}^\mu}$ are obtained from $\bm{i}_\ta{t}$ by adding $i$ at the end respectively at the position, say $p$, corresponding to $\bbox$.
	In a diagram describing $\Psi$, this last entry in $\bm{i}_\ta{t}^\lambda\smallfrown\mu$ connects (via the right factor of $\Psi$) to the residue at position $p$ and then (via the left factor) back to the last entry in $\bm{i}_{\ta{t}^\lambda\smallfrown\mu}$.
	Since the involved crossings carry distant labels, one can straighten this strand using \cref{ncinverse} and \cref{ncbraid} to obtain $\Psi_{\ta{t}}^{\ta{s}}$ with an additional vertical strand labelled $i$ on the right.
	Thus, $\Psi=\theta_i(\Psi_{\ta{t}}^{\ta{s}})$.
	Similarly, the claim holds for $\Psi^{\ta{t}}_{\ta{t}^\lambda}$ and thus for $\Psi_{\ta{t}}^{\ta{s}}=\Psi_{\ta{t}^\lambda}^{\ta{s}}\Psi_{\ta{t}}^{\ta{t}^\lambda}$, since $\theta_i$ is a functor.
\end{proof}

For the next application we consider $\Par^\level$ for fixed level $\level$ with its partial order from \cref{parorder} as subset of $I\coloneqq\bigcup_{m\in\bbZ_{\geq0}}\bbR^m$ by identifying $\lambda\in\Par^\level$ with $\bm{i}_{\ta{t}^\lambda}=\bm{i}^\circledast_{\ta{t}^\lambda}$.

\begin{thm}[Highest weight]\label{sRcycisquher}
	Consider $\sRcyc$ with updown-tableaux basis $\cB$.
	For $\lambda\in\Par^\level$ and $\bm{i}\in\bbR^m$ set $Y(\bm{i},\lambda) = \{\Psi_{\ta{t}^\lambda}^{\ta{s}}\mid\bm{i}^\circledast_{\ta{s}}=\bm{i}\}$ and $X(\lambda, \bm{i}) = \{\Psi_{\ta{s}}^{\ta{t}^\lambda}\mid\bm{i}_{\ta{s}}=\bm{i}\}$.
	
	This data endows $A\coloneqq\bigoplus_{m,n\in\mathbb{N}_0}\bigoplus_{\mathbf{i}\in\bbR^m,\mathbf{j}\in\bbR^n}\Hom_{\sRcyc}(\mathbf{i},\mathbf{j})$ with the structure of an upper finite based quasi-hereditary (super-)algebra in the sense of \cite{BS21}.
\end{thm}
\begin{proof}
	Writing $Y(\lambda)\coloneqq \bigcup_{\bm{i}\in I}Y(\bm{i},\lambda)$ and $X(\lambda)\coloneqq\bigcup_{\bm{i}\in I}X(\lambda, \bm{i})$, it follows by \cref{DefPsi} directly from \cref{spanningsetisbasis} that $\bigcup_{\lambda\in\Par^\level}Y(\lambda)\times X(\lambda)$ is a basis of $A$.
	The set $Y(\mu, \lambda)$ can only be nonempty if $\lambda=\mu$ or $\abs{\mu}>\abs{\lambda}$, and thus $\mu\leq\lambda$, similarly for $X(\lambda, \mu)$.
	It is also clear from \cref{DefPsi} that $X(\lambda, \lambda) = Y(\lambda,\lambda) = \{e_\lambda\}$ for each $\lambda\in\Par^\level$.
\end{proof}	

\subsection{The spanning set \texorpdfstring{$\cB$}{}}
We next show that the proposed basis $\cB$ in \cref{spanningsetisbasis} spans.
For this fix the filtration $\{0\}=F_{\leq -1}\subseteq F_{\leq 0}\subseteq F_{\leq 1}\subseteq\dots$ on $\Tud$ given by $F_{\leq b} = \bigcup_{\abs{\lambda}\leq b}\Tud(\lambda)$.
This induces a filtration on the $\bbk$-span $B$ of $\cB$ with pieces $B_{\leq i}$ spanned by all $ \Psi_\ta{t}^\ta{s}$ with $ \shape(\ta{t})=\shape(\ta{s})\in F_{\leq i}$.
Let $R\supseteq \Bts{\leq b}$ be the two-sided ideals in ${\sRcyc}$ generated by $\cB$ respectively $B_{\leq b}$.
Thus, $\Bts{\leq b}$ defines a filtration on $\Bts{}$ by ideals which we use to show $B=R$.
Abbreviate $\Bl{< b}\coloneqq\Bl{\leq (b-1)}$, $\Bts{<b}\coloneqq\Bts{\leq (b-1)}$.

We show now some properties of $\sRcyc$ in the following situation for fixed $b\in\mathbb{N}$:
\begin{equation}\tag{Ass${\scriptstyle{< b}}$}\label{Ass}
	\text{$\Bl{\leq b'}=\Bts{\leq b'}$ for all $b'< b$.}
\end{equation} 

\begin{prop}\label{thetainotaddablebox}
	Assume \cref{Ass} and let $\lambda\in\Par^\level$ with $\abs{\lambda}=b$.
	Then the following holds in $\sRcyc$ for any $i,j\in\bbR$ with $\Add_i(\lambda)=\emptyset$.
	
	\begin{minipage}{3.5cm}
		\begin{align}\label{tea1}
			&\theta_i(\Psi_{\ta{t}^\lambda}^{\ta{t}^\lambda})\in \Bl{<\abs{\lambda}},
		\end{align}
	\end{minipage}
	\begin{minipage}{3cm}
		\begin{align}\label{tea2}
			&\theta_j^1(\Psi_{\ta{t}^\lambda}^{\ta{t}^\lambda})=0,
		\end{align}
	\end{minipage}
	\begin{minipage}{6cm}
		\begin{align}\label{tea3}
			&\theta_i(\Psi_{\ta{t}}^{\ta{s}})\in \Bl{<\abs{\lambda}} \text{ for }\ta{t}, \ta{s}\in\Tud(\lambda).
		\end{align}
	\end{minipage}
	In particular, any diagram with a dot is zero in $\sRcyc$ by \cref{tea2}.
\end{prop}
The proof of \cref{thetainotaddablebox} will show inductively the following refinements: 
\begin{cor}\label{corobj}
	Assume \cref{Ass}.
	Then the following holds in $\sRcyc$.
	\begin{enumerate}[label=$(\alph*)$]
		\item Any object $\bm{i}$ such that $\id_{\bm{i}}\in \Bts{\leq b+1}$ which has a subsequence of the form $(a,a)$ is zero.\label{corobj1}
		\item\label{corobj2} For any $\bm{i}=(i_1, \dots, i_r)\in\Bts{<b-1}$
		\begin{equation*}
			\begin{tikzpicture}[line width=\lw, myscale=0.6]
				\node at (-1,0) (A) {$i_{r}$};
				\node at (-1,1.5) (B) {$i_{r}$};
				\node at (-2.5,0) (G) {$i_1$};
				\node at (-2.5,1.5) (H) {$i_1$};
				\node at (0,0) (C) {$i$};\node at (1,0) (D) {$i+1$};
				\node at (0,1.5) (E) {$i+1$};\node at (1,1.5) (F) {$i$};
				\draw (D.north)--(E.south) (C.north)--(F.south) (A.north)--(B.south) (G.north)--(H.south);
				\node at (-1.75, 0.75) {$\dots$};
			\end{tikzpicture}=0.
		\end{equation*} 
		\item\label{corobj3} Let $\lambda\in\Par^\level$ with $\abs{\lambda}\leq b$ and assume $\Add_i(\lambda)=\emptyset$.
		If $\theta_i(\Psi_{\ta{t}^\lambda}^{\ta{t}^\lambda})\neq0$ then there exists a subsequence of the form $(i,i\pm1,i)$ in $\res(\ta{t}^\lambda i)$. 
	\end{enumerate}
\end{cor}
\begin{rem}\label{rmsub}
	In \cref{corobj} the subsequence can in fact be chosen to involve the $i$ at the end of $\res(\ta{t}^\lambda i)$.
	The statement holds even for any $\ta{t}\in\Tud(\la)$.
\end{rem}		

\begin{proof} [Proof of \cref{thetainotaddablebox} (with \cref{corobj} and  \cref{rmsub})]
	
	The assumption and \cref{tea1} directly imply \cref{tea3}.
	We prove \cref{tea1} and \cref{tea2} parallel via induction on $b\coloneqq\abs{\lambda}$. 
	Let $(i_1, \dots, i_b)\coloneqq\bm{i}_{\ta{t}^\lambda}=\bm{i}^\circledast_{\ta{t}^\lambda}$, thus $\Psi_{\ta{t}^\lambda}^{\ta{t}^\lambda}=\id_{(i_1, \dots, i_b)}$.
	
	If $b=0$, then $\Add_i(\lambda)=\emptyset$ implies $i\not=\charge_j$ for all $j$ and both, \cref{tea1} and \cref{tea2}, follow from \cref{cycrelations}.
	
	Assume the claims hold for all $b'<b$.
	Via induction and \cref{thetaiaddablebox}, $\restr{\theta_i}{\Bl{\leq b-1}}$ is a filtered map of degree $1$.
	
	We consider four different cases.
	\begin{enumerate}
		\item\label{same} If $i=i_b$ then we have, by \cref{ncdotcrossing} and by \cref{ncdotcrossing} with \cref{ncinverse},
		\begin{equation*}
			\begin{tikzpicture}[line width=\lw, myscale=0.6]
				\node at (0,0) (A) {$i_b$};
				\node at (1,0) (B) {$i$};
				\node at (0,1.5) (C) {$i_b$};
				\node at (1,1.5) (D) {$i$};
				\draw (A.north)--(C.south) (B.north)--(D.south);
			\end{tikzpicture}=
			\begin{tikzpicture}[line width=\lw, myscale=0.6]
				\node at (0,0) (A) {$i$};
				\node at (1,0) (B) {$i$};
				\node at (0,1.5) (C) {$i$};
				\node at (1,1.5) (D) {$i$};
				\draw (A.north)--(C.south) (B.north)--(D.south);
			\end{tikzpicture}=
			\begin{tikzpicture}[line width=\lw, myscale=0.6]
				\node at (0,0) (A) {$i$};
				\node at (1,0) (B) {$i$};
				\node at (0,1.5) (C) {$i$};
				\node at (1,1.5) (D) {$i$};
				\draw (A.north)--(D.south) (B.north)--(C.south) node[pos=0.1] (x) {};
				\fill (x) circle (\dw);
			\end{tikzpicture}-\begin{tikzpicture}[line width=\lw, myscale=0.6]
				\node at (0,0) (A) {$i$};
				\node at (1,0) (B) {$i$};
				\node at (0,1.5) (C) {$i$};
				\node at (1,1.5) (D) {$i$};
				\draw (A.north)--(D.south) (B.north)--(C.south) node[pos=0.9] (x) {};
				\fill (x) circle (\dw);
			\end{tikzpicture}= -\begin{tikzpicture}[line width=\lw, myscale=0.6]
				\node at (0,0) (A) {$i$};
				\node at (1,0) (B) {$i$};
				\node at (0,2.5) (C) {$i$};
				\node at (1,2.5) (D) {$i$};
				\draw (A.north)--(1,1.25) --(C.south) (B.north)--node [pos=0.1] (x) {} (0,1.25)-- (D.south);
				\fill (x) circle (\dw) (0,1.25) circle(\dw);
			\end{tikzpicture}-\begin{tikzpicture}[line width=\lw, myscale=0.6]
				\node at (0,0) (A) {$i$};
				\node at (1,0) (B) {$i$};
				\node at (0,1.5) (C) {$i$};
				\node at (1,1.5) (D) {$i$};
				\draw (A.north)--(D.south) (B.north)--(C.south) node[pos=0.9] (x) {};
				\fill (x) circle (\dw);
			\end{tikzpicture}.
		\end{equation*}
		Therefore, $\theta_i^n(\Psi_{\ta{t}^\lambda}^{\ta{t}^\lambda})=0$ for $n\in\{0,1\}$ by induction.
		This also shows \cref{corobj}\cref{corobj1,corobj3}  in this case.
		\item If $i\not=i_b$, $\abs{i-i_b}\neq1$, then $\Add_i(\ta{t}^\lambda_{b-1})=\Add_i(\lambda)$ and \cref{ncinverse}, \cref{ncdotcrossing} give
		\begin{equation*}
			\theta_i^{n}(\Psi_{\ta{t}^\lambda}^{\ta{t}^\lambda})\quad=
			\begin{tikzpicture}[line width = \lw, myscale=0.6, yscale=0.8]
				\node at (1,0) (A) {$i_1$};
				\node at (1,3) (AA) {$i_1$};
				\node at (3,0) (B){$i_{b-1}$};
				\node at (3,3) (BB){$i_{b-1}$};
				\node at (4,0) (C){$i_b$};
				\node at (4,3) (CC){$i_b$};
				\node at (5,0) (D){$i$};
				\node at (5,3) (DD){$i$};
				\node at (4.7,1.5) (x){$n$};
				\node at (2,1.5) {$\dots$};
				\draw (A.north)--(AA.south) (B.north)--(BB.south) (C.north)--(CC.south)(D.north)--(DD.south);
				\fill (x) (5,1.5) circle(\dw);
			\end{tikzpicture}
			\quad=
			\begin{tikzpicture}[line width = \lw, myscale=0.6, yscale=0.8]
				\node at (1,1.5) (A) {$i_1$};
				\node at (3,1.5) (B){$i_{b-1}$};
				\node at (4,1.5) (C){$i$};
				\node at (5,1.5) (D){$i_b$};
				\node at (2,1.5) {$\dots$};
				\draw (1,0)-- (A) -- (1,3) (3,0) --(B) -- (3,3) (4,0) -- (D.south) (D.north) -- (4,3) (5,0) --node[pos=0.9] (x) {} (C.south) (C.north) -- (5,3);
				\fill (x) circle (\dw);
				\node[anchor=east] at (x) {$n$};
			\end{tikzpicture}.
		\end{equation*}
		By induction, $\theta_i^n(\id_{(i_1, \dots, i_{b-1})})\in \Bl{<b-1}$.
		Thus, $\id_{(i_1, \dots, i_{b-1}, i, i_b)}\in \Bl{<b}$ as $\restr{\theta_i}{\Bl{\leq b-1}}$ is filtered of degree $1$.
		\Cref{corobj}\cref{corobj1,corobj3} follow also immediately in this case. \cref{rmsub} holds, since  $\Add_i(\ta{t}^\lambda_{b-1})=\Add_i(\lambda)$
		
		\item Suppose that $i_b=i+1$.
		By definition of $\ta{t}^\lambda$, $i_b$ is the residue of the last box in the last row of $\lambda$.
		As there is no addable box with residue $i$, the last row of $\lambda$ has more than one box and then $i_{b-1}=i$.
		Thus, $(i_{b-1},i_{b})=(i,i+1)$.
		This shows \cref{corobj} \cref{corobj3} and \cref{rmsub} in this case.

		Define multi-up-down-tableaux $\ta{u}$ and $\ta{v}$ of length $b+1$ such that\footnote{(The first $b-1$ steps in $\ta{u}$ and $\ta{t}$ agree; then we remove and add the box that was added in step $b-1$. The first $b$ steps in $\ta{v}$ agree with $\ta{t}$; then we remove the box that was added in step $b$ for $\ta{t}$).} $\ta{u}_k=\ta{t}^\lambda_k$ for $k<b$, $\ta{u}_{b+1}=\ta{u}_{b-1}$ and $\ta{u}_b=\ta{u}_{b-2}$ respectively $\ta{v}_k=\ta{t}^\lambda_k$ for $k\leq b$ and $\ta{v}_{b+1}=\ta{v}_{b-1}$.
		By construction, we have $\res(\ta{u})=(i_1, \dots, i_b, i)=\res^\circledast(\ta{v})$.
		Now we can compute (modulo some sign $\pm$ which we do not specify)
		\begin{equation}\label{replacement}
			\Psi_{\ta{u}}^{\ta{v}}=
			\begin{tikzpicture}[line width = \lw, myscale=0.6]
				\node at (0,0) (A) {$i_1$};
				\node at (2,0) (B) {$i_{b-2}$};
				\node at (3,0) (C) {$i$};
				\node at (4,0) (D) {$i+1$};
				\node at (5,0) (E) {$i$};
				\node at (0,1.5) (F) {$i_1$};
				\node at (2,1.5) (G) {$i_{b-2}$};
				\node at (3,1.5) (H) {$i$};
				\node at (4,1.5) (I) {$i+1$};
				\node at (5,1.5) (J) {$i$};
				\node at (1,0.75) {$\dots$};
				\draw (A) -- (F) (B)--(G) (C)..controls +(0.2,0.8) and +(-0.2,0.8)..(D) (E)--(H) (I)..controls +(0.2,-0.8) and +(-0.2,-0.8)..(J);
			\end{tikzpicture}
			=\pm
			\begin{tikzpicture}[line width = \lw, myscale=0.6]
				\node at (0,0) (A) {$i_1$};
				\node at (2,0) (B) {$i_{b-2}$};
				\node at (3,0) (C) {$i$};
				\node at (4,0) (D) {$i+1$};
				\node at (5,0) (E) {$i$};
				\node at (0,1.5) (F) {$i_1$};
				\node at (2,1.5) (G) {$i_{b-2}$};
				\node at (3,1.5) (H) {$i$};
				\node at (4,1.5) (I) {$i+1$};
				\node at (5,1.5) (J) {$i$};
				\node at (1,0.75) {$\dots$};
				\draw (A) -- (F) (B)--(G) (C)--(H) (D)--(I) (E)--(J);
			\end{tikzpicture}=\pm\theta_i(\Psi_{\ta{t}^\lambda}^{\ta{t}^\lambda}).
		\end{equation}
		The first and last equalities here hold by definition, and the second equality used \cref{ncbraid}.
		The reader might expect two more summands from this relation, but we proved in \cref{same} that $\restr{\theta_i\circ\theta_i}{\Rl{\leq b}}=0$ and thus these terms vanish.
		We see that the number of propagating strands in $\Psi_{\ta{u}}^{\ta{v}}$ is exactly one less than the one in $\Psi_{\ta{t}^\lambda}^{\ta{t}^\lambda}$.
		Thus, $\theta_i(\Psi_{\ta{t}^\lambda}^{\ta{t}^\lambda})\in \Rl{<b}$ and \cref{tea1} holds.
		
		We also need to show \cref{tea2}.
		By \cref{ncinverse} and the induction hypothesis it suffices to show \cref{corobj}\cref{corobj2}, i.e.
		\begin{equation*}
			\begin{tikzpicture}[line width=\lw, myscale=0.6]
				\node at (-1,0) (A) {$i_{b-1}$};
				\node at (-1,1.5) (B) {$i_{b-1}$};
				\node at (-2.5,0) (G) {$i_1$};
				\node at (-2.5,1.5) (H) {$i_1$};
				\node at (0,0) (C) {$i$};\node at (1,0) (D) {$i+1$};
				\node at (0,1.5) (E) {$i+1$};\node at (1,1.5) (F) {$i$};
				\draw (D.north)--(E.south) (C.north)--(F.south) (A.north)--(B.south) (G.north)--(H.south);
				\node at (-1.75, 0.75) {$\dots$};
			\end{tikzpicture}=0.
		\end{equation*}
		Observe that $\Add_{i}(\ta{t}^\lambda_{b-1})=\emptyset$ as $\Add_{i+1}(\ta{t}^\lambda_{b-1})=\Add_{i_b}(\ta{t}^\lambda_{b-1})\neq\emptyset$.
		Thus, by induction, and from the arguments given so far, we see that either $(i_1, \dots, i_{b-1}, i, i+1)=0$ (in which case we are done) or we find a subsequence of the form $(i,i-1,i)$.
		For this subsequence we can apply the argument as in \cref{replacement} and obtain
		\begin{equation*}
			\begin{tikzpicture}[line width=\lw, myscale=0.6]
				\node at (0,0) (A) {$i$};
				\node at (1,0) (B) {$i-1$};
				\node at (2,0) (C) {$i$};
				\node at (3,0) (D) {$i+1$};
				\node at (0,1.5) (E) {$i$};
				\node at (1,1.5) (F) {$i-1$};
				\node at (2,1.5) (G) {$i$};
				\node at (3,1.5) (H) {$i+1$};
				\node at (0,3) (I) {$i$};
				\node at (1,3) (J) {$i-1$};
				\node at (2,3) (K) {$i$};
				\node at (3,3) (L) {$i+1$};
				\draw (A) -- (G) (G.north)--(L.south) (B)..controls +(0.2,0.8) and +(-0.2,0.8)..(C) (D.north)--(H.south) (H.north)--(K.south) (I.south)--(E.north) (E)..controls +(0.2,-0.8) and +(-0.2,-0.8)..(F) (F.north)--(J.south);
			\end{tikzpicture}.
		\end{equation*}
		If we apply a height move to the cup and the crossing, the statement follows from \cref{corobj}\cref{corobj2} (for a shorter sequence).
		\item Suppose that $i_b=i-1$.
		By definition of $\ta{t}^\lambda$, $i_b$ is the residue of the last box in the last row of $\lambda$.
		As there is no addable box with residue $i$, the  second last row has a box with residue $i$ but not with higher residues. (Note moreover that there are at least two rows). Now this case is similar to (iii), but one also has to use \cref{ncinverse} to move a value $i$ to the position $b-1$.
		Here, for the proof of \cref{tea1}, a subsequence $(i,i-1,i)$ is obtained.
	\end{enumerate}
	This shows  \cref{tea1,tea2} and hence also  \cref{thetainotaddablebox} and \cref{rmsub}.
\end{proof}

\begin{cor}\label{thetaipreservesspanningset}
	Assume \cref{Ass}.
	Then $\theta_i(\Bl{\leq b})\subseteq \Bl{\leq b+1}$ and $\theta_i(\Bl{< b})\subseteq \Bl{< b+1}$ for all $i\in\bbR$, $b\in\mathbb{N}_0$.
	In particular $\theta_i$ is a filtered map of degree $1$.
\end{cor}	
\begin{proof} This follows directly from \cref{thetainotaddablebox} using \cref{thetaiaddablebox}.
\end{proof}
\begin{cor}\label{identitiesinspanningset}
	Assume \cref{Ass}.
	Then $\id_{\bm{i}}\in \Bl{\leq b}$ for any object $\bm{i}=(i_1, \dots, i_b)$.
\end{cor}
\begin{proof}
	Since $\id_{i_1}\in \Rl{\leq 1}$, this follows directly from \cref{thetaipreservesspanningset}
\end{proof}	

We next want to show that $\Bl{\leq b}=\Bts{\leq b}$ for all $b$.

We fix more notation for the rest of this subsection:

\begin{nota}\label{vorAss}		
	Consider multi-up-down-tableaux $\ta{t}$ and $\ta{s}$ of shape $\lambda$.
	We define $b\coloneqq \abs{\lambda}$ so that $\Psi_{\ta{t}}^{\ta{s}}\in \Bl{\leq b}$.
	Let $(i_1, \dots, i_m)\coloneqq\bm{i}_{\ta{t}}$ and $(i_1^\circledast, \dots, i_n^\circledast)\coloneqq\bm{i}^\circledast_{\ta{s}}$.
\end{nota}

We formulate more properties for $\sRcyc$ in the following situation:

\begin{cor}\label{closedunderdots}
	Assume \cref{Ass}.
	Then $d_k\Psi_{\ta{t}}^{\ta{s}}=0=\Psi_{\ta{t}}^{\ta{s}}d_k$ for $1\leq k\leq n$, where 
	\[d_{k}=\begin{tikzpicture}[line width = \lw, myscale=0.6]
		\node at (0,1) (A){$i_1^\circledast$};
		\node at (1,1) (E){};%$i_{k-1}^\circledast$};
		\node at (3,1) (B){$i_n^\circledast$};
		\node at (2,1) (F){};%$i_{k+1}^\circledast$};
		\node at (1.5,1) (c){$i_k^\circledast$};
		\node at (0.5,0.5){$\dots$};
		\node at (2.5,0.5){$\dots$};
		\draw (A)-- (0,0) (B) -- (3,0)(E) -- (1,0)(F) -- (2,0) (c)--(1.5,0) node[midway] (d) {};
		\fill (d) circle (\dw);
	\end{tikzpicture}.\]
\end{cor}
\begin{proof}
	This follows directly from \cref{thetainotaddablebox}.
\end{proof}
\begin{prop}\label{closedunderrest}	
	Assume \cref{Ass} and define for $1\leq k\leq n$, $i\in\bbR$ the morphisms 		
	\begin{equation*}
		%\begin{lem}\label{closedundercups}
		x_{k,i}=\begin{tikzpicture}[line width = \lw, myscale=0.6]
			\node at (0,1) (A){$i_1^\circledast$};
			\node at (4,1) (B){$i_n^\circledast$};
			\node at (1,1) (E){$i_k^\circledast$};
			\node at (3,1) (F){$i_{k+1}^\circledast$};
			\node at (1.5,1) (c){$i$};
			\node at (0.5,0.5){$\dots$};
			\node at (3.5,0.5){$\dots$};
			\draw (A)-- (0,0) (B) -- (4,0)(E) -- (1,0)(F) -- (3,0) (c)..controls +(0.2,-0.8) and +(-0.2,-0.8)..(2.5,1);
			\node[fill=white] at (2.25,1) (d){$i-1$};
		\end{tikzpicture},\quad
		y_{k,i}=\begin{tikzpicture}[line width = \lw, myscale=0.6]
			\draw (1.5,1)..controls +(0.2,0.8) and +(-0.2,0.8)..(2.5,1);
			\node at (0,1) (A){$i_1^\circledast$};
			\node at (1,1) (B){$i_{k-1}^\circledast$};
			\node[fill=white] at (1.6,1) (c){$i_{k}^\circledast$};
			\node[fill=white] at (2.4,1) (d){$i_{k+1}^\circledast$};
			\node at (3.2,1) (e){$i_{k+2}^\circledast$};
			\node at (4,1) (f){$i_n^\circledast$};
			\node at (0.5,1.5){$\dots$};
			\node at (3.5,1.5){$\dots$};
			\draw (A)-- (0,2) (B) -- (1,2) (e.north)+(-0.2,0)--(3,2) (4,2)--(f);
			\pgfresetboundingbox
			\useasboundingbox (0,1) rectangle (4.2,2);
		\end{tikzpicture}, \quad	
		z_{k,i}=\begin{tikzpicture}[line width = \lw, myscale=0.6]
			\node[anchor=north] at (0,1) (A){$i_1^\circledast$};
			\node[anchor=north] at (1,1) (B){$i_{k-1}^\circledast$};
			\node[anchor=north] at (1.6,1) (c){$i_{k}^\circledast$};
			\node[anchor=north] at (2.4,1) (d){$i_{k+1}^\circledast$};
			\node[anchor=north] at (3.2,1) (e){$i_{k+2}^\circledast$};
			\node[anchor=north] at (4,1) (f){$i_n^\circledast$};
			\node[anchor=north] at (0.5,1.5){$\dots$};
			\node[anchor=north] at (3.5,1.5){$\dots$};
			\draw (A)-- (0,2) (B) -- (1,2) (2.5,2)--(c.north) (1.5,2)--(2.5,1) (3,2)--(3,1) (4,2)--(f);
			\pgfresetboundingbox
			\useasboundingbox (0,1) rectangle (4.2,2);
		\end{tikzpicture}.\vspace{3mm}
	\end{equation*}
	Then $x_{k,i}\Psi_{\ta{t}}^{\ta{s}}$, $y_{k,i}\Psi_{\ta{t}}^{\ta{s}}$, $z_{k,i}\Psi_{\ta{t}}^{\ta{s}}\in \Bl{\leq b}$ holds.
	Similarly, $\Psi_{\ta{t}}^{\ta{s}}x_{k,i}$, $\Psi_{\ta{t}}^{\ta{s}}y_{k,i}$, $\Psi_{\ta{t}}^{\ta{s}}z_{k,i}\in \Bl{\leq b}$.
\end{prop}
\begin{proof}[Proof of the case $x_{k,i}$ in \cref{closedunderrest}] If $\bbox\in\Add_i(\ta{s}_k)$, let $\ta{u}\in\Tud(\lambda)$ with $\restr{\ta{u}}{k}=\restr{\ta{s}}{k}$, $\ta{u}_{k+1}=\ta{u}_k\oplus\bbox$, $\ta{u}_{j+2}=\ta{s}_{j}$ for $k\leq j\leq n$.
	Then, $x_{k,i}\Psi_{\ta{t}}^{\ta{s}}=\pm\Psi_{\ta{t}}^{\ta{u}}\in \Rl{\leq b}$.
	Otherwise, we have $\Add_i(\ta{s}_k)=\emptyset$ and may also assume that $\abs{\ta{s}_k}=k$ (as removing boxes would correspond to cups commuting with the cup of $x_{k,i}$ up to sign).
	By \cref{corobj}, the object $(i_1, \dots, i_k, i, i-1)$ is either $0$ or we find a subsequence of the form $(i,i\pm1,i)$.
	If the subsequence is $(i,i-1,i)$, applying \cref{ncbraid} gives the diagram according to removing with $a-1$ the box with residue $a$ and then adding two boxes.
	
	On the other hand if the subsequence is $(i,i+1,i)$, we get a valid multi-up-down tableau $\ta{v}$ with $\res^\circledast(\ta{v})=(i_1, \dots, i_k, i,i-1)$ where the last two entries remove the boxes corresponding to the subsequence.
	If $\ta{v}$ can be extended to a multi-up-down tableau of shape $\mu$ by $(i_{k+1}, \dots, i_m)$, then $\abs{\mu}<b$ and $x_{k,i}\Psi_{\ta{s}}^{\ta{t}}=a\Psi_{\ta{t}^\mu}^{\ta{v}}b\in \Rl{<b}$ by construction.
	If it cannot be extended, we either find a subsequence $(i-1,i-1)$ which is $0$ by \cref{corobj}\cref{corobj1} or we try to add a box of residue $i+2$.
	But this strand can be moved to the left using \cref{ncinverse}, where we then either get $0$ by \cref{corobj}\cref{corobj1} or \cref{cycrelations} or we find a subsequence $(i+2,i+3,i+2)$.
	In the last case we can apply \cref{ncbraid} and then use the same argument as above and end up eventually with $0$.
\end{proof}
\begin{proof}[Proof of the case $z_{k,i}$ in \cref{closedunderrest}] 
	Suppose first that $i_{k+1}\notin\{i_k, i_k\pm1\}$.
	Then we claim that $\ta{s}s_k$ is an up-down-tableau.
	If in steps $k$ and $k+1$ we only add respectively remove boxes, then this is clear as the boxes neither appear in the same row nor column of the same partition.
	
	In the other two cases let $i$ be the residue of the removal.
	This means that we removed a box $\bbox$ with residue $i+1$.
	So this actually swaps with all residues which are $\leq i-1$ and $\geq i+3$.
	
	The only case left to consider is, when the added box $\beta$ has residue $i+2$.
	But note that after the removal of $\bbox$, $\bbox$ is addable again.
	As $\bbox$ has residue $i+1$, no box with residue $i+2$ can be addable.
	And vice versa, if we add $\beta$ after adding $\bbox$, the box $\beta$ lies directly to the right of $\bbox$ in the same row.
	Thus, we cannot remove $\bbox$ afterwards.
	Therefore, this case cannot appear.
	
	It remains to show the statement for $i_k=i_{k+1}\pm 1$.
	If steps $k$ and $k+1$ consist out of adding boxes $\bbox$ and $\beta$, these two boxes appear in the same row respectively column.
	This means that $\beta$ cannot be added to $\ta{s}_{k-1}$.
	By \cref{thetainotaddablebox} and \cref{thetaipreservesspanningset} we know that $z_{k,i}\Psi_{\ta{t}}^{\ta{s}}\in \Rl{< b}$.
	
	Suppose that we remove a box in step $k$ and add a box in step $k+1$.
	Let further $l$ be the step in which the box was added which was removed in step $k$.
	Without loss of generality we may assume that $l=k-1$.
	Using \cref{ncinverse}, we can move the $i_l$ pass every distant entry and every neighbored entry has to be removed prior step $k$, which results in a cup that does not interact with the crossing $z$, meaning that we can swap these two as well.
	We then either have the subsequence $(i_k+1, i_k, i_k+1)$, in which case applying $z$ gives $0$ by \cref{corobj}\cref{corobj1}.
	Or we have $(i_k+1, i_k, i_k-1)$, in which case step $\ta{s}s_k$ is an up-down-tableau and $z_{k,i}\Psi_{\ta{t}}^{\ta{s}}=\Psi_{\ta{t}}^{\ta{s}s_k}$.
	
	Suppose that we add a box in step $k$ and remove a box in step $k+1$.
	Then $z_{k,i}\Psi_{\ta{t}}^{\ta{s}}=0$ by \cref{ncinverse} and \cref{closedunderdots} or \cref{ncuntwist} respectively.
	
	Suppose that we remove in both steps a box.
	
	Let $l$ denote the index of adding the box for step $k$ and $l'$ the one for $k+1$.
	Without loss of generality we may assume that $l=k-1$ and $l'=k-2$.
	Note that $l'$ has to appear before $l$ as $\ta{s}$ is an up-down-tableau.
	
	Then we either have a subsequence $(a+1,a,a-1,a)$ or $(a-1, a, a-1, a-2)$.
	In the first case, applying $z$ gives $0$ by \cref{corobj}\cref{corobj1}.
	In the second case we have the equality displayed in \cref{rain} by (locally) using \cref{ncinverse} and \cref{nctangletwo}. 
	\begin{figure}[h]
		\begin{equation}\label{rain}
			\begin{tikzpicture}[line width = \lw, myscale=0.6]
				\node at (0,0) (A) {$a-1$};
				\node at (1,0) (B) {$a$};
				\node at (2,0) (C) {$a-1$};
				\node at (3,0) (D) {$a-2$};
				\draw (A)..controls +(0.2,-1.2) and +(-0.2,-1.2)..(D) (B)..controls +(0.2,-0.8) and +(-0.2,-0.8)..(C) (A)--(0,1.5) (B)--(1,1.5) (C.north)--(3,1.5) (D.north)--(2,1.5);
			\end{tikzpicture}
			=
			\begin{tikzpicture}[line width = \lw, myscale=0.6]
				\node at (0,0) (A) {$a-1$};
				\node at (1,0) (B) {$a-2$};
				\node at (2,0) (C) {$a$};
				\node at (3,0) (D) {$a-1$};
				\draw(A)..controls +(0.2,-0.8) and +(-0.2,-0.8)..(B) (C)..controls +(0.2,-0.8) and +(-0.2,-0.8)..(D) (A)--(0,1.5) (1,1.5)--(C.north)(B.north)--(2,1.5)(3, 1.5)--(D);
			\end{tikzpicture}
		\end{equation}
		\vspace{-3mm}
	\end{figure}	
	Now removing steps $k-1$ and $k$ from $\ta{s}$ gives a valid up-down-tableau $\ta{u}$ of shape $\lambda$.
	Thus, looking at the diagrams we see that $z_{k,i}\Psi_{\ta{t}}^{\ta{s}}$ is obtained from $\Psi_{\ta{t}}^{\ta{u}}$ via left multiplication with a cup and a distant crossing.
	Now the claim follows from \cref{closedunderrest} for $x_{k,i}$ and the first paragraph about distant crossings.
\end{proof}
\begin{proof}[Proof of the case $y_{k,i}$ in \cref{closedunderrest}] 
	If in the $k$-th step of $\ta{s}$ a box is removed and in the $k+1$-th one is added, using \cref{ncsnake} we get $y_{k,i}\Psi_{\ta{t}}^{\ta{s}}=\pm\Psi_{\ta{t}}^{\ta{u}}$, where $\ta{u}$ is obtained from $\ta{s}$ via deleting steps $k$ and $k+1$.
	
	If the $k$-th step adds a box and the $k+1$-th removes one, then $y_{k,i}\Psi_{\ta{t}}^{\ta{s}}=0$ by \cref{ncuntwist}.
	
	If in the $k$-th and $k+1$-th steps boxes are removed from $\ta{s}$, let $\ta{s}'$ be the up-down tableaux, that is obtained from $\ta{s}$ by removing all the \enquote{cups} of $\ta{s}$, i.e.~it is the same as $\ta{s}$ but whenever we would remove a box in $\ta{s}$ or add a box that later would be removed we skip this step.
	Now $\Psi_{\ta{t}}^{\ta{s}} = c\cdot\Psi_{\ta{t}}^{\ta{s}'}$, where $c$ is a diagram consisting of cups (which might intersect).
	As the $k$-th and $k+1$-th step both remove boxes, we see that $y_{k,i}\cdot\Psi_{\ta{t}}^{\ta{s}}=c'\cdot\Psi_{\ta{t}}^{\ta{s}'}$ by \cref{ncsnake}, where $c'$ also consists only of cups.
	The statement then follows from \cref{closedunderrest} for $x_{k,i}$ and $z_{k,i}$.
	
	The remaining case to consider is when two boxes are added.
	But this case immediately follows from \cref{Ass}, as $y_{k,i}\Psi_{\ta{t}}^{\ta{s}}\in\Rl{<b}$.
\end{proof}
We directly obtain from \cref{closedunderdots,closedunderrest}:
\begin{cor}\label{twosidedideal}
	Assume \cref{Ass}, then $\Bl{\leq b}=\Bts{\leq b}$ holds in $\sRcyc$.
\end{cor}
\begin{prop}\label{spanningset}
	The set $\cB$ of up-down-basis elements is a spanning set for $\sRcyc$.
\end{prop}
\begin{proof}
	By \cref{twosidedideal} we know that all $\Bl{\leq b}$ form two-sided ideals and by \cref{identitiesinspanningset} we see that all identities lie in some $\Bl{\leq b}$ for some $b$.
	These two facts together imply that the $\Psi_{\ta{t}}^{\ta{s}}\in\cB$ span $\sRcyc$.
\end{proof}
\begin{cor}\label{twicezero}
	In $\sRcyc$, any object $\bm{i}$ with a subsequence of the form $(a,a)$ is zero.
\end{cor}
\begin{proof}
	This follows now directly from \cref{spanningset} and \cref{corobj}\cref{corobj1}.
\end{proof}

\begin{cor}\label{isomorphismclassesofsRcycobjects}
	Any nonzero object of $\sRcyc$ is isomorphic to $\bm{i}_{\ta{t}^\lambda}$ for some $\lambda\in\Par^\level$.
\end{cor}
\begin{proof}
	Let $\bm{i}$ be a nonzero object in $\sRcyc$.
	If it is the residue sequence of some up-tableau, the statement follows.
	This is as all residue sequences for up-tableaux of the same shape differ only by distant crossings (which are isomorphisms by \cref{ncinverse}).
	
	Otherwise, we can find a subsequence $(i,i\pm1,i)$ in $\bm{i}$ by \cref{corobj}\cref{corobj3}.
	By \cref{twicezero} and \cref{ncbraid}, we can do the following replacement
	\begin{equation*}
		\begin{tikzpicture}[line width = \lw, myscale=0.6]
			\pgfresetboundingbox
			\node at (3,0) (C) {$i$};
			\node at (4,0) (D) {$i+1$};
			\node at (5,0) (E) {$i$};
			\node at (3,1.5) (H) {$i$};
			\node at (4,1.5) (I) {$i+1$};
			\node at (5,1.5) (J) {$i$};
			\draw(C)--(H) (D)--(I) (E)--(J);
		\end{tikzpicture}=-
		\begin{tikzpicture}[line width = \lw, myscale=0.6]
			\pgfresetboundingbox
			\node at (3,0) (C) {$i$};
			\node at (4,0) (D) {$i+1$};
			\node at (5,0) (E) {$i$};
			\node at (3,1.5) (H) {$i$};
			\node at (4,1.5) (I) {$i+1$};
			\node at (5,1.5) (J) {$i$};
			\draw (C)..controls +(0.2,0.8) and +(-0.2,0.8)..(D) (E)--(H) (I)..controls +(0.2,-0.8) and +(-0.2,-0.8)..(J);
		\end{tikzpicture},\qquad\qquad
		\begin{tikzpicture}[line width = \lw, myscale=0.6]
			\pgfresetboundingbox
			\node at (3,0) (C) {$i$};
			\node at (4,0) (D) {$i-1$};
			\node at (5,0) (E) {$i$};
			\node at (3,1.5) (H) {$i$};
			\node at (4,1.5) (I) {$i-1$};
			\node at (5,1.5) (J) {$i$};
			\draw(C)--(H) (D)--(I) (E)--(J);
		\end{tikzpicture}=
		\begin{tikzpicture}[line width = \lw, myscale=0.6]
			\pgfresetboundingbox
			\node at (3,0) (C) {$i$};
			\node at (4,0) (D) {$i-1$};
			\node at (5,0) (E) {$i$};
			\node at (3,1.5) (H) {$i$};
			\node at (4,1.5) (I) {$i-1$};
			\node at (5,1.5) (J) {$i$};
			\draw (D)..controls +(0.2,0.8) and +(-0.2,0.8)..(E) (C)--(J) (H)..controls +(0.2,-0.8) and +(-0.2,-0.8)..(I);
		\end{tikzpicture}.
	\end{equation*}
	In conjunction with \cref{ncsnake}, we see that $\bm{i}$ is isomorphic to $\bm{i'}$, where $\bm{i'}$ is obtained from $\bm{i}$ by replacing the subsequence $(i,i\pm1,i)$ with $i$.
	We can repeat this argument until we end up with a residue sequence of an up-tableau.
\end{proof}

\begin{rem}  After the proof of the Cyclotomic equivalence we see, for instance using \cref{sRcycisquher},  that the $\bm{i}_{\ta{t}^\lambda}$ for  $\lambda\in\Par^\level$ are indeed all nonzero.
\end{rem}

\subsection{Some spectral theory}

We first examine how generators of $\sVWcyc$ interact with the generalised eigenspaces: 
\begin{lem}\label{vwtechnicalidempotentrelations}
	Let $a$, $b$, $c$, $d\in\bbR$.
	In $\sVWcyc$ the following holds:
	\begin{enumerate}[label=$(\alph*)$]
		\item\label{technic1} If $\flat e_{a,b}\not=0$ then $b= a+1$, and if $e_{c,d}\flat^\ast\not=0$ then $d=c-1$. 
		\item\label{technic2} If $e_{c,d}s_k e_{a,b}\not=0$ then $(b,d)=(a+1,c-1)$ or $(c,d)=(a,b)$ or $(c,d)=(b,a)$.
	\end{enumerate}
\end{lem}
\begin{proof}
	Part \cref{technic1} holds by \cref{vwdotcap} respectively \cref{vwdotcup}.
	For \cref{technic2} we deduce from \cref{vwinverse} and \cref{vwdotcrossing} the two equations $e_{c,d}(s_ky_{k} + s_ky_{k+1})e_{a,b} = e_{c,d}(y_{k+1}s_k + y_ks_k + 2\flat\flat^\ast )e_{a,b}$ and $
	e_{c,d}(s_ky_{k} - s_ky_{k+1})e_{a,b} = e_{c,d}(y_{k+1}s_k - y_ks_k - 2)e_{a,b}$.
	If $(b,d)\not=(a+1,c-1)$ then in the first equation the cup-cap part vanishes and we get $a+b=c+d$.
	If $(a,b)\not=(c,d)$ then the second equation implies $a-b=d-c$, since the last term there vanishes.
	Thus, $e_{c,d}s_k e_{a,b}\not=0$ implies $(b,d)\not=(a+1,c-1)$ or $(b,d)\not=(a+1,c-1)$ are satisfied, or the two conditions $a+b=c+d$, $a-b=d-c$ hold, that is $(a,b)=(d,c)$.
\end{proof}
Next we show diagonalizability of $y_{k-1}$ induces diagonalizability of $y_{k}$ for any $k\in\bbN_0$: 
\begin{prop}\label{vwdiagonalizable}
	Assume $\Phi^{\leq k-1}$ is an isomorphism. Then $y_k$ acts diagonalisably on $\sVWcyc_{\leq k}$.
\end{prop}
\begin{rem}\label{remPhikiso}
	If $\Phi^{\leq k-1}$ is an isomorphism, then $y_{k-1}$ acts diagonalisably for $\sVWcyc_{\leq k-1}$  (since its preimage does so by \cref{thetainotaddablebox}), and $e_{i_1, \dots,i,i,\dots, i_{k-1}}=0$.
\end{rem}	
Throughout the following proofs we will need the following technical result:
\begin{lem}\label{replacementforonezero}
	Assume  $\Phi^{\leq k-1}$ is an isomorphism. Then we have for any $i\in\bbR$,  $x\in\sVWcyc$, $(i_1, \cdots, i_{k-2})\in\bbR^{k-2}$  that $e_{i_1, \cdots, i_{k-2},i+1,i}x e_{i_1, \cdots, i_{k-2},i,i+1}=0$.
\end{lem}
\begin{proof}
	By \cref{isomorphismclassesofsRcycobjects}, we know that $e_{i_1, \cdots, i_{k-2}}$ is isomorphic to $e_{j_1, \dots, j_l}$ with $l\leq k-2$, where $j_1, \dots, j_l$ is the residue sequence of some $\ta{t}^\lambda$, $\lambda\in\Par^\level$.
	As $\ta{t}^\lambda$ cannot have addable boxes of residue $i$ and $i+1$ at the same time, one of the two idempotents must be conjugate to some $e_{j'_1, \cdots, j'_{l'}}$ with $l'<k-2$ (or zero which directly implies the claim).
	Adding a snake to $e_{i_1, \cdots, i_{k-2},i+1,i}x e_{i_1, \cdots, i_{k-2},i,i+1}$ and using the above conjugate idempotent, we obtain an idempotent of the form $e_{j'_1, \cdots, j'_{l'},a,a}$ with $a\in\{i,i+1\}$ ($i$ if the first idempotent is conjugate to a shorter one, $i+1$ if the second one is).
	In particular, $e_{j'_1, \cdots, j'_{l'},a,a}=$ by \cref{twicezero} as $\Phi^{\leq k-1}$ is an isomorphism.
	The statement follows.
\end{proof}

\begin{proof}[Proof of \cref{vwdiagonalizable}]
	It suffices to show the claim: 
	\begin{equation}\label{rhabarberschorle}
		e_{i_1, \dots, i_k}y_k=i_k e_{i_1, \dots, i_k}\quad \text{ for any }\quad e_{i_1, \dots, i_k}.
	\end{equation}
	For $k=1$, this holds by definition of $\sVWcyc$ and the minimal polynomial \cref{defOmega} of $y_1$.
	Thus, let $k>1$.
	We abbreviate $e_{(a,b]}\coloneqq e_{i_1, \dots i_{k-2},a,b}$ and set $j\coloneqq i_{k-1}$, $i\coloneqq i_k$.
	From \cref{vwdotcrossing} we get with $e=\sum_{a',b'}e_{(a',b'] }$ the formula 
	\begin{equation}\label{master}
		e_{(c,d]} y_k e_{(a,b]}= e_{(c,d]}s_{k-1} y_{k-1} e s_{k-1}e_{(a,b] }+ e_{(c,d]} s_{k-1}e_{(a,b]}+e_{(c,d]} \flat^\ast\flat e_{(a,b] }.
	\end{equation}
	Note that the last summand vanishes in case $(a,b)=(c,d)$ by \cref{vwtechnicalidempotentrelations}.
	
	\emph{Case $j=i$}.
	If we take $(a,b)=(c,d)=(i,i)$ in \cref{master}, only $(a',b')=(i,i)$ matters by \cref{vwtechnicalidempotentrelations}, and we get $ e_{(i,i]}(y_k-i)= e_{(i,i]}s_{k-1} e_{(i,i]}$, since $y_{k-1}$ acts diagonalisably by assumption.
	Now, $(e_{(i,i]}(y_k-i) e_{(i,i]})^{2n}=0$ for $n\gg0$ whereas $(e_{(i,i]}s_{k-1} e_{(i,i]})^{2n}= e_{(i,i]}$ by \cref{vwtechnicalidempotentrelations} and \cref{vwinverse}.
	Thus, $ e_{(i,i]}=0$.
	
	\emph{Case $j=i+1$}.
	Consider \cref{master} with $(a,b)=(c,d)=(i+1,i)$.
	By \cref{vwtechnicalidempotentrelations}, only the terms $e_{(i,i+1]}$ and $e_{(i+1,i]}$ matter for $e$.
	But by \cref{replacementforonezero}, actually the term for $e_{(i,i+1]}$ vanishes as well and only $e_{(i+1,i]}$ remains.
	Then, we can use the same argument as for the case $j=i$.
	
	\emph{Case $j=i-1$}.
	If we take now $(a,b)=(c,d)=(i-1,i)$ in \cref{master}, only $(a',b')=(i-1,i)$ matters by \cref{vwtechnicalidempotentrelations,replacementforonezero}, and we can argue as for the above two cases. 
	
	\emph{Case $j\notin\{i,i\pm1\}$}.
	Let $(a,b)\in\{(i,j),(j,i)\}$ and set $z_{(a,b)}\coloneqq (s_{k-1}+\frac{1}{a-y_k})e_{(a,b]}$.
	Since the action of $y_{k-1}$ is diagonalizable, \cref{vwdotcrossing} implies that $y_kz_{(a,b)}=z_{(a,b)}y_{k-1}=az_{(a,b)}$ and $bz_{(a,b)}=z_{(a,b)}y_k=y_{k-1}z_{(a,b)}$.
	In particular, $ e_{(b,a]}z_{(a,b)}=z_{(a,b)}$.
	We get $z_{(b,a)}z_{(a,b)}= (s_{k-1}+\frac{1}{b-a})(s_{k-1}+\frac{1}{a-b})e_{(a,b]} = (1-\frac{1}{(a-b)^2})e_{(a,b]}$.
	Since $a-b\not=\pm1$ and $z_{(a,b)}(y_k-b)e_{(a,b]}=0$, we get $(y_k-b)e_{(a,b]}=0$ for $(a,b)\in\{(i,j),(j,i)\}$.
	
	We showed that $y_k$ is diagonalizable.		
\end{proof}
The following two results follow directly from the proof of \cref{vwdiagonalizable}.
\begin{cor}\label{twicezerovw}
	If $y_{k-1}$ acts diagonalisably on $\sVWcyc_{\leq k-1}$, then $e_{i_1, \dots, i_{k-2},i,i}=0$.
\end{cor}
\begin{cor}\label{idempotentnotneeded}
	If $y_{k-1}$ acts diagonalisably on $\sVWcyc_{\leq k-1}$, then $e_{a,b}((b-a)s_k + 1) = ((b-a)s_k + 1)e_{b,a}$ given that $a\notin\{b,b-1\}$.
\end{cor}
\subsection{Proof of the Isomorphism Theorem}
Now we are going to prove \cref{isoncvw}.
We will begin by outlining our strategy.

Consider the functor $\Phi\colon\sRcyc\to\mathrm{Kar}(\sVWcyc)$.
This functor is filtered by the number of strands, and we can consider its restriction $\Phi^{\leq k}$ to at most $k$ strands (on either side).
We then will prove \cref{isoncvw,spanningsetisbasis} by induction on $k$.
For $k=1$ this is an easy calculation.
Given the theorem for all $k'\leq k$, we will show that $y_{k+1}$ act diagonalisably and use this to check the relations involving the $k+1$-st strand.
For both calculations we will use the basis of $\sRcyc$ (on the first $k-1$ strands) to exclude and simplify many cases in the calculations.

From these considerations it will also follow that the functor is full and by arguments from \cite{AMR06} it follows that the spanning set of $\sRcyc$ has the same size as a basis for $\sVWcyc$ (up to this filtration degree).
\begin{prop}\label{psiwelldefined}
	Let $k\in\mathbb{N}_0$ and assume $\Phi^{\leq k-1}$ is an isomorphism of algebras.
	Then $\Phi^{\leq k}$ is a well-defined algebra homomorphism.
	%If $y_{k-1}$ acts diagonalisably on $\sVWcyc_{\leq k-1}$ and not both of $e_{i_1, \dots, i_{k-2}, i}$ and $e_{i_1, \dots, i_{k-2}, i+1}$ are nonzero, then $\Phi^{\leq k}$ is well-defined.
\end{prop}
\begin{proof}
	If $k=1$, the only relations are the cyclotomic relations \cref{cycrelations} for $\sRcyc$ and \cref{vwcycrelation} with \cref{defOmega} which exactly correspond to each other, and thus the functor is well-defined (in this case the assumption is vacuous).
	
	If $k>1$ it suffices by assumption to verify the compatibility with the relations involving the last strand.
	Recall from \cref{tea2} that all dots in $\sRcyc$ become zero which fits with the fact that $y_k$ acts diagonalizable by \cref{vwdiagonalizable}.
	We can ignore all terms involving dots in the relations, since they are zero and sent to zero.
	Again we abbreviate $e_{(a,b]}\coloneqq e_{i_1, \dots i_{k-2},a,b}$, $e_{(a,b,c]}\coloneqq e_{i_1, \dots a,b,c}$.
	
	\emph{Relation \cref{ncdotcap}}: Both sides are zero and are sent to zero.\hfill\\
	\emph{Relation \cref{ncdotcrossing}}: The right-hand sides are sent to zero by \cref{twicezerovw} respectively by \cref{replacementforonezero} using that $\Phi^{\leq k-1}$ is an isomorphism.\hfill\\
	\emph{Relation \cref{nctangleone}}: We may assume that $a\neq b$, $b+1$ as otherwise both sides are sent to zero by definition.
	The LHS of the relation is (using \cref{idempotentnotneeded}) sent to 
	\begin{equation}\label{LHS3}
		\eta_{a,b}\flat_{k-1}((a-b) s_{k-2}+1)e_{(a,b+1]}=\eta_{a,b}\flat_{k-1}((a-b) s_{k-2})e_{(a,b+1]}.	
	\end{equation}
	Here, the second summand vanishes by \cref{vwtechnicalidempotentrelations}.
	Similarly, the RHS is sent to 
	\begin{equation}\label{RHS3}
		\eta_{b+1,a}\flat_{k-2}((b+1-a) s_{k}+1)e_{(a,b+1]}=\eta_{b+1,a}\flat_{k-2}((b+1-a) s_{k})e_{(a,b+1]}.
	\end{equation}
	Now \cref{LHS3}=\cref{RHS3} holds by the defining property \cref{scalartwo} of the $\eta$'s and \cref{vwtangletwo}.\hfill\\
	\emph{Relation \cref{ncsnake}}: By \cref{vwtechnicalidempotentrelations} the middle idempotent in the image is uniquely determined by the outer idempotents and the compatibility follows from \cref{vwsnake}.	\hfill\\	
	\emph{Relation \cref{ncuntwist}}: The image is zero by \cref{vwuntwist} noting that $i_k-i_{k-1}=a-(a+1)=-1$.\hfill\\
	\emph{Relation \cref{ncinverse}}: The first case is clear by \cref{twicezerovw}, the second and third case follow from \cref{replacementforonezero}.
	For the remaining one note that the image of the LHS is $\eta_{a,b}\eta_{b,a}((b-a)s_k+1)((a-b)s_k+1)e_{(a,b]}$ by \cref{idempotentnotneeded}.
	This equals $\eta_{a,b}\eta_{b,a}(1-(a-b)^2)e_{(a,b]}$ by \cref{vwinverse}.
	By the first defining property of the $\eta$'s, $\eta_{a,b}\eta_{b,a}(1-(a-b)^2)=1$ and the desired compatibility holds.\hfill\\		
	\emph{Relation \cref{ncbraid}}:	First assume $(a,b,c)\not=(a,a\pm1,a)$.
	Then the RHS is zero and sent to zero.
	The left-hand side is mapped to zero if any of the pairs $(a,b)$, $(a,c)$, $(b,c)$ are of the form $(i,i)$ or $(i,i+1)$ by definition of $\Phi$.
	Otherwise, the image of the first term is 	
	\begin{equation*}
		\begin{aligned}	
			&\eta_{a,b}((b-a)s_{k-1}+1)\eta_{a,c}((c-a)s_{k}+1)\eta_{b,c}((c-b)s_{k-1}+1)\\
			= \;\;&\eta_{a,b}\eta_{a,c}\eta_{b,c} \left(1+(b-a)(c-b)s_{k-1}^2+ (c-a)s_{k}+(b-a)s_{k-1}+(c-b)s_{k-1}\right.\\
			&\left.+(c-a)(c-b)s_{k}s_{k-1}+(b-a)(c-a)s_{k-1}s_{k}+(b-a)(c-a)(c-b)s_{k-1}s_{k}s_{k-1}\right),
		\end{aligned}
	\end{equation*}		
	whereas the image of the second term equals
	\begin{equation*}
		\begin{aligned}
			&\eta_{b,c}((c-b)s_{k}+1)\circ\eta_{a,c}((c-a)s_{k-1}+1)\circ\eta_{a,b}((b-a)s_{k}+1)\\
			=\;\;&\eta_{b,c}\eta_{a,c}\eta_{a,b}\left(1+(c-b)(b-a)s_{k}^2	+(b-a)s_{k}+(c-a)s_{k-1}+(c-b)s_{k}\right.\\
			&\left.+(c-a)(b-a)s_{k-1}s_{k} +(c-b)(c-a)s_{k}s_{k-1}+(c-b)(c-a)(b-a)s_{k}s_{k-1}s_{k}\right).
		\end{aligned}
	\end{equation*}	
	The two images agree in all expressions involving one or two $s_i$'s.
	The other terms match by \cref{vwinverse} and \cref{vwbraid}.
	
	Next assume that $(a,b,c)=(a,a\pm1,a)$.
	We need to show that 
	\begin{equation}\label{claimimZug}
		e_{(a, a+1,a]} = -e_{(a,a+1, a]}\flat^\ast_{k-1}\flat_{k-1}e_{(a,a-1,a]}\flat^\ast_{k-2}\flat_{k-2}e_{(a,a+1,a]}.
	\end{equation}
	We first rewrite $e_{(a, a+1,a]}$ by plugging in the relation \cref{vwdotcrossing} three times.
	We always simplify using that the $y_j$'s act by scalars (and thus the double cross can be straightened by \cref{vwinverse}) and that certain cups or caps vanish because of \cref{vwtechnicalidempotentrelations}.
	\begin{align*}
		e_{(a,a+1,a]} &=e_{(a,a+1,a]} s_{k-2}e_{(a,a+1,a]} =e_{(a,a+1,a]}(-s_{k-1}-\flat^\ast_{k-1}\flat_{k-1}) s_{k-2}e_{(a,a+1,i]} \\
		&=-e_{(a,a+1,a]} s_{k-1}s_{k-2}e_{(a,a+1,a]}-e_{(a,a+1,a]} s_{k-1}\flat^\ast_{k-1}\flat_{k-1}\flat_{k-2}^\ast\flat_{k-2}e_{(a,a+1,a]}.
	\end{align*}
	Only with the idempotent $e_{(a,a-1a]}$ in the middle, the last term is nonzero.
	Thus, \cref{claimimZug} follows if we show that $e_{(a,a+1,a]} s_{k-1}s_{k-2}e_{(a,a+1,a]}=0$.
	
	For this we observe that (again by \cref{vwdotcrossing}, diagonalizability and \cref{vwtechnicalidempotentrelations})
	\begin{align}\label{helper}
		e_{(a,a+1,a]}s_{k-1}e_{(a,a+1,a]}&=(a-(a+1))e_{(a,a+1,a]}=-e_{(a,a+1,a]},\\
		\label{helper2}e_{(a,a+1,a]}s_{k-2}e_{(a,a+1,a]}&=(a+1-a)e_{(a,a+1,a]}=e_{(a,a+1,a]},
	\end{align}
	and compute (using \cref{twicezerovw} and \cref{vwtechnicalidempotentrelations} in the second and fourth step)
	\begin{align*}
		&e_{(a,a+1,a]} s_{k-1}s_{k-2}e_{(a,a+1,a]}
		\overset{\cref{helper}}{=}-e_{(a,a+1,a]} s_{k-1}s_{k-2}e_{(a,a+1,a]}s_{k-1}e_{(a,a+1,a]}\\
		=&-e_{(a,a+1,a]} s_{k-1}s_{k-2}s_{k-1}e_{(a,a+1,a]}
		\overset{\cref{vwbraid}}{=}-e_{(a,a+1,a]} s_{k-2}s_{k-1}s_{k-2}e_{(a,a+1,a]}\\
		=&-e_{(a,a+1,a]} s_{k-2}e_{(a,a+1,a]}s_{k-1}s_{k-2}e_{(a,a+1,a]}
		\overset{\cref{helper2}}{=}-e_{(a,a+1,a]} s_{k-1}s_{k-2}e_{(a,a+1,a]}.
	\end{align*}
	Therefore, $e_{(a,a+1,a]} s_{k-1}s_{k-2}e_{(a,a+1,a]}=0$ and \cref{claimimZug} is proven.\hfill\\
	The case $(a,b,c)=(a,a-1,a)$ is treated analogously.
\end{proof}
\begin{proof}[Proof of \cref{isoncvw}]
	We prove that $\Phi^{\leq k}$ is an isomorphism by induction on $k$.
	For $k=0$ there is nothing to show.
	Now assume the statement for $k-1$.
	Then $\Phi^{\leq k}$ is well-defined by \cref{psiwelldefined}.
	Furthermore, the spanning set $\cB$ for $\sRcyc$ has the same size as a basis of $\sVWcyc$, see \cite{AMR06}*{Lemma 5.1}.
	Hence, it suffices to show that $\Phi^{\leq k}$ is full.
	For this let ${\bm{i}}\in\bbR^m$, $e_{\bm{j}}\in\bbR^{n}$, $m,n\leq k$.
	It is clear that $e_{\bm{i}}y_je_{\bm{j}}\in\im\Psi^{\leq }$ for all $1\leq j\leq k$.
	By \cref{vwtechnicalidempotentrelations}, we also have $e_{\bm{i}}\flat_ke_{\bm{j}}$ and $e_{\bm{i}}\flat^\ast_ke_{\bm{j}}\in\im\Psi^{\leq k}$ whenever they make sense.
	We claim that $e_{\bm{j}}s_{k-1}e_{\bm{i}}\in\im\Phi^{\leq k}$. 
	By induction, it suffices to show that $e_{\bm{i}}s_ke_{\bm{j}}\in\im\Phi^{\leq k}$ for $m=n=k-1$.
	If $i_k\notin\{ i_{k-1},i_{k-1}-1\}$ this is clear by definition of $\Phi^{\leq k}$.
	If $i_k=i_{k-1}$ then $e_{\bm{i}}=0$ by \cref{twicezerovw} and there is nothing to do.
	
	Thus, assume $i_k+1=i_{k-1}\eqqcolon i$. 
	By \cref{vwtechnicalidempotentrelations}, we have $e_{\bm{i}}s_{k-1}e_{\bm{j}}=0$ unless $(j_k, j_{k-1})=(i_{k-1},i_k)$ or $(j_{k-1},j_k)=(i_{k-1},i_k)$.
	For the former, we have $e_{(i+1,i]}s_{k-1}e_{(i,i+1]}=0$ by \cref{replacementforonezero}.
	For the latter, we get $e_{\bm{i}}s_ke_{\bm{j}}=e_{\bm{i}}e_{\bm{j}}$ by \cref{vwdotcrossing}.
	Therefore, $e_{\bm{j}}s_ke_{\bm{i}}\in\im\Phi^{\leq k}$ as claimed.
	Similarly, if $i_k=i_{k-1}+1$ we have $e_{\bm{i}}s_{k-1}e_{\bm{j}}=0$ by \cref{replacementforonezero} and thus $e_{\bm{i}}s_ke_{\bm{j}}=0$.
	
	Altogether, $\im\Psi^{\leq k}$ contains a generating set for the morphism and thus $\Psi^{\leq k}$ is full.
	It follows that $\Psi$ is an isomorphism.
\end{proof}
\subsection{Proof of the Basis Theorem and the Cyclotomic Equivalence}
\begin{proof}[Proof of \cref{spanningsetisbasis}]
	Since the cardinality of $\cB$ equals the cardinality, see \cite{AMR06}*{Lemma 5.1}, of a basis of $\sVWcyc$, the Basis Theorem follows from 
	the Isomorphism Theorem~\ref{isoncvw} and \cref{spanningset}.
\end{proof}
\begin{proof}[Proof of \cref{cycequiv}]
	By the Isomorphism \cref{isoncvw} it is enough to show that the functor is essentially surjective. Write $1=\sum e_{\bm{i}}$ 
	for pairwise orthogonal nonzero idempotents. We claim that $e_{\bm{i}}$ is primitive for all $\bm{i}$. If the claim holds we are done, since then the image contains (up to equivalence) all primitive idempotents. By \cref{isomorphismclassesofsRcycobjects} we can restrict ourselves to the case $\bm{i}=\bm{i}^\la$ for $\la\in\Par^\level$. Then the claim follows from  \cref{sRcycisquher}.  (One could also directly use \cref{sRcycisquher}.)
\end{proof}

\section{Gradings, free $\bbZ$-actions and categories of representations}
Instead of working with (strict monoidal) $\gosVec$-categories $\cC$, we could equivalently work with (strict monoidal) $\sVec$-categories $\cC^\bbZ$, but equipped with a free $\bbZ$-action given by (strict monoidal) isomorphisms $\langle i\rangle$, $i\in\bbZ$, such that $\langle i\rangle\langle j\rangle=\langle i+j\rangle$. 

More precisely we have the following,  see \cite[(2.1)]{MazStrKoszul}:
\begin{lem}\label{correspondence}
	There is a correspondence
	\begin{eqnarray*}
		\CatZorbit\coloneqq\{\gosVec\text{-categories}\}& \leftrightarrow &\{\sVec\text{-categories with a free $\bbZ$-action}\}\eqqcolon\CatZact\\
		\cC&\mapsto &\cC^\bbZ\nonumber\\
		\cC_\bbZ&\mapsfrom&\cC\nonumber
	\end{eqnarray*}
\end{lem}		
Here, a \emph{$\bbZ$-action} means an action by automorphisms $\langle i\rangle, i\in\bbZ$ such that $\langle i\rangle \langle j\rangle=\langle i+j\rangle$ (and freely means that the stabilizer of every object is trivial).

In $\cC^\bbZ$, the objects are $\langle i\rangle c$, with $i\in\bbZ$, $c\in\cC$ and $\Hom_{\cC^\bbZ}(\langle i\rangle c,\langle j\rangle c')\coloneqq\Hom_{\cC}(c, c')_{i-j}$.
The \emph{orbit category} $\cC_\bbZ$ has objects the orbits $[c]$ of objects in $\cC$ with a fixed representative $\hat{c}$.
The morphisms are $\Hom_{\cC_\bbZ}([c],[c'])_i\coloneqq\Hom_{\cC_\bbZ}(\langle i\rangle \hat{c},\hat{c})=\Hom_{\cC_\bbZ}(\hat{c},\langle -i\rangle \hat{c})$.

\begin{rem}	
	One could work with any group $G$ and with $G$-graded vector spaces instead. If we work with $G=\bbZ/{2\bbZ}$ and with $\operatorname{Vec}$ instead of $\sVec$ our notion of supercategories turns into the notion of supercategories using free ${\bbZ_2}$-actions as defined e.g.~in \cite{KKO}. 
\end{rem}

Concretely, in $(\sR)^\bbZ$, objects are $\langle i\rangle \mathbf{a}$, $i\in\bbZ$ with $\mathbf{a}\in\sR$ and $\Hom_{\sR^\bbZ}(\langle i\rangle\mathbf{a},\langle j\rangle \mathbf{b})=\Hom_{\sR}(\mathbf{a}, \mathbf{b})_{i-j}$, the degree $i-j$ morphisms in $\sR$.
As monoidal supercategory with $\bbZ$-action, $\sR^\bbZ$ is generated by objects $a=\langle 0\rangle a$, $a\in\bbR$, and morphisms $(f\colon \mathbf{a}\to\langle -i\rangle \mathbf{b})\in\sVec$ for any $(f\colon \mathbf{a}\to \mathbf{b})\in\gosVec$ from \cref{genshoms} of degree $i$, subject to \cref{ncdotcap}-\cref{ncbraid} interpreted in the same way.

\begin{rem} \label{lessconfusing}
	Given $\cC\in\CatZorbit$ there is an equivalence $(\cC^{\mathrm{op}})^\bbZ\cong (\cC^\bbZ)^{\mathrm{op}}$ given by 
$\langle i\rangle c \mapsto \langle -i\rangle c$ noting that the following holds for morphisms
$\Hom_{(\cC^{\mathrm{op}})^\bbZ}(\langle i\rangle c, \langle j\rangle d)
	=\Hom_{\cC^{\mathrm{op}}}(c, d)_{i-j}
	=\Hom_{\cC}(d, c)_{i-j}=
	\Hom_{\cC^\bbZ}(\langle -j\rangle d, \langle -i\rangle c)
	=\Hom_{(\cC^\bbZ)^{\mathrm{op}}}(\langle -i\rangle c, \langle -j\rangle d)$.
\end{rem}

\begin{rem}\label{orbitcat}
	We can view $\CatZact$ and $\CatZorbit$ as categories with morphisms given by functors compatible with the $\bbZ$-action respectively by $\gsVec$-functors\footnote{Note that we do not take $\gosVec$-functors here.
	Instead, we view any $\gosVec$-category as a $\gsVec$-category.}.
	Then the correspondence from \cref{correspondence} extends to a functor 
	\begin{equation}\label{functorofbicats}
		\CatZact\to\CatZorbit\text{ sending a morphism }F\colon\cC\to\cD \text{ to }F_\bbZ\colon\cC_\bbZ\to\cD_\bbZ,
	\end{equation}
	with $F_\bbZ$ defined as follows.
	On objects $F_\bbZ([c])=[F(c)]$ and $f\in\Hom_{\cC_\bbZ}([\hat{c_1}], [\hat{c_2}])_{i}=\Hom_\cC(
	\langle i\rangle \hat{c_1}, \hat{c_2})$ is sent to $F_\bbZ(f)=F(f)\in \Hom_{\cD_\bbZ}(F_\bbZ([\hat{c_1}]), F_\bbZ([\hat{c_2}]))_{m_1+i-m_2}$, where $m_i\in\bbZ$ for $i=1,2$ such that $ F(\hat{c_i})=\langle m_i\rangle \widehat{F(\hat{c_i})}$.
	Here we use that the $\bbZ$-action is free and that $F(f)\in \Hom_\cD(F(\langle i\rangle \hat{c_1}), F(\hat{c_2}))=\Hom_\cD(\langle i\rangle F(\hat{c_1}), F(\hat{c_2})).$ 
\end{rem}
\begin{warn}	
	A $\gsVec$-functor $F\colon\cC\to\cD$ might not have a preimage under \cref{functorofbicats}.
	The functors relevant for representation theory however usually have lifts. For instance, the functors $\Sigma$ and $\Tau$  have graded lifts which are given on objects by  $a\mapsto\langle\epsilon\rangle (a+1)$ respectively $a\mapsto-\langle\epsilon\rangle a$. 
	For an example of the existence and construction of graded lifts which are less obvious see e.g.~\cite{Stgradedtransl}.
\end{warn}

\begin{defi}
	A left (resp.~right) $\cC$-module for $\cC\in\CatZorbit$ is a co(ntra)variant $\gosVec$-functor $M\colon\cC\to\gsVec$.
	The categories $\CRep$ (and $\RepC$) of left (resp.~right) modules can again be viewed as $\gosVec$-categories as explained in \cite{Kelly}. 
	These are objects in $\CatZact$ with $\bbZ$-action given by $\langle i\rangle M(c)=\langle i\rangle (M(c))$, where the $\bbZ$-action on $\gsVec$ is given by $(\langle i\rangle V)_{n+i}=V_n$ for $i,n\in\bbZ$.

	A left (resp.~right) $\cC$-module for $\cC\in\CatZact$ is a co(ntra)variant functor $M\colon\cC\to\sVec$ of $\sVec$-categories.
	We denote by $\CRep$ (and $\RepC$) the corresponding $\sVec$-category of left (resp.~right) modules.  
	This is  an object in $\CatZact$ with $\bbZ$-action given by $\langle i\rangle (M)(c)=M(\langle -i\rangle c)$ (resp.~$\langle i\rangle (M)(c)=M(\langle i\rangle c)$). 
\end{defi}
\begin{rem} We have $\RepC\coloneqq\cC^\text{op}\text{-}\Rep$ using the opposite category, \cite[\S1.4]{Kelly}.
\end{rem}
The following are important examples of left and right modules: 
\begin{defi}\label{defproj}	
	Let $\cC\in\CatZact$ or $\cC\in\CatZorbit$.
	The corresponding \emph{projective modules} are ${P_{c}}\coloneqq\Hom_{\cC}(c,\_)\in\CRep$ and ${{}_{c}P}\coloneqq\Hom_{\cC}(\_,c)\in\RepC$.
	The \emph{regular $\cC$-modules} $\cC$ are defined as $\cC=\bigoplus_c P_{c}\in\CRep$ and $\cC=\bigoplus_c {}_{c}P\in\RepC$.
\end{defi}

For readers who refer less categorical notions the following remark is important: 
\begin{rem}\label{usualalgebra}
	The data of a module $M\in \sRRep$ or $M\in\RepsR$ is, by taking $\bigoplus_{\bm{i}}M(\bm{i})$, equivalent to the data of an ordinary (locally unital) left, respectively right, module for the electric KLR superalgebra from \cref{locallyunital}.
	The notion of projective and regular modules then boils down to the usual notion of projective modules for a (locally unital) superalgebra.
\end{rem}

\begin{lem} Let $\cC\in\CatZact$ and $\cD\in\CatZorbit$.
	Then $\DRep$ can be viewed as object in $\CatZact$ by setting $(\langle a\rangle M)(d):=\langle a\rangle(M(d))$ and there are isomorphisms in $\CatZact$:
	\begin{equation}\label{strange}
		\begin{aligned}[c]
			\CRep&\cong\cC_\bbZ\text{-}\Rep\\
			M\;\;\;&\mapsto\;M_\bbZ\\
			M^\bbZ\;\;&\mapsfrom\;\; M
		\end{aligned}
		\qquad\text{and}\qquad
		\begin{aligned}[c]
			\DRep&\cong \cD^\bbZ\text{-}\Rep\\
			M\;\;\;&\mapsto\;M^\bbZ\\
			M_\bbZ\;\;&\mapsfrom\;\; M
		\end{aligned}
	\end{equation}
\end{lem}
\begin{proof}
	In the first case let $M_\bbZ([c])=\oplus_{m\in\bbZ} M(\langle -m\rangle \hat{c})$ and $M^\bbZ(\langle m\rangle\hat{c})=M([c])_{-m}$.
	Any $f\in\Hom_{C_\bbZ}([c],[c'])_k$ defines an element in $\Hom_{\cC}(\langle k-a\rangle\hat{c},\langle -a\rangle\hat{c'})$ for any $a\in\bbZ$ and then in $\Hom_{\sVec}(M(\langle k-a\rangle\hat{c}),M(\langle -a\rangle\hat{c'}))=\Hom_{\sVec}(M_\bbZ([c])_{a-k},M_\bbZ([c'])_{a})$.
	These maps, for $a\in\bbZ$, are the components of $M_\bbZ([c])(f)$.
	
	Conversely, if $f\in\Hom_{\cC}(\langle m\rangle \hat{c},\langle n\rangle \hat{c'})=\Hom_{\cC_\bbZ}([c],[c']))_{m-n}$ we get $M^\bbZ(f)=M(f)\in\Hom_{\sVec}(M([c])_{-m},M(c')_{-n})=\Hom_\sVec(M^\bbZ(\langle m\rangle \hat{c}),M^\bbZ(\langle n\rangle \hat{c'}))$.
	We omit checking that these define the isomorphisms.
	The second case is analogous.
\end{proof}

\begin{rem}Consider the case $\cD=\sR$ or $\cD=\sRcyc$.
	Under the second isomorphism \cref{strange} the projective module $\langle m\rangle P^{(\level)}_{\bm{i}}$ correspond to $P^{(\level)}_{\langle m\rangle\bm{i}}$ for $m\in\bbZ$.
\end{rem}

\begin{nota}\label{Ko}
	Let $\cC\in\CatZact$.
	Given an additive subcategory $\mathcal{A}$ of $\CRep$ closed under the $\bbZ$-action, we denote by $K_0'(\mathcal{A})$ the usual additive Grothendieck group.
	This is a $\bbZ[q,q^{-1}]$-module by identifying $q$ with $\langle1\rangle$ in case $\cA$ is invariant under the $\bbZ$-action.
	We write then $K_0(\mathcal{A})\coloneqq\bbQ(q)\otimes_{\bbZ[q,q^{-1}]}K_0'(\mathcal{A}).$
\end{nota}
This definition applies in particular to the following categories: 
\begin{defi}\label{proj}
	For $\cC\in\CatZact$ let $\cCproj$ be the idempotent closed additive subcategory of $\CRep$ generated by the projectives $P_ c$, $c\in\cC$. Given   $\cC\in\CatZorbit$  we write by abuse of language $\cCproj$ for the category $\cDproj$ where $\cD=\cC^\bbZ$. 	
\end{defi}

\begin{rem}\label{confusing}
	The identity on objects and morphisms defines a \emph{contravariant} functor $\operatorname{id}\colon\sR\to\sR^{\mathrm{op}}$ which induces via \cref{lessconfusing} a contravariant functor $\sR^\bbZ\to(\sR^\bbZ)^{\mathrm{op}}$. It induces a $q$-\emph{antilinear} map on $K_0$ of the representation categories. 
\end{rem}

\section{Projective modules for the (cyclotomic) electric KLR algebras}\label{sec7}
In this section we study the category of projective modules for $\sR$ and $\sRcyc$ with their Grothendieck groups.
We start with some definitions.

\begin{defi} Let $\sRproj$ be the idempotent closed additive subcategory of $\sRRep$ generated by the projectives $P_ {\bm{i}}$.
	Similarly, we define $\sRcycproj$ for $\sRcyc$ and denote here the projective module associated with $\bm{i}$ as $P^{\level}_{\bm{i}}$ to indicate the dependence on $\level$.
	Let $\projsR$, $\projsRcyc$ be the analogues for right modules.
\end{defi}
Similarly, let $\sRproj^\bbZ$ be the idempotent closed additive subcategory of $\sR^\bbZ\text{-}Rep$ generated by the projectives $P_ {\langle m\rangle\bm{i}}$, $m\in\bbZ$.

\begin{nota} Given $\sVec$-categories $\cC$ and $\cD$, we denote by $\cC\boxtimes\cD$ the \emph{Deligne--Kelly tensor product} of $\cC$ and $\cD$. Given $M\in\CRep$, $N\in\DRep$, we have the \emph{outer tensor product} $M\boxtimes N\in \cC\boxtimes\cD\text{-}\Rep$ given by $M\boxtimes N(c,d)\coloneqq M(c)\otimes N(d)$.
\end{nota}
\begin{rem}More precisely, objects of $\cC\boxtimes\cD$  are pairs $(c,d)$ with $c\in\cC$, $d\in\cD$ and $\Hom_{\cC\boxtimes\cD}((c,d),(c',d'))=\Hom_\cC(c,c')\otimes \Hom_\cD(d,d')$.  The tensor product is in $\gsVec$ or $\gosVec$ if the original categories were enriched in these. 
	For details on the abstract definition see \cite[6.5]{Kelly}. Note that this construction is compatible with \cref{correspondence} in the sense that 
$(\cC\boxtimes\cD)_\bbZ\cong \cC_\bbZ\boxtimes\cD_\bbZ$ and $(\cC\boxtimes\cD)^\bbZ\cong \cC^\bbZ\boxtimes\cD^\bbZ$. 
\end{rem}	

\subsection{Tensor products of projective modules for $\sR$ and $\sRcyc$}
Using horizontal stacking of diagrams in $\sR$ we have a canonical map $\sR\boxtimes\sR\to\sR$ which allows us to view the regular module $\sR$ 
as a $(\sR,\sR\boxtimes\sR)$-bimodule.
As in \cite[\S2.6]{KL09} this provides induction and restriction functors and the following definition: 
\begin{defi}
	For $M,N\in \sRRep$ define their \emph{tensor product} 
	\begin{equation}\label{tensorproduct}
		M\cdot N\coloneqq\ind_{\sR\boxtimes\sR}^{\sR}M\boxtimes N \in\sRRep.
	\end{equation}
	The tensor product $M\cdot N$ of two right $\sR$-modules is defined similarly.
\end{defi}	
The following statements about $\sRproj$ and $\projsR$ are clear from the definitions:
\begin{lem}\label{tensorprodofproj}
	We have $P_{\bm{i}}\cdot P_{\bm{j}}\cong P_{\bm{i}\bm{j}}$ and ${{}_{\bm{i}}P}\cdot{{}_{\bm{j}}P}\cong {{}_{\bm{ij}}P}$.
	In particular, $K_0(\sRproj)$ and $K_0(\projsR)$ are $\bbQ(q)$-algebras with multiplication given by tensor product.
\end{lem}
\begin{rem}The tensor product $\cdot$ provides a monoidal structure on $\sRproj$ with unit object $\mathbf{1}=P_{\emptyset}$.
	Moreover, $\sRcycRep$ is a right module category over $\sRproj$, see \cref{tensorprodofprojcyc} below.
	The same holds for $\projsR$ with $\mathbf{1}={}_\emptyset P$ and $\RepsRcyc$.
\end{rem}	
\begin{nota}
	For any object $\bm{i}$ in $\sR$ let $P^\level_{\bm{i}}\in\sRcycRep$ and ${}_{\bm{i}}P^\level\in\RepsRcyc$ be the corresponding projective module (in contrast to $P_{\bm{i}}\in\sRRep$ and ${}_{\bm{i}}P\in\RepsR$).
\end{nota}

Horizontal stacking of diagrams gives a morphism $\sRcyc\boxtimes\sR\to\sRcyc$.
Thus, given 
$M\in\sRcycRep$ and $N\in\sRRep$ we obtain $M\cdot N\coloneqq\ind_{\sRcyc\boxtimes\sR}^{\sRcyc}M\boxtimes N\in\sRcycRep$.
Similarly, for right modules.
The following is immediate from the definitions.
\begin{lem}\label{tensorprodofprojcyc} 
	We have $P^\level_{\bm{i}}\cdot P_{\bm{j}}\cong P^\level_{\bm{i}\bm{j}}$ and ${{}_{\bm{i}}P^\level}\cdot{{}_{\bm{j}}P}\cong {{}_{\bm{ij}}P^\level}$.
	In particular, the tensor product turns $K_0(\sRcycproj)$ into a right module for $K_0(\sRproj)$ and $K_0(\projsRcyc)$ into a right module for $K_0(\projsR)$.
\end{lem}

\begin{defi}
	For $\lambda\in\Par^\level$ let $P^\level_\lambda\coloneqq P^{\level}_{\bm{i}_{t^\lambda}}$ and ${}^\lambda\!P^\level\coloneqq{}_{\bm{i}_{t^\lambda}^\circledast}P$.
	From \cref{sRcycisquher}, we also get the \emph{left standard module} $\Delta_\lambda\in\sRcycRep$ and the \emph{right standard $\sRcyc$-module} ${}^\lambda\!\Delta\in\RepsRcyc$ defined as the respective quotients by all morphism which factor through some $P_\mu$ with $\mu<\la$.
\end{defi}

\cref{tensorprodofprojcyc} does in fact not require the level to be generic.
If it is however generic, then $P^\level_{\bm{i}}=0$ or we find $\la\in\Par^\level$ such that $P^\level_{\bm{i}}\cong P^\level_\lambda$, see \cref{isomorphismclassesofsRcycobjects}, similarly for right modules.
This observation should motivate the following:

\begin{lem}\label{basisK0cyc} The following sets each from a $\bbZ$-basis:
	\begin{align*}
		\text{ for }& K_0'(\sRcycproj):&\quad \{[\langle i\rangle P_\lambda^\level]\mid \la\in\Par, i\in\bbZ\},&&\{[\langle i\rangle\Delta_\lambda]\mid \la\in\Par, i\in\bbZ\},
		\\
		\text{for }& K_0'(\projsRcyc):&\quad\{[\langle i\rangle {}^\lambda\!P^\level]\mid \la\in\Par, i\in\bbZ\},&&\{[\langle i\rangle {}^\lambda\!\Delta]\mid \la\in\Par, i\in\bbZ\},
	\end{align*}
	and the sets
	$\{[P_\lambda^\level]\}$, $\{[\Delta_\lambda]\}$, $\{[{}^\lambda\!P^\level]\}$, $\{[{}^\lambda\!\Delta]\}$ with $\la\in\Par$ form $\bbQ(q)$-bases of $K_0$.	
\end{lem}
\begin{proof}
	The statements for $[P_\lambda^\level]$ follow directly from \cref{sRcycisquher}.
	By definition of $\Delta_\lambda$ and the upper-finite labelling set, the \enquote{base change} matrix is upper triangular with $1$'s on the diagonal and only finitely many non-zero entries in each row.
	Therefore, it is invertible and the $[\Delta_\lambda]$ form a basis as well.
	Alternatively, one could apply \cite[Theorem 8.3]{Brundanhwgr}.
	The same arguments work for right modules.
\end{proof}
\subsection{Bar involutions and pairings}
Given a $\gsVec$-category and $M,N\in\CRep$ we define $\HOM_{\CRep}(M,N)\coloneqq\bigoplus\Hom_{\CRep}(M\langle i\rangle,N)\in\gsVec$ which is the space of morphisms when $\CRep$ is viewed as a $\gsVec$-category.

For a graded (super)vector space $V=\oplus_{n\in\bbZ} V_n$ with $V_n=0$ for $ n\ll 0$ we let $\gdim(V)=\sum \operatorname{dim}V_nq^n\in\mathbb{N}[q^{-1}][[q]]$ be its graded dimension.  

\begin{defi}
	Define the grading-reversing contravariant functor 
	\[\overline{\phantom{u}}\colon \sRproj\to\projsR, \quad \overline{P}\coloneqq\HOM_{\sR}(P,\sR).\]
	It satisfies $\overline{P_{\bm{i}}\langle a\rangle}={{}_{\bm{i}}P}\langle -a\rangle$. It  descends to a functor $\overline{\phantom{u}}\colon \sRcycproj\to\projsRcyc$ which satisfies the analogous property on projectives $P^\level$.  	
	
	We also define $\overline{\phantom{u}}\colon \projsR^{(\level)}\to\sRproj^{(\level)}$ by the same formula. It  satisfies $\overline{{}_{\bm{i}}P\langle a\rangle} = P_{\bm{i}}\langle -a\rangle$ and descends again to the cyclotomic quotients.
	
	In particular, we have $\overline{\overline{P}}=P$ for any left or right $\sR$ or $\sRcyc$-module $P$.
	This is why we also call this functor \emph{Bar involution}.
\end{defi}

The following is immediate from the monoidality of $\overline{\phantom{u}}$.
\begin{lem}\label{compatibilitysrcycbarinv}
	The Bar involutions on $\sRcycproj$ and $\sRproj$ are compatible with the right module structure, that is $\overline{M\cdot N}\cong\overline{M}\cdot\overline{N}$ for $M\in\bothproj$, $N\in\sRproj$.
	The same holds true for the right $\sRproj$-module structure on $\projsRcyc$.
\end{lem}
\begin{defi}
	We define a $q$-bilinear pairing 
	\begin{align*}
		(\_,\_)\colon K_0(\projsRcyc)\otimes K_0(\sRcycproj)&\to\bbQ(q),\\ 
		[P']\otimes[P]&\mapsto \gdim(P'\otimes_{\sRcyc}P).
	\end{align*}	
\end{defi}
\begin{rem}
	The pairing $(\_,\_)$ is related to the $\HOM$ pairing as follows.
	Given $P$ and $Q\in\sRcycproj$ we have 
	\begin{equation*}
		\gdim\HOM_{\sRcyc}(P,Q)=(\overline{P},Q).
	\end{equation*}
	For two $P'$ and $Q'\in\projsRcyc$ we have 
	\begin{equation}\label{deltanabla}
		\gdim\HOM_{\sRcyc}(P',Q')=(Q',\overline{P'}).
	\end{equation}
\end{rem}
The next lemma essentially follows from \cref{ncsnake}, but we prove it to make sure that all the grading shifts agree.
\begin{lem}
	The bilinear form satisfies
	\begin{equation*}
		([P']\cdot[Q],[P])=([P'],[P]\cdot[\Sigma(Q)])
	\end{equation*}
\end{lem}
\begin{proof}
	We may assume that $P'={{}_{\bm{j}}P}$, $P={P_{\bm{i}}}$ and $Q={{}_{k}P}$.
	Then we have $P'\cdot Q= {{}_{\bm{j}k}P}$ and $P\cdot\Sigma(Q)=P\cdot P_{k+1}\langle-\epsilon\rangle = P_{\bm{i}k+1}\langle-\epsilon\rangle$.
	And thus,
	\begin{equation*}
		{{}_{\bm{j}k}P}\otimes_{\sRcyc} P_{\bm{i}} = \HOM_{\sRcyc}(\bm{i},\bm{j}k) 
		= \HOM_{\sRcyc}(\bm{i}k+1,\bm{j})\langle-\epsilon\rangle = {{}_{\bm{j}}P}\otimes_{\sRcyc} P_{\bm{i}k+1}\langle-\epsilon\rangle.\qedhere
	\end{equation*}
\end{proof}

\subsection{Relations in Grothendieck groups}
As preparation for the categorification results in the next section we calculate some crucial relations in $K_0(\sRproj)$ and $K_0(\projsR)$. For this we extend the parameters $b_{ij}$ from \cref{defel} to  $i,j\in\bbR$:

\begin{defi}\label{newbij} For  $i,j\in\bbR$ let $b_{ij}=-2$ if $j=i,i+1$, let $b_{ij}=0$ if $|i-j|\notin\bbZ$, and  set $b_{ij}=4\cdot\sgn(j-i)(-1)^{j-i}$ otherwise.
\end{defi}
\begin{prop}\label{relationsinH}
	In  $\sRproj$ and $\projsR$ we have for any $i\not=j\in\bbR$:
	\begin{align*}
		%\item $P_{iji}\langle1\rangle\oplus P_{iji}\langle-1\rangle = P_{\cdots i^{(2)}j\cdots}\langle3\rangle \oplus P_{\cdots ji^{(2)}\cdots}\langle-3\rangle \oplus P_{i}\langle2\rangle\oplus P_{i}\langle0\rangle$ if $j=i\pm1$.
		%\label{relp1left} 
		{P_{iji}}\langle1\rangle\oplus{P_{iji}}\langle-1\rangle &\cong {P_{iij}}\langle3\rangle\oplus{P_{jii}}\langle-3\rangle\oplus {P_{i}}\langle\epsilon+1\rangle\oplus {P_{i}}\langle\epsilon-1\rangle&&\text{if $j=i+1$},	\\	
		%\item\label{relm1left} 
		{P_{iji}}\langle1\rangle\oplus{P_{iji}}\langle-1\rangle& \cong {P_{iij}}\langle-3\rangle\oplus{P_{jii}}\langle3\rangle\oplus {P_{i}}\langle\epsilon+1\rangle\oplus {P_{i}P}\langle\epsilon-1\rangle&& \text{if $j=i-1$},\\
		%			\item\label{relp1right} $
		{{}_{iji}P}\langle1\rangle\oplus{{}_{iji}P}\langle-1\rangle &\cong {{}_{iij}P}\langle-3\rangle\oplus{{}_{jii}P}\langle3\rangle\oplus {{}_{i}P}\langle\epsilon-1\rangle\oplus {{}_{i}P}\langle-\epsilon-1\rangle&& \text{if $j=i+1$},\\
		%\item\label{relm1right} $
		{{}_{iji}P}\langle1\rangle\oplus{{}_{iji}P}\langle-1\rangle &\cong {{}_{iij}P}\langle3\rangle\oplus{{}_{jii}P}\langle-3\rangle\oplus {{}_{i}P}\langle1-\epsilon\rangle\oplus {{}_{i}P}\langle-1-\epsilon\rangle&&\text{if $j=i-1$},\\
		{P_{ij}}&\cong{P_{ji}}\langle b_{ij}\rangle \text{ and } {{}_{ij}P}\cong {{}_{ji}P}\langle -b_{ij}\rangle &&\text{otherwise}.	
	\end{align*}
\end{prop}
\begin{proof}
	\setlength{\dw}{5pt}
	The morphism \begin{tikzpicture}[line width=\lw, myscale=0.6]
		\draw (0,0) -- (1,1) (1,0)--(0,1);\node[fill=white, anchor=north] at (0,0) {$i$}; \node[fill=white, anchor=north] at (1,0) {$j$};\node[fill=white, anchor=south] at (0,1) {$j$}; \node[fill=white, anchor=south] at (1,1) {$i$};
	\end{tikzpicture} has degree $b_{ij} = -b_{ji}$.
	It defines homogeneous degree $0$ maps ${P_{ji}}\to{P_{ij}P}\langle b_{ji}\rangle$ and ${{}_{ij}P}\to{{}_{ji}P}\langle -b_{ij}\rangle$.
	Since both are isomorphisms by \cref{ncinverse} the first two claims follow.
	
	Of the remaining relations we will only prove the first one as they are all similar.
	For this let $j=i+1$ and consider 
	\begin{align*}
		B_1\colon &{P_{iij}}\langle3\rangle\oplus{P_{jii}}\langle-3\rangle\oplus {P_{i}}\langle1+\epsilon\rangle\oplus{P_{i}}\langle1+\epsilon\rangle\to{P_{iji}}\langle1\rangle\oplus{P_{iji}}\langle-1\rangle\\
		B_0\colon &{P_{iji}}\langle1\rangle\oplus{P_{iji}}\langle-1\rangle\to{P_{iij}}\langle3\rangle\oplus{P_{jii}}\langle-3\rangle\oplus {P_{i}}\langle1+\epsilon\rangle\oplus{P_{i}}\langle1+\epsilon\rangle
	\end{align*}
	given by the matrices
	\vspace{-4mm}
	\begin{equation*}
		B_1=\begin{pmatrix}
			-\begin{tikzpicture}[myscale=0.8, scale=0.3,line width=\lw]
				\node at (0,0) (A) {$\scriptstyle{i}$};
				\node at (1,0) (B) {$\scriptstyle{j}$};
				\node at (2,0) (C) {$\scriptstyle{i}$};
				\node at (0,3) (D) {$\scriptstyle{i}$};
				\node at (1,3) (E) {$\scriptstyle{i}$};
				\node at (2,3) (F) {$\scriptstyle{j}$};
				\draw (A.north)--(0,1.5)--(E.south);
				\draw (B.north)--(2,1.5)--(F.south);
				\draw (C.north)--(1,1.5)--(D.south);
			\end{tikzpicture}&
			\begin{tikzpicture}[myscale=0.8, scale=0.3,line width=\lw]
				\node at (0,0) (A) {$\scriptstyle{i}$};
				\node at (1,0) (B) {$\scriptstyle{j}$};
				\node at (2,0) (C) {$\scriptstyle{i}$};
				\node at (0,3) (D) {$\scriptstyle{j}$};
				\node at (1,3) (E) {$\scriptstyle{i}$};
				\node at (2,3) (F) {$\scriptstyle{i}$};
				\draw (A.north)--(1,1.5)--(F.south);
				\draw (B.north)--(0,1.5)--(D.south);
				\draw (C.north)--(2,1.5)--(E.south);
			\end{tikzpicture}&
			-\begin{tikzpicture}[myscale=0.8, scale=0.35,line width=\lw]
				\node at (0,0) (A) {$\scriptstyle{i}$};
				\node at (1,0) (B) {$\scriptstyle{j}$};
				\node at (2,0) (C) {$\scriptstyle{i}$};
				\node at (2,2.5) (D) {$\scriptstyle{i}$};
				\draw (A.north)..controls +(0.2,0.8) and 	+(-0.2,0.8)..(B.north);
				\draw (C)--(D);
			\end{tikzpicture}&
			0\\
			\begin{tikzpicture}[myscale=0.8, scale=0.35,line width=\lw]
				\node at (0,0) (A) {$\scriptstyle{i}$};
				\node at (1,0) (B) {$\scriptstyle{j}$};
				\node at (2,0) (C) {$\scriptstyle{i}$};
				\node at (0,3) (D) {$\scriptstyle{i}$};
				\node at (1,3) (E) {$\scriptstyle{i}$};
				\node at (2,3) (F) {$\scriptstyle{j}$};
				\draw (A.north)--(0,1.5)--(E.south);
				\draw (B.north)--(2,1.5)--(F.south);
				\draw (C.north)--(1,1.5)--(D.south) node[pos=0.8](a) {} ;
				\fill (a) circle(\dw);
			\end{tikzpicture}&
			-\begin{tikzpicture}[myscale=0.8, scale=0.35,line width=\lw]
				\node at (0,0) (A) {$\scriptstyle{i}$};
				\node at (1,0) (B) {$\scriptstyle{j}$};
				\node at (2,0) (C) {$\scriptstyle{i}$};
				\node at (0,3) (D) {$\scriptstyle{j}$};
				\node at (1,3) (E) {$\scriptstyle{i}$};
				\node at (2,3) (F) {$\scriptstyle{i}$};
				\draw (A.north)--(1,1.5)--(F.south);
				\draw (B.north)--(0,1.5)--(D.south);
				\draw (C.north)--(2,1.5)--(E.south) node[pos=0.8](a) {} ;
				\fill (a) circle(\dw);
			\end{tikzpicture}&
			0&
			-\begin{tikzpicture}[myscale=0.8, scale=0.35,line width=\lw]
				\node at (0,0) (A) {$\scriptstyle{i}$};
				\node at (1,0) (B) {$\scriptstyle{j}$};
				\node at (2,0) (C) {$\scriptstyle{i}$};
				\node at (2,2.5) (D) {$\scriptstyle{i}$};
				\draw (A.north)..controls +(0.2,0.8) and 	+(-0.2,0.8)..(B.north);
				\draw (C)--(D);
			\end{tikzpicture}\\
		\end{pmatrix},\qquad B_0=\begin{pmatrix}
			\begin{tikzpicture}[myscale=0.8, scale=0.35,line width=\lw]
				\node at (0,0) (A) {$\scriptstyle{i}$};
				\node at (1,0) (B) {$\scriptstyle{i}$};
				\node at (2,0) (C) {$\scriptstyle{j}$};
				\node at (0,3) (D) {$\scriptstyle{i}$};
				\node at (1,3) (E) {$\scriptstyle{j}$};
				\node at (2,3) (F) {$\scriptstyle{i}$};
				\draw (A.north)--(1,1.5)--(F.south);
				\draw (B.north)--node[pos=0.2](a) {}(0,1.5)--(D.south);
				\fill (a) circle (\dw);
				\draw (C.north)--(2,1.5)--(E.south);
			\end{tikzpicture}&
			\begin{tikzpicture}[myscale=0.8, scale=0.35,line width=\lw]
				\node at (0,0) (A) {$\scriptstyle{i}$};
				\node at (1,0) (B) {$\scriptstyle{i}$};
				\node at (2,0) (C) {$\scriptstyle{j}$};
				\node at (0,3) (D) {$\scriptstyle{i}$};
				\node at (1,3) (E) {$\scriptstyle{j}$};
				\node at (2,3) (F) {$\scriptstyle{i}$};
				\draw (A.north)--(1,1.5)--(F.south);
				\draw (B.north)--(0,1.5)--(D.south);
				\draw (C.north)--(2,1.5)--(E.south);
			\end{tikzpicture}\\
			\begin{tikzpicture}[myscale=0.8, scale=0.35,line width=\lw]
				\node at (0,0) (A) {$\scriptstyle{j}$};
				\node at (1,0) (B) {$\scriptstyle{i}$};
				\node at (2,0) (C) {$\scriptstyle{i}$};
				\node at (0,3) (D) {$\scriptstyle{i}$};
				\node at (1,3) (E) {$\scriptstyle{j}$};
				\node at (2,3) (F) {$\scriptstyle{i}$};
				\draw (A.north)--(0,1.5)--(E.south);
				\draw (B.north)--(2,1.5)--(F.south);
				\draw (C.north)--node[pos=0.2](a) {}(1,1.5)--(D.south);
				\fill (a) circle (\dw);
			\end{tikzpicture}&
			\begin{tikzpicture}[myscale=0.8, scale=0.35,line width=\lw]
				\node at (0,0) (A) {$\scriptstyle{j}$};
				\node at (1,0) (B) {$\scriptstyle{i}$};
				\node at (2,0) (C) {$\scriptstyle{i}$};
				\node at (0,3) (D) {$\scriptstyle{i}$};
				\node at (1,3) (E) {$\scriptstyle{j}$};
				\node at (2,3) (F) {$\scriptstyle{i}$};
				\draw (A.north)--(0,1.5)--(E.south);
				\draw (B.north)--(2,1.5)--(F.south);
				\draw (C.north)--(1,1.5)--(D.south);
			\end{tikzpicture}\\
			\begin{tikzpicture}[myscale=0.8, scale=0.35,line width=\lw]
				\node at (0,-1) (A) {$\scriptstyle{i}$};
				\node at (0,1.5) (B) {$\scriptstyle{i}$};
				\node at (1,1.5) (C) {$\scriptstyle{j}$};
				\node at (2,1.5) (D) {$\scriptstyle{i}$};
				\draw (C.south)..controls +(0.2,-0.8) and 	+(-0.2,-0.8)..(D.south);
				\draw (A)--(B);
			\end{tikzpicture}&
			0\\
			0&\begin{tikzpicture}[myscale=0.8, scale=0.35,line width=\lw]
				\node at (0,-1) (A) {$\scriptstyle{i}$};
				\node at (0,1.5) (B) {$\scriptstyle{i}$};
				\node at (1,1.5) (C) {$\scriptstyle{j}$};
				\node at (2,1.5) (D) {$\scriptstyle{i}$};
				\draw (C.south)..controls +(0.2,-0.8) and 	+(-0.2,-0.8)..(D.south);
				\draw (A)--(B);
			\end{tikzpicture}\\
		\end{pmatrix}.
		\vspace{2mm}
	\end{equation*}
	Note that all the entries are homogeneous and provide two degree zero maps.
	They are mutually inverses by \cref{BoBone} in \Cref{BoBonesec}.
\end{proof}	

\section{Categorification Theorems}
In this section we finally apply our results to deduce some categorification results.
\subsection{A categorification of the $q$-\texorpdfstring{$\epsilon$lectrical}{electrical} algebra}
The first result is  the following \emph{$q$-$\epsilon$lectric Categorification Theorem} analogous to \cite{KL09}, \cite{Rouquier}:
\begin{samepage}
	\begin{thm}\label{categorification}
		There are $\bbQ(q)$-linear algebra isomorphisms 
		\begin{align*}
			\Phiq\colon \el&\to K_0(\sRZproj) & \Phiqq\colon \elinv&\to K_0(\projsRZ),\\
			\cE_i&\mapsto [{P_i}],&\cE_i&\mapsto [{{}_iP}].
		\end{align*}
	\end{thm}	
	An important step in the proof is to establish the well-definedness of the maps.
\end{samepage}
\begin{proof}[Proof of q-electric Categorification Theorem]
	By \cref{relationsinH} the assignments extend to a well-defined algebra homomorphism.
	Recall from \cref{grel} that the algebra $\el$ is a filtered with $\cE_{i_1}\cdots \cE_{i_k}$ in filtration degree $k$.
	On the other hand $\sRZproj$ is a filtered category in the sense of \cite{FLP23}*{Section 4.3}, where ${P_{i_1\dots i_k}}$ sits in filtration degree $k$.
	This induces a filtration on $K_0(\sRZproj)$ so that $\Phiq$ is actually a morphism of filtered algebras.
	We obtain a commutative diagram in vector spaces with vertical isomorphisms: 		
	\begin{equation*}
		\begin{tikzcd}
			\el\arrow[r, "\Phiq"]\arrow[d, "\gr"]&K_0(\sRZproj)\arrow[d, "\gr"]\\
			\gr \el\arrow[r, "\gr\Phiq"]&\gr K_0(\sRZproj)
		\end{tikzcd}
	\end{equation*}
	Thus, it suffices to show that $\gr\Phiq$ is an isomorphism.
	Now by \cite{FLP23}*{Theorem 4.19} we know that $\gr K_0(\sRZproj)\cong K_0(\gr \sRZproj)$.
	The category $\gr \sRZproj$ arises by quotienting out everything that factors through a lower filtration degree.
	In our case this means that we kill every cup and cap.
	From the defining relations \cref{ncdotcap}-\cref{ncbraid} we see that $\gr(\sRZproj)$ is equivalent to $R\mathrm{-proj}$ from \cite{KL09} if we ignore the $\bbZ$-grading.
	On the other hand, the algebra $\gr \el$ is by \cref{grel} up to a different $q$-shifts exactly the algebra $\bm{\mathrm{f}}$ from \cite{KL09}.
	One quickly checks that the $q$-shifts match the different grading.
	Then, the statement follows from \cite{KL09}*{Theorem~1.1}.
\end{proof}

\subsection{Categorified involutions}		
\begin{thm}[Compatibilities Theorem]\label{comp}	
	The following diagrams commute:
	\begin{equation*}
		\begin{tikzcd}[baseline = (W.base)]
			\el\arrow[r, "\Phiq"]\arrow[d, "\overline{\phantom{U}}"]&|[alias=W]|K_0(\sRZproj)\arrow[d, "\overline{\phantom{U}}"]\\
			\elinv\arrow[r, "\Phiqq"]&K_0(\projsRZ)		
		\end{tikzcd},\qquad\qquad
		\begin{tikzcd}[baseline = (W.base)]
			\el\arrow[r, "\Phiq"]\arrow[d, "\sigma"]&|[alias=W]|K_0(\sRZproj)\arrow[d, "\Sigma"]\\
			\elinv\arrow[r, "\Phiqq"]&K_0(\projsRZ)
		\end{tikzcd}.
	\end{equation*}
\end{thm}
\begin{proof}
	For the commutative diagrams, it suffices to check the claim for $\cE_i$.
	We have $\overline{P_i}={}_iP$, whence the left diagram commutes. Similarly, $\Sigma({P_i})={}_{i+1}P\langle-\epsilon\rangle$.
\end{proof}
The classes $[P_{\bm{i}}]$ and $[{}_{\bm{i}}P]$  provide a \emph{canonical basis} of $\el$ respectively $\elinv$. 
\begin{rem}
	Alternatively one could work directly with the additive closure of the Karoubian closure  of $\sR(\bbZ)$ and take its $K_0$. 
	Then the bar involution is categorified via the functor in \cref{confusing}, whereas  $\sigma$ and $\tau$ from \cref{sigmadef,taudef} 
	are categorified by $\Sigma$ and $\Tau$ respectively from \cref{barinvsigmacat}.  
\end{rem}

\subsection{Categorification of the $q$-\texorpdfstring{$\epsilon$lectrical}{electrical} Fock spaces}
We next categorify  the $\level=1$ (dual) Fock space of charge zero. We show  that the right $\sRZproj$-module structure on $\sRZcycproj$ categorifies the right action of $\el$ on $\Fockzero$, similarly for the right $\projsRZ$-module structure on $\projsRZcyc$ and  the action of $\elinv$ on $\dFockzero$.  \begin{thm}[Fock space categorification]\label{leftactioncat}	
	The $q$-linear map
	\begin{equation*}
		\Psi\colon\Fockzero\to K_0(\sRZcycproj),\;v_\lambda\mapsto[\Delta_\lambda]
	\end{equation*}
	is an isomorphism and the following diagram is commutative:
	\begin{equation*}
		\begin{tikzcd}
			\Fock\otimes \el\arrow[d, "\Psi \otimes\Phi_\charge"]\arrow[r]&\Fock\arrow[d,"\Psi"]\\
			K_0(\sRZcycproj)\otimes K_0(\sRZproj)\arrow[r]&K_0(\sRZcycproj)
		\end{tikzcd}
	\end{equation*}
\end{thm}
\begin{proof}
	The first part is obvious from \cref{basisK0cyc}.
	For the second part we need to compute $[\Delta_\lambda][P_i]$.
	The module $\Delta_{\lambda}$ has a basis given by $\XX_{\ta{t}^\lambda}^{\ta{s}}$.
	The module $P_{\lambda}$ has a basis given by all $\XX_{\ta{t}}^{\ta{s}}$, where $\res^l\ta{t} = \bm{i}_\lambda$.
	Additionally, $P_{\lambda}.P_i$ has a basis given by all $\XX_{\ta{t}}^{\ta{s}}$ where $\res^l\ta{t} = \bm{i}_\lambda i$.
	Therefore, $\Delta_{\lambda}.P_i$ has a basis given by $\XX_{\ta{t}}^{\ta{s}}$ where $\res^l\ta{t} = \bm{i}_\lambda i$ but only in the last step of $\ta{t}$ a box might be removed.
	Now all such $\XX_{\ta{t}}^{\ta{s}}$ where the last step removes a box form a submodule of $\Delta_{\lambda}.P_i$.
	This is nonzero if and only if $\lambda$ has a removable box $\bbox$ of content $i-1$ in which case it is isomorphic to $\Delta_{\lambda\ominus\bbox}\langle d\rangle$, where $d$ is the degree of the diagram
	\begin{tikzpicture}[line width = \lw, myscale=0.6]
		\draw (0,0)node[anchor=north] {$i_1$}--(0,1.5) (1,0) node[anchor=north] {$i_{l-1}$}--(1,1.5)(3,0) node[anchor=north] {$i_{l+1}$}--(3,1.5)(4,0) node[anchor=north] {$i_k$}--(4,1.5);
		\draw (2,0) node[anchor=north] {$i-1$} ..controls +(0.2,0.8) and +(-0.2,0.8)..(5,0) node[anchor=north] {$i$};
		\node at (0.5,0.75) {$\dots$};
		\node at (3.5,1) {$\dots$};
	\end{tikzpicture} and if $\bbox$ was in row $r$, then $l=\lambda_1+\dots+\lambda_r$.
	The quotient of $\Delta_{\lambda}.P_i$ by this submodule is nonzero if and only if $\lambda$ admits an addable box $\beta$ of residue $i$, in which case this quotient is isomorphic to $\Delta_{\lambda\oplus\beta}\langle d'\rangle$, where $d'$ is the degree of the diagram
	\begin{tikzpicture}[line width = \lw, myscale=0.6]
		\draw (0,0)node[anchor=north] {$i_1$}--(0,1.5) (1,0) node[anchor=north] {$i_l$}--(1,1.5)(2,0) node[anchor=north] {$i_{l+1}$}--(2,1.5)(3,0) node[anchor=north] {$i_k$}--(3,1.5);
		\draw (4,0) node[anchor=north] {$i$}--(1.5,1.5);
		\node at (0.5,0.75) {$\dots$};
		\node at (2.5,0.5) {$\dots$};
	\end{tikzpicture} and if $\beta$ is added in row $r$, then $l=\lambda_1+\dots+\lambda_r$.
	
	It remains to check that $d$ and $d'$ give the same degree shifts as the $K$ appearing in the comultiplication of $\cE_i$.
	Observe that we act (by our implicit choice of tensor product) on $v_{\lambda_j+1-j}$ by $K_{\beta_i}$ if $1-j$ is even and by $K_{\beta_i'}$ if $1-j$ is odd.
	This means that we get $q^0$ contribution for every even $\lambda_j$, a $q^{(-1)^i4}$ contribution for even $1-j$ and odd $\lambda_j$ and $q^{(-1)^{i+1}4}$ contribution for odd $1-j$ and odd $\lambda_j$.
	In other words every even $\lambda_j$ gives a $q^0$ contribution and every odd $\lambda_j$ gives a $q^{(-1)^{i+1-j}4}$ contribution.
	On the other hand, observe that the crossings swap $i$ with rows $\lambda_{k}$, $\lambda_{k-1}$ until $\lambda_{r+1}$ (if $\lambda$ has $k$ rows).
	Now if $\lambda_j$ is even, swapping with this row gives degree $0$.
	If $\lambda_j$ is odd, as in the even case, consecutive pairs of crossings cancel in their degree, and we are left with the degree of \tikz[line width=\lw, myscale=0.6]{\draw (0,0) -- (1,1) (1,0)--(0,1);\node[fill=white, anchor=north] at (0,0) {$1-j$}; \node[fill=white, anchor=north] at (1,0) {$i$};}.
	This has exactly degree $4(-1)^{i+1-j}$ as $i>1-j$ and the diagram in the theorem commutes.
\end{proof}
\begin{thm}[Dual Fock space categorification]\label{rightactioncat}The $q$-linear map 
	\begin{equation*}
		\Psi'\colon\dFockzero\to K_0(\projsRZcyc),\;v^\lambda\mapsto[{}^\lambda\!\Delta]
	\end{equation*}
	is an isomorphism and the following diagram is commutative:
	\begin{equation*}
		\begin{tikzcd}
			\dFockzero\otimes \elinv\arrow[d, "\Psi'\otimes\Phiqq"]\arrow[r]&\dFockzero\arrow[d,"\Psi'"]\\
			K_0(\projsRZcyc)\otimes K_0(\projsRZ)\arrow[r]&K_0(\projsRZcyc).
		\end{tikzcd}
	\end{equation*}
\end{thm}
\begin{proof}
	This is similar to the left module version from \cref{leftactioncat}.
\end{proof}
\begin{prop}[Compatibility with bar involution]
	The following diagram commutes.
	\begin{equation*}
		\begin{tikzcd}
			\Fock\arrow[r, "\overline{\phantom{u}}"]\arrow[d, "\Psi'"]&\dFock\arrow[d,"\Psi"]\\
			K_0(\sRZcycproj)\arrow[r, "\overline{\phantom{u}}"]&K_0(\projsRZcyc).
		\end{tikzcd}
	\end{equation*}
\end{prop}
\begin{proof}
	The vector space $\Fock$ is generated by $v_\emptyset$ as an $\el$-module.
\end{proof}
\begin{rem}
The canonical basis of $\elinv$ resp. $\elinv$ induces a canonical basis of $\Fock$ and $\dFock$. They  correspond to the classes $[P^\level_{\bm{i}}]$  and $[{}_{\bm{i}}P^\level]$ respectively.
\end{rem}
\begin{prop}[Compatibility with pairing]
	We have $(w,v)=(\PhiqF'(w),\PhiqF(v))$ for all $w\in\dFock$, $v\in\Fock$.
\end{prop}
\begin{proof}
	It suffices to check that $([{}^\lambda\Delta], [\Delta_\mu])=\delta_{\lambda\mu}$.
	But this is immediate from \cref{sRcycisquher} and \cref{deltanabla} using that projective $\sRcyc$-modules have $\Delta$-flags and the Ext-vanishing, see  \cite[Theorem 3.14]{BS21}, between $\Delta$'s and $\nabla$'s.
\end{proof}
\subsection{Universal categorification and higher level Fock spaces}	
To incorporate $\chargevec$ and  higher level Fock spaces we work now with $\sR(\bbR)$ instead of  $\sR(\bbZ)$. 
\begin{defi}
	The \emph{universal electric algebra} $\ground$-algebra $\el(\bbR)$  is generated by $\cE_i$, $i\in\bbR$, with relations \cref{electricijdistant}, \cref{electricipone}, \cref{electricimone}  using \cref{newbij}. 
\end{defi}	
For a fixed level $\level$ and a generic charge vector $\chargevec$, see \cref{defgeneric}, set  $\bbR(\Rld)=\bigcup_{j=1}^\level (\delta_j+\bbZ1)\subset\bbR$. We consider the full monoidal supersubcategory $\sR(\Rld)=\sR(\bbR(\ell,\chargevec))$ of $\sR(\bbR)$ with objects sequences of elements in  $\bbR(\Rld)$. We also let $\sR^\level(\chargevec)$ be the associated level $\level$ cyclotomic quotient. 

Denote by $\el(\Rld)$ the $\ground$-algebra generated by the $\cE_i$ for $i\in \bbR(\Rld)$.   In particular, $\el(1,\chargevec)=\el$ if $\delta_1\in\bbZ$ and $\el(\Rld)\cong{\el}\otimes\cdots\otimes\el$, the $\level$-fold tensor product of $\el$, since $\chargevec$ is generic. Similarly define $\elinv(\Rld)$ with $\elinv(\Rld)\cong{\elinv}\otimes\cdots\otimes\elinv$. \\

With these definitions we obtain as in \cref{categorification} directly the following: 

\begin{thm}[Universal categorification]\label{univcategorification}
	There are algebra isomorphisms 
	\begin{align*}
		\Phiq\colon \el(\Rld)&\to K_0(\sRdproj) & \Phiqq\colon \elinv(\Rld)&\to K_0(\projsRd),\\
		\cE_i&\mapsto [{P_i}],&\cE_i&\mapsto [{{}_iP}].
	\end{align*}
\end{thm}	
Recall from \cref{lFock} the higher level Fock space $\Fockzero_{\chargevec,\level}$.  
\begin{thm}[Higher level Fock space categorification]\label{higherleftactioncat}	
	The $q$-linear map
	\begin{equation*}
		\Psi_\level\colon\Fockzero_{\chargevec,\level}		
		\to K_0(\sR^\level(\chargevec)),\;v_\lambda\mapsto[\Delta_\lambda]
	\end{equation*}
	is an isomorphism and the following diagram is commutative:
	\begin{equation*}
		\begin{tikzcd}
			\Fockzero_{\chargevec,\level}\otimes \el(\Rld)\arrow[d, "\Psi_\level \otimes\Phiq"]\arrow[r]&\Fockzero_{\chargevec,\level}	\arrow[d,"\Psi_\level"]\\
			K_0(\sRcycdproj)\otimes K_0(\sRdproj)\arrow[r]&K_0(\sRcycdproj).
		\end{tikzcd}
	\end{equation*}
\end{thm}	
\begin{proof}The proof of \cref{leftactioncat} can be just copied. 
\end{proof}
Similar to  \cref{lFock} there is the  higher level dual Fock space 
\[\dFockzero_{\chargevec,\level}=  
\dFockzero_{\delta_1}\otimes \cdots \otimes\dFockzero_{\delta_\level}\] of level $\level$ and charge $\chargevec$.  \cref{rightactioncat} directly generalises to the following

\begin{thm}[Higher level dual Fock space categorification]\label{higherrightactioncat}The $q$-linear map 
	\begin{equation*}
		\Psi_\level'\colon\dFockzero_{\delta_1}\otimes \cdots \otimes\dFockzero_{\delta_\level}\to K_0(\projsRcycd),\;v^\lambda\mapsto[{}^\lambda\!\Delta]
	\end{equation*}
	is an isomorphism and the following diagram is commutative:
	\begin{equation*}
		\begin{tikzcd}
			\dFockzero\otimes \elinv(\Rld)\arrow[d, "\Psi_\level'\otimes\Phiqq"]\arrow[r]&\dFockzero\arrow[d,"\Psi_\level'"]\\
			K_0(\projsRcycd)\otimes K_0(\projsR)\arrow[r]&K_0(\projsRcycd).
		\end{tikzcd}
	\end{equation*}
\end{thm}	
\begin{proof}The proof of \cref{rightactioncat} can be just copied. 
\end{proof}
\begin{rem}
	We consider in this article only generic charge vectors, see \cref{defgeneric}. 
	This allows to distinguish the components of a multi-partition. 
	In fact, the charge uniquely determines the corresponding component and the combinatorics of different components do not interact with each other. 
	Correspondingly, the factors in the $\Fockzero_{\delta_1}\otimes \cdots \otimes\Fockzero_{\delta_\level}$ and $\dFockzero_{\delta_1}\otimes \cdots \otimes\dFockzero_{\delta_\level}$ are independent in the sense that $\el(\Rld)$ respectively $\elinv(\Rld)$ act componentwise.  
\end{rem}
\begin{rem}Via \cref{isoncvw} we could alternatively use modules over the super Brauer algebras for the categorification of Fock spaces and thus also categories of representations of the  periplectic Lie superalgebras. By \cite[Corollary~7.3.2]{Coulembierpn}, the categories of finite dimensional representations of $\mathfrak{p}(n)$  are equivalent to a subquotient category of the  
categories categorifying the Fock spaces. Each  $\mathfrak{p}(n)$  corresponds to a layer in a filtration on $\Fock$. For another realization in terms of periplectic Khovanov algebras see \cite[Theorem~6.6]{NehmeKhovanovpn}.
\end{rem} 
\begin{rem}
One might want to define and study more involved arbitrary higher level Fock spaces generalizing work of Uglov, \cite{Uglov}, to the electric Lie algebra setting.  We expect that these can be categorified using parabolic category $\cO$ for the periplectic Lie superalgebras. For this parabolic  category $\cO$ needs to be revisited and studied in more detail first extending e.g.~the works  \cite{ChenCoulembier}, \cite{ChenPeng}. 
\end{rem}

\appendix
\section{Proofs of some results from \cref{quantumelectrical}}
In this section we collect some technical proofs of statements from \cref{quantumelectrical}.
\subsection{Proof of \cref{lbHopfalgebra}}
\begin{proof}\label{prooflbHopf}
	All the maps clearly satisfy the Hopf algebra conditions if we show that they are well-defined, i.e.~compatible with the relations. For $\eps$, this is a straight-forward calculation which is omitted. For $\Delta$, the compatible with \eqref{Borelmkk} and \eqref{Borelminverse} is obvious.  
	For \eqref{Borelmfk}, we calculate
	\begin{align*}
		\Delta(K_\lambda)\Delta(F_i) &= K_\lambda\otimes K_\lambda F_i + K_\lambda F_i\otimes K_\lambda K_{\beta_i} \\
		&= q^{\langle\lambda,\alpha_i^\vee\rangle}(K_\lambda \otimes F_iK_\lambda + F_iK_\lambda\otimes K_{\beta_i}K_\lambda)
		= q^{\langle\lambda,\alpha_i^\vee\rangle}\Delta(F_i)\Delta(K_\lambda).
	\end{align*}
	For  \eqref{Borelmdistant}, we assume $\abs{i-j}>1$ and compute 		
	\begin{align*}
		&\Delta(F_i)\Delta(F_j) = 1\otimes F_iF_j + F_j\otimes F_iK_{\beta_j} + F_i\otimes K_{\beta_i}F_j + F_iF_j\otimes K_{\beta_i+\beta_j}\\
		&= q^{b_{ij}}1\otimes F_jF_i + q^{\langle \beta_j, \alpha_i^\vee\rangle}F_j\otimes K_{\beta_j}F_i + q^{-\langle \beta_i, \alpha_j^\vee\rangle}F_i\otimes F_jK_{\beta_i} + q^{b_{ij}}F_jF_i\otimes K_{\beta_j+\beta_i}\\
		&= q^{b_{ij}}\Delta(F_j)\Delta(F_i).
	\end{align*}
	Here we used that $b_{ij}=-b_{ji}$ if $\abs{i-j}>1$, see \cref{shiftinv}.
	
	Of the remaining Serre relations \eqref{Borelmserre1} and  \eqref{Borelmserre2}  we only consider one, since the arguments are similar. We calculate the parts: 
	\begin{align*}
		\Delta(F_i^2F_{i+1}) =\;& {1\otimes F_i^2F_{i+1}} + {F_{i+1}\otimes F_i^2K_{\beta_{i+1}}} + {F_i\otimes F_iK_{\beta_i}F_{i+1}} 
		+ {F_i\otimes K_{\beta_i}F_iF_{i+1}} \\
		& +{F_iF_{i+1}\otimes F_iK_{\beta_i}K_{\beta_{i+1}}} + {F_iF_{i+1}\otimes K_{\beta_i}F_iK_{\beta_{i+1}}} 
		+ {F_i^2\otimes K_{2\beta_i}F_{i+1}} \\
		&+ {F_i^2F_{i+1}\otimes K_{2\beta_i+\beta_{i+1}}}\\
		=\;&{1\otimes F_i^2F_{i+1}} + {F_{i+1}\otimes F_i^2K_{\beta_{i+1}}} + {(q^{-4}+q^{-2})F_i\otimes F_iF_{i+1}K_{\beta_i}} \\
		&+{(1+q^{2})F_iF_{i+1}\otimes F_iK_{\beta_i}K_{\beta_{i+1}}} + {q^{-8}F_i^2\otimes F_{i+1}K_{2\beta_i}} \\
		& + {F_i^2F_{i+1}\otimes K_{2\beta_i+\beta_{i+1}}}.\\
		\Delta(F_iF_{i+1}F_i) =\;
		& {1\otimes F_iF_{i+1}F_i} + {F_i\otimes F_iF_{i+1}K_{\beta_i}} + { F_{i+1} \otimes F_iK_{\beta_{i+1}}F_i} \\
		&+ {F_i \otimes K_{\beta_i}F_{i+1}F_i} +{ F_{i+1}F_i\otimes F_iK_{\beta_{i+1}}K_{\beta_i}} + {F_i^2\otimes K_{\beta_i}F_{i+1}K_{\beta_i}} \\
		&+ {F_iF_{i+1} \otimes K_{\beta_i}K_{\beta_{i+1}}F_i}+ {F_i^2F_{i+1}\otimes K_{2\beta_i+\beta_{i+1}}}\\
		=\;& {1\otimes F_iF_{i+1}F_i} + {F_i\otimes F_iF_{i+1}K_{\beta_i}} + { q^2 F_{i+1} \otimes F_i^2K_{\beta_{i+1}}} \\
		&+ {q^{-2} F_i \otimes F_{i+1}F_iK_{\beta_i}} +{ F_{i+1}F_i\otimes F_iK_{\beta_{i+1}}K_{\beta_i}} + {q^{-4}F_i^2\otimes F_{i+1}K_{2\beta_i}} \\
		&+ {q^4 F_iF_{i+1} \otimes F_iK_{\beta_i}K_{\beta_{i+1}}}+ {F_i^2F_{i+1}\otimes K_{2\beta_i+\beta_{i+1}}}.\\
		\Delta(F_{i+1}F_i^2) =\;& 
		{1\otimes F_{i+1}F_i^2} + { F_i\otimes F_{i+1}F_iK_{\beta_i}} + { F_i \otimes F_{i+1}K_{\beta_i}F_i}+ {F_{i+1} \otimes K_{\beta_{i+1}}F_i^2} \\
		&+ {F_i^2\otimes F_{i+1}K_{2\beta_i}} +{F_{i+1} F_i\otimes K_{\beta_{i+1}}F_iK_{\beta_i}} \\
		&+ {F_{i+1}F_i \otimes K_{\beta_{i+1}}K_{\beta_i}F_i} + {F_{i+1}F_i^2\otimes K_{2\beta_i+\beta_{i+1}}}\\
		=\;& {1\otimes F_{i+1}F_i^2} + {(1+q^2) F_i\otimes F_{i+1}F_iK_{\beta_i}} + { q^4 F_{i+1} \otimes K_{\beta_{i+1}}F_i^2}\\
		&+ {F_i^2\otimes F_{i+1}K_{2\beta_i}} +{(q^2+q^4)F_{i+1} F_i\otimes F_iK_{\beta_i+\beta_{i+1}}} \\
		&+ {F_{i+1}F_i^2\otimes K_{2\beta_i+\beta_{i+1}}}
	\end{align*}

	Now the first  terms from each term give zero thanks to the Serre relation in the second tensor factor. The same for the last term thanks to the Serre relation in the first tenor factor. But then also all other terms cancel (remember to multiply the three cases by $q^{3}$,  $-[2]$ and $q^{-3}$ respectively!). This shows that  $\Delta$ is well-defined.
	For $S$ we compute
	\begin{equation*}
		S(K_\lambda F_i) = S(F_i)S(K_\lambda ) = -F_iK_{-\beta_i}K_{-\lambda} = -q^{-\langle \lambda, \alpha_i^\vee\rangle}K_{-\lambda}F_iK_{\beta_i} = S(q^{-\langle\lambda,\alpha_i^\vee\rangle}F_iK_\lambda),
	\end{equation*}
	\begin{equation*}
		S(F_iF_j) = F_jK_{-\beta_j}F_iK_{-\beta_i} = q^{b_{ij}+b_{ji}-b_{ji}}F_iK_{-\beta_i}F_jK_{-\beta_j} = S(q^{b_{ij}}F_jF_i),
	\end{equation*}
	For the Serre relations \eqref{Borelmserre1}, and similarly for \eqref{Borelmserre2},  we calculate
	\begin{align*}
		&S(q^{3}F_i^2F_{i+1}-[2]F_iF_{i-1}F_i+q^{-3}F_{i+1}F_i^2) \\
		=& -q^{3}F_{i+1}K_{\beta_{i+1}}(F_iK_{\beta_i})^2+[2]F_iK_{\beta_i}F_{i+1}K_{\beta_{i+1}}F_iK_{\beta_i}-q^{-3}(F_iK_{\beta_i})^2F_{i+1}K_{\beta_{i+1}}\\
		=&-q^{3-6}F_{i+1}F_i^2K_{2\beta_i+\beta_{i+1}}+[2]F_iF_{i+1}F_iK_{2\beta_i+\beta_{i+1}}-q^{-3+6}F_i^2F_{i+1}K_{2\beta_i+\beta_{i+1}} = 0.
	\end{align*}
	And therefore, $S$ is also well-defined and \cref{lbHopfalgebra} is proven.
\end{proof}	

\subsection{Proof of \cref{Kbalanced}}\label{proofKbalanced}
\begin{proof}
	The statement is clear for $\Delta$ and $\eps$.
	For $S$, it suffices to show that
	\begin{equation*}
		S(aK_\lambda\otimes b) = S(a\otimes K_\lambda b)
	\end{equation*}
	holds in $\Borelm\otimes_\Cartan\Borelp$ for any $a\in\Borelm$, $b\in\Borelp$, $\la\in X$. 
	
	By definition of $S$ in \cref{corHopf} we get that 
	\begin{equation}\label{S1}
		S(aK_\lambda\otimes b) = (1\otimes S(b))(S(a K_\lambda)\otimes 1)= (1\otimes S(b))(K_{-\lambda}\otimes 1)(S(a)\otimes 1)
	\end{equation}
	since $S$ is an antipode on the factors and $\Borelm$ and $\Borelp$ are subalgebras.
	Similarly, 
	\begin{equation}\label{S2}
		S(a\otimes K_\lambda b) = (1\otimes S(K_\lambda b))(S(a)\otimes 1) = (1\otimes S(b))(1\otimes K_{-\lambda})(S(a)\otimes 1).
	\end{equation}
	Since \cref{S1}=\cref{S2} in $\Borelm\otimes_\Cartan\Borelp$, we showed that $S$ is $\Cartan$-balanced.
	
	It remains to consider the multiplication.
	It is $\Cartan$-balanced if the equalities
	\begin{align}\label{mbalance}
		(1\otimes K_\lambda)(a\otimes 1) = (K_\lambda a\otimes 1)\quad\quad\text{and}\quad\quad (1\otimes b)(K_\lambda\otimes 1) &= (1\otimes bK_\lambda)
	\end{align}
	hold in $\Borelm\otimes_\Cartan\Borelp$ for any $a\in\Borelm$, $b\in\Borelp$, $\la\in X$.
	By linearity, it suffices to assume $a=K\overline{a}$ for some $K\in\Cartan$ and some monomial $\overline{a}=F_{i_1}\cdots F_{i_r}$ in the $F_i$s.
	Note that then the term $\langle a'_{(3)}, b_{(3)}\rangle$ in \cref{mult} can only get nonzero contributions from monomial summands in $a'_{(3)}$ which are contained in $\Cartan$, i.e.~contain no $F_i$s.
	Similarly, for $a'_{(1)}$ using the term $\langle S^{-1}(a'_{(1)}), b_{(1)}\rangle$.
	By the definition of $\Delta$ in \Cref{lbHopfalgebra} this implies that only $a'_{(1)}=K$, $a'_{(2)}=K\overline{a}$, $a'_{(3)}=K\prod_{j=1}^rK_{\beta_{i_j}}$ contributes.
	Thus, 
	\begin{equation*}
		(1\otimes K_\lambda)(a\otimes 1)=\langle K^{-1},K_\la\rangle K\overline{a}\otimes K_\la\langle K\prod_{j=1}^rK_{\beta_{i_j}},K_\la\rangle=
		\langle\prod_{j=1}^rK_{\beta_{i_j}},K_\la\rangle a\otimes K_\la.
	\end{equation*}
	This simplifies in $\Borelm\otimes_\Cartan\Borelp$to $q^ca\otimes K_\la=q^caK_\la\otimes 1=q^cK\overline{a}K_\la\otimes 1$ 
	with $c={\sum_{j=1}^r(\beta_{i_j},\lambda)}$.
	But $K\overline{a}K_\la$ is by \cref{Borelmkk} and \cref{Borelmfk} equal to $q^dK_\la a\otimes 1$ where $d=\sum_{j=1}^r\langle \la,\alpha_{i_j}\rangle=-\sum_{j=1}^r(\beta_{i_j},\lambda)=-c$.
	Thus, the first equality in \cref{mbalance} holds.
	The second can be shown analogously using $E_i$s instead of $F_i$s.
	Therefore, the multiplication is $\Cartan$-balanced, and $\Borelm\otimes_\Cartan\Borelp$ is a Hopf algebra.
\end{proof}

\subsection{Proof of \cref{elcoideal}}\label{proofcoideal}
\begin{proof}To prove \cref{elcoideal} we need to show that we get a well-defined injective algebra homomorphism.
	We first check consistency with the relations of $\el$.
	
	For \cref{electricijdistant} we have
	\begin{align*}
		&\operatorname{j}(\cE_i)\operatorname{j}(\cE_j)=
		(F_i+q^{\epsilon-1}E_{i-1}K_{-\alpha_{i-1}})(F_j+q^{\epsilon-1}E_{j-1}K_{-\alpha_j}) \\
		=\;\;& F_iF_j + q^{\epsilon-1} E_{i-1}K_{-\alpha_{i-1}}F_j + q^{\epsilon-1} F_iE_{j-1}K_{-\alpha_j} + q^{2\epsilon+2}E_{i-1}K_{-\alpha_{i-1}}E_{j-1}K_{-\alpha_j}\\
		=\;\;&\begin{multlined}[t]q^{b_{ij}}F_jF_i + q^{\epsilon-1-b_{i-1,j}}F_jE_{i-1}K_{-\alpha_{i-1}} + q^{\epsilon-1+b_{j-1,i}}E_{j-1}K_{-\alpha_j}F_i\\+ q^{2\epsilon+2+b_{i-1,j-1}}E_{j-1}K_{-\alpha_j}E_{i-1}K_{-\alpha_{i-1}}\end{multlined}\\
		=\;\;&\begin{multlined}[t]q^{b_{ij}}F_jF_{i} + q^{\epsilon-1+b_{ij}}F_jE_{i-1}K_{-\alpha_{i-1}} + q^{\epsilon-1+b_{ij}}E_{j-1}K_{-\alpha_j}F_i \\+ q^{2\epsilon-2+b_{i,j}}E_{j-1}K_{-\alpha_j}E_{i-1}K_{-\alpha_{i-1}}\end{multlined}\\\
		=\;\;&q^{b_{ij}}(F_j+q^{\epsilon-1}E_{j-1}K_{-\alpha_j})(F_i+q^{\epsilon-1}E_{i-1}K_{-\alpha_{i-1}})=\operatorname{j}(\cE_j)\operatorname{j}(\cE_i).
	\end{align*}
	Here we used that $b_{i+k,j+k}=b_{ij}$ and $b_{k-1,l}=b_{l,k}$, see \Cref{shiftinv}.
	
	For \cref{electricipone}, we first compute the result $R$ of $\operatorname{j}$ applied to the left-hand side of \cref{electricipone}.
	We break the task into pieces.
	\emph{Piece 1:}
	First, let us look at the sum of those summands in $R$ that contain three $F$'s.	 Together with \cref{uqgfserre} we get
	\begin{equation*}
		q^3 F_iF_iF_{i+1} - [2] F_iF_{i+1}F_i + q^{-3} F_{i+1}F_iF_i = 0.
	\end{equation*}
	\emph{Piece 2:} The sum of the summands in $R$ that contain three $E$'s is (up to the common $q^{3\epsilon+3}$ factor and calculated for $i+1$ instead)
	\begin{align*}
		&\begin{multlined}[t]q^3 E_{i}K_{-\alpha_{i}}E_{i}K_{-\alpha_{i}}E_{i+1}K_{-\alpha_{i+1}} - [2] E_{i}K_{-\alpha_{i}}E_{i+1}K_{-\alpha_{i+1}}E_{i}K_{-\alpha_{i}} \\+ q^{-3} E_{i+1}K_{-\alpha_{i+1}}E_{i}K_{-\alpha_{i}}E_{i}K_{-\alpha_{i}}\end{multlined}\\
		=&(q^{3} E_{i}E_{i}E_{i+1} - [2] E_{i}E_{i+1}E_{i} + q^{-3} E_{i+1}E_{i}E_{i})K_{-2\alpha_i-\alpha_{i+1}}\stackrel{\cref{uqgeserre}}=0.
	\end{align*}
	\emph{Piece 3:} Next we have those terms that contain two $F$'s.
	We split this case into three subcases, whether we have two $F_i$ or $F_{i+1}$ before $F_i$ or $F_{i+1}$ after $F_i$.
	
	In case of two $F_i$, we get (ignoring the common factor $q^{\epsilon-1}$)
	\begin{align*}
		&(q^{3} F_iF_iE_iK_{-\alpha_i} - [2] F_iE_iK_{-\alpha_i}F_i + q^{-3} E_iK_{-\alpha_i}F_iF_i)\\ 
		= \:&(q^{3} F_iF_iE_iK_{-\alpha_i} - (q^3+q)F_iE_iF_iK_{-\alpha_i} + q E_iF_iF_iK_{-\alpha_i})\\
		= \:&(q^{3} F_i[F_i,E_i]K_{-\alpha_i} + q [E_i, F_i]F_iK_{-\alpha_i}) +q[2]F_i\\
		=\:&(\frac{-q^3}{q-q^{-1}}F_i(K_{\alpha_i}-K_{-\alpha_i})K_{-\alpha_i} + \frac{q}{q-q^{-1}}(K_{\alpha_i}-K_{-\alpha_i})F_iK_{-\alpha_i})\\
		=\:&(\frac{-q^3}{q-q^{-1}}F_i(1-K_{-2\alpha_i}) + \frac{q^{-1}}{q-q^{-1}}F_i + \frac{q^{3}}{q-q^{-1}}F_iK_{-2\alpha_i})=-q[2]F_i\eqqcolon(*).
	\end{align*}
	%needed b_{i,i+1}=-2$ alternatively b_{i,i+1}=4
	Next assume $F_{i+1}$ appears before a unique $F_i$.
	Ignoring a factor $q^{\epsilon-1}$ we get
	\begin{align*}
		&-[2]E_{i-1}K_{-\alpha_{i-1}}F_{i+1}F_i + q^{-3} F_{i+1}E_{i-1}K_{-\alpha_{i-1}}F_i + q^{-3} F_{i+1}F_iE_{i-1}K_{-\alpha_{i-1}}\\
		=\:&\begin{multlined}[t]-(1+q^{-2})E_{i-1}F_{i+1}F_iK_{-\alpha_{i-1}} + q^{-4-\beta_{i-1,i+1}} E_{i-1}F_{i+1}F_iK_{-\alpha_{i-1}} \\+ q^{-3-\beta_{i-1,i+1}-\beta_{i-1,i}} E_{i-1}F_{i+1}F_iK_{-\alpha_{i-1}}=0.\end{multlined}
	\end{align*}
	%needed $b_{i+1,i}=4$ if $b_{ii} = 2$ or $b_{i+1,i} = 2$ if $b_{ii} = 2$
	The remaining case for two $F$'s is when $F_{i+1}$ appears after a unique $F_i$.
	Then, 
	\begin{align*}
		& q^{3} E_{i-1}K_{-\alpha_{i-1}}F_iF_{i+1} + q^{3} F_iE_{i-1}K_{-\alpha_{i-1}}F_{i+1}-[2]F_iF_{i+1}E_{i-1}K_{-\alpha_{i-1}}\\
		=\:&\begin{multlined}[t] q^{2} E_{i-1}F_iF_{i+1}K_{-\alpha_{i-1}} + q^{3-\beta_{i-1,i}} E_{i-1}F_iF_{i+1}K_{-\alpha_{i-1}}\\-[2]q^{-\beta_{i-1,i}-\beta_{i-1,i+1}}E_{i-1}F_iF_{i+1}K_{-\alpha_{i-1}} = 0.\end{multlined}
	\end{align*}
	\emph{Piece 5:} Now it remains to look at the case, where we have only one $F$.
	Similar to before, we split this case into three subcases.
	Namely, we look at the cases where two $E_{i-1}$ appear, $E_{i}$ appears before $E_{i-1}$ respectively $E_{i}$ appears after $E_{i-1}$.
	
	If we have two $E_{i-1}$, we get (we calculate again for $i+1$)
	\begin{align*}
		&q^3 E_{i}K_{-\alpha_{i}}E_{i}K_{-\alpha_{i}}F_{i+2} - [2] E_{i}K_{-\alpha_{i}}F_{i+2}E_{i}K_{-\alpha_{i}} + q^{-3} F_{i+2}E_{i}K_{-\alpha_{i}}E_{i}K_{-\alpha_{i}}\\
		=\;&(q^{3+2\beta_{i,i+2}} F_{i+2}E_{i}K_{-\alpha_{i}}E_{i} - [2]q^{\beta_{i,i+2}} F_{i+2}E_{i}K_{-\alpha_{i}}E_{i} + q^{-3} F_{i+2}E_{i}K_{-\alpha_{i}}E_{i})K_{-\alpha_{i}}\\
		\stackrel{\cref{uqgeserre}}=\;&0.
	\end{align*}
	
	The next case is when $E_{i}$ appears before $E_{i-1}$.
	We calculate in $\Uqg$ the following
	\begin{align*}
		&\:\:-[2]F_iE_iK_{-\alpha_i}E_{i-1}K_{-\alpha_{i-1}} + q^{-3} E_iK_{-\alpha_i}F_iE_{i-1}K_{-\alpha_{i-1}} + q^{-3}E_iK_{-\alpha_i}E_{i-1}K_{-\alpha_{i-1}}F_i\\
		=&\:\:-[2](F_iE_iK_{-\alpha_i}E_{i-1}+ q^{-1} E_iF_iK_{-\alpha_i}E_{i-1}+ q^{-2+\beta_{i-1,i}}E_iF_iK_{-\alpha_i}E_{i-1})K_{-\alpha_{i-1}}\\
		\stackrel{\cref{betaij}}{=}&\:\:[2]\left([E_{i},F_i]K_{-\alpha_i}E_{i-1}\right)K_{-\alpha_{i-1}} = [2]\left(\frac{1-K_{-2\alpha_i}}{q-q^{-1}}E_{i-1}\right)K_{-\alpha_{i-1}} \\
		{=}&\:\:\frac{q+q^{-1}}{q-q^{-1}}E_{i-1}K_{-\alpha_{i-1}} -\frac{q+q^3}{q-q^{-1}}E_{i-1}K_{-2\alpha_i-\alpha_{i-1}}=\colon (**).
	\end{align*}
	Finally, we look at the case where $E_i$ appears after $E_{i-1}$.
	Then,
	\begin{align*}
		&q^{3} F_{i}E_{i-1}K_{-\alpha_{i-1}}E_{i}K_{-\alpha_{i}} + q^{3} E_{i-1}K_{-\alpha_{i-1}}F_{i}E_{i}K_{-\alpha_{i}} - [2] E_{i-1}K_{-\alpha_{i-1}}E_{i}K_{-\alpha_{i}}F_{i}\\
		=&\left(q^{4-\beta_{i-1,i}} E_{i-1}K_{-\alpha_{i-1}}F_{i}E_{i}+ q^{3} E_{i-1}K_{-\alpha_{i-1}}F_{i}E_{i}- (q+q^3) E_{i-1}K_{-\alpha_{i-1}}E_{i}F_{i}\right)K_{-\alpha_{i}}\\
		=&-(q+q^3) E_{i-1}K_{-\alpha_{i-1}}[E_i,F_i]K_{-\alpha_{i}} = -\frac{q+q^3}{q-q^{-1}} E_{i-1}K_{-\alpha_{i-1}}(1-K_{-2\alpha_i})=\colon (***).
	\end{align*}
	Adding up the nonzero intermediate results (which only appear in Pieces 3-4) and recalling the $q^x$ scaling factors, we find that $\operatorname{j}$ maps $q^3\cE_i^2\cE_j-[2]\cE_i\cE_j\cE_i+q^{-3}\cE_j\cE_i^2$ to
	\begin{align*}
		&\;\quad (*)+(**)+(***) =\:-q^\epsilon[2]F_i + q^{2\epsilon-2}\frac{q^{-1}-q^{3}}{q-q^{-1}}E_{i-1}K_{-\alpha_{i-1}}\\
		=\;&-q^\epsilon[2]F_i - q^{2\epsilon-1}[2]E_{i-1}K_{-\alpha_{i-1}} = -q^{\epsilon}[2](F_i+q^{\epsilon-1}E_{i-1}K_{-\alpha_{i-1}})=\operatorname{j}( -q^{\epsilon}[2]\cE_i).
	\end{align*}
	Thus, \cref{electricipone} is satisfied.
	A similar calculation shows the compatibility with \cref{electricimone}.
	
	This proves that $\operatorname{j}\colon\el\to\Uqg$ is well-defined, It remains to show injectivity.
	By definition of the map, the image of a word in the generators $\cE_i$ of $\el$ has a unique summand that contains only $F_i$'s (and it is moreover the same word in these $F_i$'s).
	Now the statement follows from \cref{grel}, since the algebra $\el$ is filtered with associated graded isomorphic to the subalgebra of $\Uqg$ generated by the $F_i$s.
\end{proof}
\subsection{Proof of \cref{heckeaction}}\label{proofHecke}
\begin{proof}
	We show the first statement and the most complicated braid relations in the general mixed cases as claimed in \Cref{heckeaction3}.
	The remaining cases are then straight-forward adaptions of the easy checks.
	We start by showing that $H$ is $\Uqg$-linear.
	
	It clearly commutes with $K_\la$, so we consider $E_i$ and $F_i$. The computation for $E_i$ is very  similar to the one for $F_i$, so we only treat the $F_i$.
	Recall that $v_jF_i=\delta_{ij}v_{i+1}$.
	Let $\eta_i\coloneqq\beta_i$ if $\odot=\odot_1$ and $\eta_i\coloneqq\beta_{i-1}'$ if $\odot=\odot_2$.
	Then 
	\begin{equation*}
		\begin{tikzcd}
			v_i\odot v_j\arrow[r, "F_a"]\arrow[dd, "H"]&\delta_{ia}v_{i+1}\odot v_jK_{\eta_i} + \delta_{ja}v_i\odot v_{j+1}\arrow[d, "H"]\\
			&\begin{array}[t]{c}\delta_{ia}a_{i+1,j}v_jK_{\eta_i}\odot v_{i+1} \\
				+ \delta_{ia}\delta_{i+1<j}(q^{-1}-q)v_{i+1}\odot v_jK_{\eta_i}\\
				+ \delta_{ja}a_{i,j+1}v_{j+1}\odot v_i + \delta_{ja}\delta_{i<j+1}(q^{-1}-q)v_i\odot v_{j+1}
			\end{array}\\
			\begin{array}[t]{c}
				a_{ij}v_j\odot v_i +\\\delta_{i<j}(q^{-1}-q) v_i\odot v_j
			\end{array}
			\arrow[r, "F_a"]&\begin{multlined}[c]
				\delta_{ia}a_{ij}v_j\odot v_{i+1} +\delta_{ia}\delta_{i<j}(q^{-1}-q) v_{i+1}\odot v_jK_{\eta_i}\\
				+\delta_{ja}a_{ij}v_{j+1}\odot v_iK_{\eta_j} + \delta_{ja}\delta_{i<j}(q^{-1}-q) v_i\odot v_{j+1}
			\end{multlined}
		\end{tikzcd}
	\end{equation*}
	With the abbreviation $c_{ji}\coloneqq\langle\eta_i, e_j\rangle$, we therefore have to show the following
	\begin{equation}\label{eq:phif}
		\begin{aligned}
			a_{i+1,j}c_{ji}&=a_{ij}&&\text{for }j\notin\{i,i+1\},\\
			a_{i+1,j}c_{ji}&=a_{ij} + (q^{-1}-q)c_{ji}&&\text{for }j= i+1,\\
			a_{i+1,j}c_{ji}+q^{-1}-q&=a_{ij}&&\text{for }j= i,\\
			a_{i,j+1}&=a_{ij}c_{ij}&&\text{for any }i,j.
		\end{aligned}
	\end{equation}
	The following table shows all the different possibilities for $i$ and $j$.
	\begin{center}
		\begin{tblr}{width=\textwidth, colspec={cccccccc}}
			\toprule
			&parity of $i-l$&parity of $j-l$&$a_{i+1,j}$&$c_{ji}$&$a_{ij}$&$c_{ij}$&$a_{i,j+1}$\\
			\midrule
			\SetCell[r=4]{m}$j>i+1$	&even&even&	$q^{-3}$&$q^4$		&$q$		&$1$		&$q$\\
			&even&odd&	$q$		&$1$		&$q$		&$1$		&$q$\\
			&odd&even&	$q$		&$q^{-4}$	&$q^{-3}$	&$q^4$		&$q$\\
			&odd&odd&	$q$		&$1$		&$q$		&$q^{-4}$	&$q^{-3}$\\
			\midrule
			\SetCell[r=4]{m}$j<i$	&even&even&	$q^{-1}$&$1$		&$q^{-1}$	&$q^4$		&$q^3$\\
			&even&odd&	$q^{-1}$&$q^4$		&$q^3$		&$q^{-4}$	&$q^{-1}$\\
			&odd&even&	$q^{-1}$&$1$		&$q^{-1}$	&$1$		&$q^{-1}$\\
			&odd&odd&	$q^3$	&$q^{-4}$	&$q^{-1}$	&$1$		&$q^{-1}$\\
			\midrule
			\SetCell[r=2]{m}$j=i$	&even&even&	$q^{-1}$&$q^2$		&$q^{-1}$	&$q^2$		&$q$\\
			&odd&odd&	$q^3$	&$q^{-2}$	&$q^{-1}$	&$q^{-2}$	&$q^{-3}$\\
			\midrule
			\SetCell[r=2]{m}$j=i+1$	&even&odd&	$q^{-1}$&$1$		&$q$		&$1$		&$q$\\
			&odd&even&	$q^{-1}$&$q^{-4}$	&$q^{-3}$	&$q^4$		&$q$\\
			\bottomrule
		\end{tblr}
	\end{center}
	Now it is easy to see that \cref{eq:phif} is satisfied.
	This shows that $H$ is $\Uqg$-linear.
	
	Next we show that $H$ satisfies the Hecke relations. 
	
	We begin by computing $H^2(v_i\odot v_j)$.
	In case $i\neq j$ we have 
	\begin{align*}
		H^2(v_i\odot v_j) &= H(a_{ij}v_j\odot v_i +\delta_{i<j}(q^{-1}-q) v_i\odot v_j)\\
		&= \begin{multlined}[t]a_{ij}a_{ji}v_i\odot v_j + a_{ij}\delta_{j<i}(q^{-1}-q)v_j\odot v_i \\+ \delta_{i<j}(q^{-1}-q)a_{ij}v_j\odot v_i + \delta_{i<j}(q^{-1}-q)^2v_i\odot v_j\end{multlined}\\
		&= a_{ij}a_{ji}v_i\odot v_j + (q^{-1}-q)(a_{ij}v_j\odot v_i + \delta_{i<j}(q^{-1}-q)v_i\odot v_j)\\
		&=v_i\odot v_j + (q^{-1}-q)H(v_i\odot v_j).
	\end{align*}
	In case $i=j$, we have 
	\begin{equation*}
		H^2(v_i\odot v_i) = H(a_{ii}v_i\odot v_i)
		=q^{-2}v_i\odot v_i
		= v_i\odot v_i + (q^{-1}-q)H(v_i\odot v_i),
	\end{equation*}
	and thus the first Hecke relation is satisfied.
	The second Hecke relation is obvious. 
	
	It thus remains to show the braid relations for $V\odot V\odot V$, where each of the $\odot$ can be chosen from $\{\odot_1,\odot_2\}$.
	So let $V\odot_l V\odot_r V$ with $l,r\in\{1,2\}$.
	To simplify notation abbreviate $v_{ijk}:=v_i\odot v_j\odot v_k$.
	We will also write $a_{ij}^l$ and $a_{ij}^r$ to emphasize the dependence on $l$ and $r$ in the definition.
	We compute
	\begin{align*}
		&H_1H_2H_1(v_{ijk}) = H_1H_2(a_{ij}^lv_{jik} + \delta_{i<j}(q^{-1}-q)v_{ijk})\\
		={}& H_1(a_{ij}^la_{ik}^r v_{jki} + \delta_{i<k}a_{ij}^l(q^{-1}-q)v_{jik} + \delta_{i<j}(q^{-1}-q)a_{jk}^r v_{ikj} + \delta_{i<j<k}(q^{-1}-q)^2v_{ijk})\\
		={}& \begin{multlined}[t]{ \underbrace{a_{ij}^la_{ik}^r a_{jk}^lv_{kji}}_{\scriptstyle\circled{1}}}
			+ {\underbrace{\delta_{j<k}(q^{-1}-q)a_{ij}^l a_{ik}^rv_{jki}} _{\scriptstyle\circled{2}}}
			+{\underbrace {\delta_{i<k}a_{ji}^la_{ij}^l(q^{-1}-q)v_{ijk}}_{\scriptstyle\circled{3}}}\\
			+{\underbrace {\delta_{j<i<k}a_{ij}^l(q^{-1}-q)^2v_{jik}} _{\scriptstyle\circled{4}}}
			+{\underbrace {\delta_{i<j}(q^{-1}-q)a_{jk}^r a_{ik}^lv_{kij}} _{\scriptstyle\circled{5}}}
			+{\underbrace {\delta_{k>i<j}(q^{-1}-q)^2a_{jk}^r v_{ikj}} _{\scriptstyle\circled{6}}}\\
			+ {\underbrace{\delta_{i<j<k}a_{ij}^l(q^{-1}-q)^2v_{jik}} _{\scriptstyle\circled{4}}}
			+{\underbrace {\delta_{i<j<k}(q^{-1}-q)^3v_{ijk}} _{\scriptstyle\circled{7}},}
			\vspace{-5mm}
		\end{multlined}	
	\end{align*}
	which we need to compare with \vspace{-5mm}
	\begin{align*}
		&H_2H_1H_2(v_{ijk}) = H_2H_1(a_{jk}^r v_{ikj} + \delta_{j<k}(q^{-1}-q)v_{ijk})\\
		=\;& H_2(a_{ik}^ra_{jk}^l v_{kij} + \delta_{i<k}a_{jk}^r (q^{-1}-q)v_{ikj} + \delta_{j<k}(q^{-1}-q)a_{ij}^lv_{jik} + \delta_{i<j<k}(q^{-1}-q)^2v_{ijk})\\
		=\;& \begin{multlined}[t]{\underbrace{a_{ij}^r a_{ik}^la_{jk}^r v_{kji}} _{\scriptstyle\circled{1}}}
			+{\underbrace {\delta_{i<j}(q^{-1}-q)a_{ik}^la_{jk}^r v_{kij}}_{\scriptstyle\circled{5}}}
			+ {\underbrace {\delta_{i<k}a_{jk}^r a_{kj}^r (q^{-1}-q)v_{ijk}}_{\scriptstyle\circled{3}}}\\
			+{\underbrace {\delta_{i<k<j}a_{jk}^r (q^{-1}-q)^2v_{ikj}}_{\scriptstyle\circled{6}}}
			+ {\underbrace{\delta_{j<k}(q^{-1}-q)a_{ij}^ra_{ik}^r v_{jki}} _{\scriptstyle\circled{2}}}
			+ {\underbrace{\delta_{j<k>i}(q^{-1}-q)^2a_{ij}^lv_{jik}} _{\scriptstyle\circled{4}}}\\
			+ {\underbrace {\delta_{i<j<k}a_{jk}^r (q^{-1}-q)^2v_{ikj}}_{\scriptstyle\circled{6}}}
			+{\underbrace {\delta_{i<j<k}(q^{-1}-q)^3v_{ijk}}_{\scriptstyle\circled{7}}}
			\vspace{-10mm}
		\end{multlined}.
	\end{align*}
	The parts $\scriptstyle\circled{2}$, $\scriptstyle\circled{5}$, $\scriptstyle\circled{7}$ agree in the two expressions.
	Let us consider now $\scriptstyle\circled{3}$, $\scriptstyle\circled{4}$, $\scriptstyle\circled{6}$.
	
	If $i=j=k$, then the terms for $\scriptstyle\circled{3}$, $\scriptstyle\circled{4}$, $\scriptstyle\circled{6}$ match since $a_{tt}^l=q^{-1}=a_{tt}^r$ for any $t$. 
	
	Next assume $i,j,k$ are pairwise distinct.
	Then the parts $\scriptstyle\circled{3}$ agree if $a_{ji}^la_{ij}^l=a_{jk}^r a_{kj}^r$ which holds by \cref{rkas}.
	The parts $\scriptstyle\circled{4}$ agree if $\delta_{j<i<k}a_{ij}^l+\delta_{i<j<k}a_{ij}^l=\delta_{j<k>i}a_{ij}^l$ which obviously holds.
	Similarly, for $\scriptstyle\circled{6}$.		
	Assume now $i=j\not=k$.
	Then the respective sums $\scriptstyle\circled{3}+\scriptstyle\circled{4}+\scriptstyle\circled{6}$ are $\delta_{i<k}a_{ii}^la_{ii}^l(q^{-1}-q)v_{iik}$ and $(\delta_{i<k}a_{ik}^r a_{ki}^r+\delta_{j<k>i}a_{ii}^l)(q^{-1}-q)^2v_{iik}$.
	They agree, since $a_{ii}^la_{ii}^l=q^{-2}=1+q^{-1}(q^{-1}-q)= a_{ik}^r a_{ki}^r+a_{ii}^l(q^{-1}-q)$ by \cref{rkas}. 
	
	Assume next $i\not=j=k$.
	Then the sums are $\delta_{i<k}a_{ji}^la_{ij}^l(q^{-1}-q)+\delta_{k>i<j}(q^{-1}-q)^2a_{jj}^r$ 
	and $\delta_{k>i<j}(q^{-1}-q)^2a_{jj}^r$.
	They agree, since $a_{ji}^la_{ij}^l=1+q^{-2}=q^{-1}(q^{-1}-q)$ by \cref{rkas}.
	
	Assume finally $i=k\not=j$.
	Then both sums $\scriptstyle\circled{3}+\scriptstyle\circled{4}+\scriptstyle\circled{6}$ vanish. \\
		\begin{figure}[!htbp]\label{annoyingtable}
		\begin{center}
		\begin{tblr}{ccccc}
			\toprule
			parity of $i-l$&parity of $j-l$&parity of $k-l$&$a_{ij}^la_{ik}^r a_{jk}^l$&$a_{ij}^r a_{ik}^la_{jk}^r$\\
			\midrule
			\SetCell[r=4]{m}even&\SetCell[r=2]{m} even&even&	$q^{3}$		&$q^{3}$\\
			%\cmidrule{3-5}
			&&odd&		$q^{-1}$	&$q^{-1}$\\
			\cmidrule{2-5}
			&\SetCell[r=2]{m}odd&even&		$q^{-1}$	&$q^{-1}$\\
			%\cmidrule{3-5}
			&&odd&		$q^{-1}$	&$q^{-1}$\\
			\midrule
			\SetCell[r=4]{m}odd&\SetCell[r=2]{m}even&even&		$q^{-1}$	&$q^{-1}$\\
			%\cmidrule{3-5}
			&&odd&		$q^{-1}$	&$q^{-1}$\\
			\cmidrule{2-5}
			&\SetCell[r=2]{m}odd&even&		$q^{-1}$	&$q^{-1}$\\
			%\cmidrule{3-5}
			&&odd&		$q^{3}$		&$q^{3}$\\
			\bottomrule
		\end{tblr}
	\end{center}
	\caption{Comparision of values}
	\end{figure}
	It remains to compare the two parts labelled $\scriptstyle\circled{1}$. 
	
	If $i=j=k$ then they agree since $a_{tt}^l=q^{-1}=a_{tt}^r$ for any $t$. 
	If $i=j\not=k$ then we ask if $a_{ik}^ra_{ik}^l=a_{ik}^la_{ik}^r$ which is obviously true. 
	If $i=k\not=j$ then we ask if $a_{ij}^la_{ji}^l=a_{ij}^ra_{ji}^r$ which holds by \cref{rkas}. 
	If $j=k\not=i$ then we ask if $a_{ij}^la_{ji}^r=a_{ij}^ra_{il}^j$ which is clearly true. 
	
	Therefore, the $\scriptstyle\circled{1}$-parts agree if at least two of $i$, $j$, and $k$ are equal, and it remains to consider $\scriptstyle\circled{1}$ in the case where $i$, $j$, $k$ are distinct and $r\not=l$. 
	Using \cref{rkas} we can reduce to the case $i<j<k$.
	We compute the values depending on whether $i-l$, $j-l$ and $k-l$ are even or odd in \cref{annoyingtable}.
	Note that $r$ has the opposite parity of $l$ since $r\not=l$. We see that the $\scriptstyle\circled{1}$-parts agree as well.
	This finishes the proof. 
\end{proof}
\section{Proof of \cref{relationsinH}}\label{BoBonesec}
\begin{lem} \label{BoBone}$B_0$ and $B_1$ are mutually inverse. 
\end{lem}

\begin{proof} We compute
	\begin{align*}
		&B_1B_0= \begin{pmatrix}
			-\begin{tikzpicture}[line width=\lw, myscale=0.8, scale=0.35,yscale=0.7]
				\node[anchor=north] at (0,0) (A) {$\scriptstyle{i}$};
				\node[anchor=north] at (1,0) (B) {$\scriptstyle{j}$};
				\node[anchor=north] at (2,0) (C) {$\scriptstyle{i}$};
				\node[anchor=south] at (0,6) (D) {$\scriptstyle{i}$};
				\node[anchor=south] at (1,6) (E) {$\scriptstyle{j}$};
				\node[anchor=south] at (2,6) (F) {$\scriptstyle{i}$};
				\draw (C.north)--(1,1.5)--(0,3)--(1,4.5)--(F.south);
				\draw (A.north)--(0,1.5)--(1,3) node[pos=0.8](a) {} --(0,4.5)--(D.south);
				\fill (a) circle(\dw);
				\draw (B.north)--(2,1.5)--(2,3)--(2,4.5)--(E.south);
			\end{tikzpicture}+
			\begin{tikzpicture}[line width=\lw, myscale=0.8, scale=0.35,yscale=0.7]
				\node[anchor=north] at (0,0) (A) {$\scriptstyle{i}$};
				\node[anchor=north] at (1,0) (B) {$\scriptstyle{j}$};
				\node[anchor=north] at (2,0) (C) {$\scriptstyle{i}$};
				\node[anchor=south] at (0,6) (D) {$\scriptstyle{i}$};
				\node[anchor=south] at (1,6) (E) {$\scriptstyle{j}$};
				\node[anchor=south] at (2,6) (F) {$\scriptstyle{i}$};
				\draw (B.north)--(0,1.5)--(0,3)--(0,4.5)--(E.south);
				\draw (C.north)--(2,1.5)--(1,3)--(2,4.5)--(F.south);
				\draw (A.north)--(1,1.5)--(2,3) node[pos=0.8](a) {} --(1,4.5)--(D.south);
				\fill (a) circle(\dw);
			\end{tikzpicture}-
			\begin{tikzpicture}[line width=\lw, myscale=0.8, scale=0.35,yscale=0.7]
				\node[anchor=north] at (0,0) (A) {$\scriptstyle{i}$};
				\node[anchor=north] at (1,0) (B) {$\scriptstyle{j}$};
				\node[anchor=north] at (2,0) (C) {$\scriptstyle{i}$};
				\node[anchor=south] at (0,6) (D) {$\scriptstyle{i}$};
				\node[anchor=south] at (1,6) (E) {$\scriptstyle{j}$};
				\node[anchor=south] at (2,6) (F) {$\scriptstyle{i}$};
				\draw (A.north)..controls +(0.2,0.8) and +(-0.2,0.8)..(B.north);
				\draw (E.south)..controls +(0.2,-0.8) and +(-0.2,-0.8)..(F.south);
				\draw (C.north)--(D.south);
			\end{tikzpicture}&
			-\begin{tikzpicture}[line width=\lw, myscale=0.8, scale=0.35,yscale=0.7]
				\node[anchor=north] at (0,0) (A) {$\scriptstyle{i}$};
				\node[anchor=north] at (1,0) (B) {$\scriptstyle{j}$};
				\node[anchor=north] at (2,0) (C) {$\scriptstyle{i}$};
				\node[anchor=south] at (0,6) (D) {$\scriptstyle{i}$};
				\node[anchor=south] at (1,6) (E) {$\scriptstyle{j}$};
				\node[anchor=south] at (2,6) (F) {$\scriptstyle{i}$};
				\draw (C.north)--(1,1.5)--(0,3)--(1,4.5)--(F.south);
				\draw (A.north)--(0,1.5)--(1,3)--(0,4.5)--(D.south);
				\draw (B.north)--(2,1.5)--(2,3)--(2,4.5)--(E.south);
			\end{tikzpicture}+
			\begin{tikzpicture}[line width=\lw, myscale=0.8, scale=0.35,yscale=0.7]
				\node[anchor=north] at (0,0) (A) {$\scriptstyle{i}$};
				\node[anchor=north] at (1,0) (B) {$\scriptstyle{j}$};
				\node[anchor=north] at (2,0) (C) {$\scriptstyle{i}$};
				\node[anchor=south] at (0,6) (D) {$\scriptstyle{i}$};
				\node[anchor=south] at (1,6) (E) {$\scriptstyle{j}$};
				\node[anchor=south] at (2,6) (F) {$\scriptstyle{i}$};
				\draw (B.north)--(0,1.5)--(0,3)--(0,4.5)--(E.south);
				\draw (C.north)--(2,1.5)--(1,3)--(2,4.5)--(F.south);
				\draw (A.north)--(1,1.5)--(2,3)--(1,4.5)--(D.south);
			\end{tikzpicture}\\
			\begin{tikzpicture}[line width=\lw, myscale=0.8, scale=0.35,yscale=0.7]
				\node[anchor=north] at (0,0) (A) {$\scriptstyle{i}$};
				\node[anchor=north] at (1,0) (B) {$\scriptstyle{j}$};
				\node[anchor=north] at (2,0) (C) {$\scriptstyle{i}$};
				\node[anchor=south] at (0,6) (D) {$\scriptstyle{i}$};
				\node[anchor=south] at (1,6) (E) {$\scriptstyle{j}$};
				\node[anchor=south] at (2,6) (F) {$\scriptstyle{i}$};
				\draw (C.north)--(1,1.5)--(0,3) node[pos=0.8](a) {} --(1,4.5)--(F.south);
				\fill (a) circle(\dw);
				\draw (A.north)--(0,1.5)--(1,3) node[pos=0.8](a) {} --(0,4.5)--(D.south);
				\fill (a) circle(\dw);
				\draw (B.north)--(2,1.5)--(2,3)--(2,4.5)--(E.south);
			\end{tikzpicture}-
			\begin{tikzpicture}[line width=\lw, myscale=0.8, scale=0.35,yscale=0.7]
				\node[anchor=north] at (0,0) (A) {$\scriptstyle{i}$};
				\node[anchor=north] at (1,0) (B) {$\scriptstyle{j}$};
				\node[anchor=north] at (2,0) (C) {$\scriptstyle{i}$};
				\node[anchor=south] at (0,6) (D) {$\scriptstyle{i}$};
				\node[anchor=south] at (1,6) (E) {$\scriptstyle{j}$};
				\node[anchor=south] at (2,6) (F) {$\scriptstyle{i}$};
				\draw (B.north)--(0,1.5)--(0,3)--(0,4.5)--(E.south);
				\draw (C.north)--(2,1.5)--(1,3) node[pos=0.8](a) {} --(2,4.5)--(F.south);
				\fill (a) circle(\dw);
				\draw (A.north)--(1,1.5)--(2,3) node[pos=0.8](a) {} --(1,4.5)--(D.south);
				\fill (a) circle(\dw);
			\end{tikzpicture}&
			\begin{tikzpicture}[line width=\lw, myscale=0.8, scale=0.35,yscale=0.7]
				\node[anchor=north] at (0,0) (A) {$\scriptstyle{i}$};
				\node[anchor=north] at (1,0) (B) {$\scriptstyle{j}$};
				\node[anchor=north] at (2,0) (C) {$\scriptstyle{i}$};
				\node[anchor=south] at (0,6) (D) {$\scriptstyle{i}$};
				\node[anchor=south] at (1,6) (E) {$\scriptstyle{j}$};
				\node[anchor=south] at (2,6) (F) {$\scriptstyle{i}$};
				\draw (A.north)--(0,1.5)--(1,3)--(0,4.5)--(D.south);
				\draw (B.north)--(2,1.5)--(2,3)--(2,4.5)--(E.south);
				\draw (C.north)--(1,1.5)--(0,3) node[pos=0.8](a) {} --(1,4.5)--(F.south);
				\fill (a) circle(\dw);
			\end{tikzpicture}-
			\begin{tikzpicture}[line width=\lw, myscale=0.8, scale=0.35,yscale=0.7]
				\node[anchor=north] at (0,0) (A) {$\scriptstyle{i}$};
				\node[anchor=north] at (1,0) (B) {$\scriptstyle{j}$};
				\node[anchor=north] at (2,0) (C) {$\scriptstyle{i}$};
				\node[anchor=south] at (0,6) (D) {$\scriptstyle{i}$};
				\node[anchor=south] at (1,6) (E) {$\scriptstyle{j}$};
				\node[anchor=south] at (2,6) (F) {$\scriptstyle{i}$};
				\draw (A.north)--(1,1.5)--(2,3)--(1,4.5)--(D.south);;
				\draw (B.north)--(0,1.5)--(0,3)--(0,4.5)--(E.south);
				\draw (C.north)--(2,1.5)--(1,3) node[pos=0.8](a) {} --(2,4.5)--(F.south);
				\fill (a) circle(\dw);
			\end{tikzpicture}-
			\begin{tikzpicture}[line width=\lw, myscale=0.8, scale=0.35,yscale=0.7]
				\node[anchor=north] at (0,0) (A) {$\scriptstyle{i}$};
				\node[anchor=north] at (1,0) (B) {$\scriptstyle{j}$};
				\node[anchor=north] at (2,0) (C) {$\scriptstyle{i}$};
				\node[anchor=south] at (0,6) (D) {$\scriptstyle{i}$};
				\node[anchor=south] at (1,6) (E) {$\scriptstyle{j}$};
				\node[anchor=south] at (2,6) (F) {$\scriptstyle{i}$};
				\draw (A.north)..controls +(0.2,0.8) and +(-0.2,0.8)..(B.north);
				\draw (E.south)..controls +(0.2,-0.8) and +(-0.2,-0.8)..(F.south);
				\draw (C.north)--(D.south);
			\end{tikzpicture}\\
		\end{pmatrix}\\
		\overset{\eqref{ncdotcrossingii}}{=}&\begin{pmatrix}
			-\begin{tikzpicture}[line width=\lw, myscale=0.8, scale=0.35,yscale=0.7]
				\node[anchor=north] at (0,0) (A) {$\scriptstyle{i}$};
				\node[anchor=north] at (1,0) (B) {$\scriptstyle{j}$};
				\node[anchor=north] at (2,0) (C) {$\scriptstyle{i}$};
				\node[anchor=south] at (0,4.5) (D) {$\scriptstyle{i}$};
				\node[anchor=south] at (1,4.5) (E) {$\scriptstyle{j}$};
				\node[anchor=south] at (2,4.5) (F) {$\scriptstyle{i}$};
				\draw (A.north)--(0,1.5)--(1,3)--(F.south);
				\draw (B.north)--(2,1.5)--(2,3)--(E.south);
				\draw (C.north)--(1,1.5)--(0,3)--(D.south);
			\end{tikzpicture}+
			\begin{tikzpicture}[line width=\lw, myscale=0.8, scale=0.35,yscale=0.7]
				\node[anchor=north] at (0,0) (A) {$\scriptstyle{i}$};
				\node[anchor=north] at (1,0) (B) {$\scriptstyle{j}$};
				\node[anchor=north] at (2,0) (C) {$\scriptstyle{i}$};
				\node[anchor=south] at (0,4.5) (D) {$\scriptstyle{i}$};
				\node[anchor=south] at (1,4.5) (E) {$\scriptstyle{j}$};
				\node[anchor=south] at (2,4.5) (F) {$\scriptstyle{i}$};
				\draw (A.north)--(1,1.5)--(2,3)--(F.south);
				\draw (B.north)--(0,1.5)--(0,3)--(E.south);
				\draw (C.north)--(2,1.5)--(1,3)--(D.south);
			\end{tikzpicture}-
			\begin{tikzpicture}[line width=\lw, myscale=0.8, scale=0.35,yscale=0.7]
				\node[anchor=north] at (0,0) (A) {$\scriptstyle{i}$};
				\node[anchor=north] at (1,0) (B) {$\scriptstyle{j}$};
				\node[anchor=north] at (2,0) (C) {$\scriptstyle{i}$};
				\node[anchor=south] at (0,4.5) (D) {$\scriptstyle{i}$};
				\node[anchor=south] at (1,4.5) (E) {$\scriptstyle{j}$};
				\node[anchor=south] at (2,4.5) (F) {$\scriptstyle{i}$};
				\draw (A.north)..controls +(0.2,0.8) and +(-0.2,0.8)..(B.north);
				\draw (E.south)..controls +(0.2,-0.8) and +(-0.2,-0.8)..(F.south);
				\draw (C.north)--(D.south);
			\end{tikzpicture}&0\\0&
			-\begin{tikzpicture}[line width=\lw, myscale=0.8, scale=0.35,yscale=0.7]
				\node[anchor=north] at (0,0) (A) {$\scriptstyle{i}$};
				\node[anchor=north] at (1,0) (B) {$\scriptstyle{j}$};
				\node[anchor=north] at (2,0) (C) {$\scriptstyle{i}$};
				\node[anchor=south] at (0,4.5) (D) {$\scriptstyle{i}$};
				\node[anchor=south] at (1,4.5) (E) {$\scriptstyle{j}$};
				\node[anchor=south] at (2,4.5) (F) {$\scriptstyle{i}$};
				\draw (A.north)--(0,1.5)--(1,3)--(F.south);
				\draw (B.north)--(2,1.5)--(2,3)--(E.south);
				\draw (C.north)--(1,1.5)--(0,3)--(D.south);
			\end{tikzpicture}+
			\begin{tikzpicture}[line width=\lw, myscale=0.8, scale=0.35,yscale=0.7]
				\node[anchor=north] at (0,0) (A) {$\scriptstyle{i}$};
				\node[anchor=north] at (1,0) (B) {$\scriptstyle{j}$};
				\node[anchor=north] at (2,0) (C) {$\scriptstyle{i}$};
				\node[anchor=south] at (0,4.5) (D) {$\scriptstyle{i}$};
				\node[anchor=south] at (1,4.5) (E) {$\scriptstyle{j}$};
				\node[anchor=south] at (2,4.5) (F) {$\scriptstyle{i}$};
				\draw (B.north)--(0,1.5)--(0,3)--(E.south);
				\draw (A.north)--(1,1.5)--(2,3)--(F.south);
				\draw (C.north)--(2,1.5)--(1,3)--(D.south);
			\end{tikzpicture}-
			\begin{tikzpicture}[line width=\lw, myscale=0.8, scale=0.35,yscale=0.7]
				\node[anchor=north] at (0,0) (A) {$\scriptstyle{i}$};
				\node[anchor=north] at (1,0) (B) {$\scriptstyle{j}$};
				\node[anchor=north] at (2,0) (C) {$\scriptstyle{i}$};
				\node[anchor=south] at (0,4.5) (D) {$\scriptstyle{i}$};
				\node[anchor=south] at (1,4.5) (E) {$\scriptstyle{j}$};
				\node[anchor=south] at (2,4.5) (F) {$\scriptstyle{i}$};
				\draw (A.north)..controls +(0.2,0.8) and +(-0.2,0.8)..(B.north);
				\draw (E.south)..controls +(0.2,-0.8) and +(-0.2,-0.8)..(F.south);
				\draw (C.north)--(D.south);
			\end{tikzpicture}
		\end{pmatrix}=\begin{pmatrix}
			\begin{tikzpicture}[line width=\lw, myscale=0.8, scale=0.35,yscale=0.7]
				\node[anchor=north] at (0,0) (A) {$\scriptstyle{i}$};
				\node[anchor=north] at (1,0) (B) {$\scriptstyle{j}$};
				\node[anchor=north] at (2,0) (C) {$\scriptstyle{i}$};
				\node[anchor=south] at (0,4.5) (D) {$\scriptstyle{i}$};
				\node[anchor=south] at (1,4.5) (E) {$\scriptstyle{j}$};
				\node[anchor=south] at (2,4.5) (F) {$\scriptstyle{i}$};
				\draw (A.north)--(D.south);
				\draw (B.north)--(E.south);
				\draw (C.north)--(F.south);
			\end{tikzpicture}&0\\0&
			\begin{tikzpicture}[line width=\lw, myscale=0.8, scale=0.35,yscale=0.7]
				\node[anchor=north] at (0,0) (A) {$\scriptstyle{i}$};
				\node[anchor=north] at (1,0) (B) {$\scriptstyle{j}$};
				\node[anchor=north] at (2,0) (C) {$\scriptstyle{i}$};
				\node[anchor=south] at (0,4.5) (D) {$\scriptstyle{i}$};
				\node[anchor=south] at (1,4.5) (E) {$\scriptstyle{j}$};
				\node[anchor=south] at (2,4.5) (F) {$\scriptstyle{i}$};
				\draw (A.north)--(D.south);
				\draw (B.north)--(E.south);
				\draw (C.north)--(F.south);
			\end{tikzpicture}\\
		\end{pmatrix}
	\end{align*}
	where we used \eqref{ncbraid} for the last equality. On the other hand we compute
	\begin{align*}
		B_0B_1=&\begin{pmatrix}
			-\begin{tikzpicture}[myscale=0.8, scale=0.35,yscale=0.7,line width=\lw]
				\node[anchor=north] at (0,0) (A) {$\scriptstyle{i}$};
				\node[anchor=north] at (1,0) (B) {$\scriptstyle{i}$};
				\node[anchor=north] at (2,0) (C) {$\scriptstyle{j}$};
				\node[anchor=south] at (0,6) (D) {$\scriptstyle{i}$};
				\node[anchor=south] at (1,6) (E) {$\scriptstyle{i}$};
				\node[anchor=south] at (2,6) (F) {$\scriptstyle{j}$};
				\draw (A.north)--(1,1.5)--(2,3)--(1,4.5)--(D.south);
				\draw (B.north)-- node[pos=0.2] (a) {}(0,1.5)--(0,4.5)--(E.south);
				\fill (a) circle (\dw);
				\draw (C.north)--(2,1.5)--(1,3)--(2,4.5)--(F.south);
			\end{tikzpicture}+\begin{tikzpicture}[myscale=0.8, scale=0.35,yscale=0.7,line width=\lw]
				\node[anchor=north] at (0,0) (A) {$\scriptstyle{i}$};
				\node[anchor=north] at (1,0) (B) {$\scriptstyle{i}$};
				\node[anchor=north] at (2,0) (C) {$\scriptstyle{j}$};
				\node[anchor=south] at (0,6) (D) {$\scriptstyle{i}$};
				\node[anchor=south] at (1,6) (E) {$\scriptstyle{i}$};
				\node[anchor=south] at (2,6) (F) {$\scriptstyle{j}$};
				\draw (A.north)--(1,1.5)--(2,3)--(1,4.5)--(D.south) node[pos=0.8] (a) {} ;
				\fill (a) circle(\dw);
				\draw (B.north)--(0,1.5)--(0,4.5)--(E.south);
				\draw (C.north)--(2,1.5)--(1,3)--(2,4.5)--(F.south);
			\end{tikzpicture}&
			\begin{tikzpicture}[myscale=0.8, scale=0.35,yscale=0.7,line width=\lw]
				\node[anchor=north] at (0,0) (A) {$\scriptstyle{i}$};
				\node[anchor=north] at (1,0) (B) {$\scriptstyle{i}$};
				\node[anchor=north] at (2,0) (C) {$\scriptstyle{j}$};
				\node[anchor=south] at (0,6) (D) {$\scriptstyle{j}$};
				\node[anchor=south] at (1,6) (E) {$\scriptstyle{i}$};
				\node[anchor=south] at (2,6) (F) {$\scriptstyle{i}$};
				\draw (A.north)--(1,1.5)--(2,3)--(2,4.5)--(E.south);
				\draw (B.north)-- node[pos=0.2] (a) {}(0,1.5)--(0,3)--(1,4.5)--(F.south);
				\fill (a) circle (\dw);
				\draw (C.north)--(2,1.5)--(1,3)--(0,4.5)--(D.south);
			\end{tikzpicture}-\begin{tikzpicture}[myscale=0.8, scale=0.35,yscale=0.7,line width=\lw]
				\node[anchor=north] at (0,0) (A) {$\scriptstyle{i}$};
				\node[anchor=north] at (1,0) (B) {$\scriptstyle{i}$};
				\node[anchor=north] at (2,0) (C) {$\scriptstyle{j}$};
				\node[anchor=south] at (0,6) (D) {$\scriptstyle{j}$};
				\node[anchor=south] at (1,6) (E) {$\scriptstyle{i}$};
				\node[anchor=south] at (2,6) (F) {$\scriptstyle{i}$};
				\draw (A.north)--(1,1.5)--(2,3)--(2,4.5)--(E.south) node[pos=0.8] (a) {} ;
				\fill (a) circle(\dw);
				\draw (B.north)--(0,1.5)--(0,3)--(1,4.5)--(F.south);
				\draw (C.north)--(2,1.5)--(1,3)--(0,4.5)--(D.south);
			\end{tikzpicture}&
			-\begin{tikzpicture}[myscale=0.8, scale=0.35,yscale=0.7,line width=\lw]
				\node[anchor=north] at (0,0) (A) {$\scriptstyle{i}$};
				\node[anchor=north] at (1,0) (B) {$\scriptstyle{i}$};
				\node[anchor=north] at (2,0) (C) {$\scriptstyle{j}$};
				\node[anchor=south] at (2,4.5) (F) {$\scriptstyle{i}$};
				\draw (A.north)--(1,1.5)--(2,3)--(F.south);
				\draw (B.north)-- node[pos=0.2] (a) {}(0,1.5)--(0,3)..controls +(0.2,0.8) and 	+(-0.2,0.8)..(1,3)--(2,1.5)--(C.north);
				\fill (a) circle (\dw);
			\end{tikzpicture}&
			-\begin{tikzpicture}[myscale=0.8, scale=0.35,yscale=0.7,line width=\lw]
				\node[anchor=north] at (0,0) (A) {$\scriptstyle{i}$};
				\node[anchor=north] at (1,0) (B) {$\scriptstyle{i}$};
				\node[anchor=north] at (2,0) (C) {$\scriptstyle{j}$};
				\node[anchor=south] at (2,4.5) (F) {$\scriptstyle{i}$};
				\draw (A.north)--(1,1.5)--(2,3)--(F.south);
				\draw (B.north)--(0,1.5)--(0,3)..controls +(0.2,0.8) and 	+(-0.2,0.8)..(1,3)--(2,1.5)--(C.north);
			\end{tikzpicture}\\
			-\begin{tikzpicture}[myscale=0.8, scale=0.35,yscale=0.7,line width=\lw]
				\node[anchor=north] at (0,0) (A) {$\scriptstyle{j}$};
				\node[anchor=north] at (1,0) (B) {$\scriptstyle{i}$};
				\node[anchor=north] at (2,0) (C) {$\scriptstyle{i}$};
				\node[anchor=south] at (0,6) (D) {$\scriptstyle{i}$};
				\node[anchor=south] at (1,6) (E) {$\scriptstyle{i}$};
				\node[anchor=south] at (2,6) (F) {$\scriptstyle{j}$};
				\draw (B.north)--(2,1.5)--(2,3)--(1,4.5)--(D.south);
				\draw (C.north)-- node[pos=0.2] (a) {}(1,1.5)--(0,3)--(0,4.5)--(E.south);
				\fill (a) circle (\dw);
				\draw (A.north)--(0,1.5)--(1,3)--(2,4.5)--(F.south);
			\end{tikzpicture}+\begin{tikzpicture}[myscale=0.8, scale=0.35,yscale=0.7,line width=\lw]
				\node[anchor=north] at (0,0) (A) {$\scriptstyle{j}$};
				\node[anchor=north] at (1,0) (B) {$\scriptstyle{i}$};
				\node[anchor=north] at (2,0) (C) {$\scriptstyle{i}$};
				\node[anchor=south] at (0,6) (D) {$\scriptstyle{i}$};
				\node[anchor=south] at (1,6) (E) {$\scriptstyle{i}$};
				\node[anchor=south] at (2,6) (F) {$\scriptstyle{j}$};
				\draw (B.north)--(2,1.5)--(2,3)--(1,4.5)--(D.south) node[pos=0.8] (a) {} ;
				\fill (a) circle(\dw);
				\draw (C.north)--(1,1.5)--(0,3)--(0,4.5)--(E.south);
				\draw (A.north)--(0,1.5)--(1,3)--(2,4.5)--(F.south);
			\end{tikzpicture}&
			\begin{tikzpicture}[myscale=0.8, scale=0.35,yscale=0.7,line width=\lw]
				\node[anchor=north] at (0,0) (A) {$\scriptstyle{j}$};
				\node[anchor=north] at (1,0) (B) {$\scriptstyle{i}$};
				\node[anchor=north] at (2,0) (C) {$\scriptstyle{i}$};
				\node[anchor=south] at (0,6) (D) {$\scriptstyle{j}$};
				\node[anchor=south] at (1,6) (E) {$\scriptstyle{i}$};
				\node[anchor=south] at (2,6) (F) {$\scriptstyle{i}$};
				\draw (B.north)--(2,1.5)--(2,3)--(2,4.5)--(E.south);
				\draw (C.north)--node[pos=0.2] (a) {}(1,1.5)--(0,3)--(1,4.5)--(F.south);
				\fill (a) circle (\dw);
				\draw (A.north)--(0,1.5)--(1,3)--(0,4.5)--(D.south);
			\end{tikzpicture}-\begin{tikzpicture}[myscale=0.8, scale=0.35,yscale=0.7,line width=\lw]
				\node[anchor=north] at (0,0) (A) {$\scriptstyle{j}$};
				\node[anchor=north] at (1,0) (B) {$\scriptstyle{i}$};
				\node[anchor=north] at (2,0) (C) {$\scriptstyle{i}$};
				\node[anchor=south] at (0,6) (D) {$\scriptstyle{j}$};
				\node[anchor=south] at (1,6) (E) {$\scriptstyle{i}$};
				\node[anchor=south] at (2,6) (F) {$\scriptstyle{i}$};
				\draw (B.north)--(2,1.5)--(2,3)--(2,4.5)--(E.south) node[pos=0.8] (a) {} ;
				\fill (a) circle(\dw);
				\draw (C.north)--(1,1.5)--(0,3)--(1,4.5)--(F.south);
				\draw (A.north)--(0,1.5)--(1,3)--(0,4.5)--(D.south);
			\end{tikzpicture}&
			-\begin{tikzpicture}[myscale=0.8, scale=0.35,yscale=0.7,line width=\lw]
				\node[anchor=north] at (0,0) (A) {$\scriptstyle{j}$};
				\node[anchor=north] at (1,0) (B) {$\scriptstyle{i}$};
				\node[anchor=north] at (2,0) (C) {$\scriptstyle{i}$};
				\node[anchor=south] at (2,4.5) (F) {$\scriptstyle{i}$};
				\draw (B.north)--(2,1.5)--(2,3)--(F.south);
				\draw (C.north)-- node[pos=0.2] (a) {}(1,1.5)--(0,3)..controls +(0.2,0.8) and 	+(-0.2,0.8)..(1,3)--(0,1.5)--(A.north);
				\fill (a) circle (\dw);
			\end{tikzpicture}&
			-\begin{tikzpicture}[myscale=0.8, scale=0.35,yscale=0.7,line width=\lw]
				\node[anchor=north] at (0,0) (A) {$\scriptstyle{j}$};
				\node[anchor=north] at (1,0) (B) {$\scriptstyle{i}$};
				\node[anchor=north] at (2,0) (C) {$\scriptstyle{i}$};
				\node[anchor=south] at (2,4.5) (F) {$\scriptstyle{i}$};
				\draw (B.north)--(2,1.5)--(2,3)--(F.south);
				\draw (C.north)--(1,1.5)--(0,3)..controls +(0.2,0.8) and 	+(-0.2,0.8)..(1,3)--(0,1.5)--(A.north);
			\end{tikzpicture}\\
			-\begin{tikzpicture}[myscale=0.8, scale=0.35,yscale=0.7,line width=\lw]
				\node[anchor=north] at (0,1.5) (A) {$\scriptstyle{i}$};
				\node[anchor=south] at (0,6) (D) {$\scriptstyle{i}$};
				\node[anchor=south] at (1,6) (E) {$\scriptstyle{i}$};
				\node[anchor=south] at (2,6) (F) {$\scriptstyle{j}$};
				\draw (F.south)--(2,4.5)--(1,3)..controls +(0.2,-0.8) and 	+(-0.2,-0.8)..(2,3)--(1,4.5)--(D.south);
				\draw (A.north)--(0,3)--(0,4.5)--(E.south);
			\end{tikzpicture}&
			\begin{tikzpicture}[myscale=0.8, scale=0.35,yscale=0.7,line width=\lw]
				\node[anchor=north] at (0,1.5) (A) {$\scriptstyle{i}$};
				\node[anchor=south] at (0,6) (D) {$\scriptstyle{j}$};
				\node[anchor=south] at (1,6) (E) {$\scriptstyle{i}$};
				\node[anchor=south] at (2,6) (F) {$\scriptstyle{i}$};
				\draw (D.south)--(0,4.5)--(1,3)..controls +(0.2,-0.8) and 	+(-0.2,-0.8)..(2,3)--(2,4.5)--(E.south);
				\draw (A.north)--(0,3)--(1,4.5)--(F.south);
			\end{tikzpicture}&
			-\begin{tikzpicture}[myscale=0.8, scale=0.35,yscale=0.7,line width=\lw]
				\node[anchor=north] at (0,1.5) (A) {$\scriptstyle{i}$};
				\node[anchor=south] at (2,4.5) (F) {$\scriptstyle{i}$};
				\draw (A.north)--(0,3)..controls +(0.2,0.8) and 	+(-0.2,0.8)..(1,3)..controls +(0.2,-0.8) and 	+(-0.2,-0.8)..(2,3)--(F.south);
			\end{tikzpicture}&0\\
			\begin{tikzpicture}[myscale=0.8, scale=0.35,yscale=0.7,line width=\lw]
				\node[anchor=north] at (0,1.5) (A) {$\scriptstyle{i}$};
				\node[anchor=south] at (0,6) (D) {$\scriptstyle{i}$};
				\node[anchor=south] at (1,6) (E) {$\scriptstyle{i}$};
				\node[anchor=south] at (2,6) (F) {$\scriptstyle{j}$};
				\draw (F.south)--(2,4.5)--(1,3)..controls +(0.2,-0.8) and 	+(-0.2,-0.8)..(2,3)--(1,4.5)--(D.south) node[pos=0.8] (a) {} ;
				\fill (a) circle(\dw);
				\draw (A.north)--(0,3)--(0,4.5)--(E.south);
			\end{tikzpicture}&
			-\begin{tikzpicture}[myscale=0.8, scale=0.35,yscale=0.7,line width=\lw]
				\node[anchor=north] at (0,1.5) (A) {$\scriptstyle{i}$};
				\node[anchor=south] at (0,6) (D) {$\scriptstyle{j}$};
				\node[anchor=south] at (1,6) (E) {$\scriptstyle{i}$};
				\node[anchor=south] at (2,6) (F) {$\scriptstyle{i}$};
				\draw (D.south)--(0,4.5)--(1,3)..controls +(0.2,-0.8) and 	+(-0.2,-0.8)..(2,3)--(2,4.5)--(E.south) node[pos=0.8] (a) {} ;
				\fill (a) circle(\dw);
				\draw (A.north)--(0,3)--(1,4.5)--(F.south);
			\end{tikzpicture}&0&
			-\begin{tikzpicture}[myscale=0.8, scale=0.35,yscale=0.7,line width=\lw]
				\node[anchor=north] at (0,1.5) (A) {$\scriptstyle{i}$};
				\node[anchor=south] at (2,4.5) (F) {$\scriptstyle{i}$};
				\draw (A.north)--(0,3)..controls +(0.2,0.8) and 	+(-0.2,0.8)..(1,3)..controls +(0.2,-0.8) and 	+(-0.2,-0.8)..(2,3)--(F.south);
			\end{tikzpicture}
		\end{pmatrix}\\
		\overset{\left(\begin{smallmatrix}\eqref{ncinverse}&\eqref{ncbraid}&\eqref{nctangleone}&\eqref{nctangleone}\\\eqref{ncbraid}&\eqref{ncinverse}&\eqref{ncuntwist}&\eqref{ncuntwist}\\\eqref{ncuntwist}&\eqref{nctangleone}&\eqref{ncsnake}&0\\\eqref{ncuntwist}&\eqref{nctangleone}&0&\eqref{ncsnake}&\end{smallmatrix}\right)}{=}&\begin{pmatrix}
			\begin{tikzpicture}[myscale=0.8, scale=0.35,yscale=0.7,line width=\lw]
				\node[anchor=north] at (0,0) (A) {$\scriptstyle{i}$};
				\node[anchor=north] at (1,0) (B) {$\scriptstyle{i}$};
				\node[anchor=north] at (2,0) (C) {$\scriptstyle{j}$};
				\node[anchor=south] at (0,3) (D) {$\scriptstyle{i}$};
				\node[anchor=south] at (1,3) (E) {$\scriptstyle{i}$};
				\node[anchor=south] at (2,3) (F) {$\scriptstyle{j}$};
				\node at (1,1.5) (a) {} ;
				\fill (a) circle(\dw);
				\draw (A.north)--(1,1.5)--(D.south);
				\draw (B.north)-- node[pos=0.2] (a) {}(0,1.5)--(E.south);
				\fill (a) circle (\dw);
				\draw (C.north)--(F.south);
			\end{tikzpicture}-\begin{tikzpicture}[myscale=0.8, scale=0.35,yscale=0.7,line width=\lw]
				\node[anchor=north] at (0,0) (A) {$\scriptstyle{i}$};
				\node[anchor=north] at (1,0) (B) {$\scriptstyle{i}$};
				\node[anchor=north] at (2,0) (C) {$\scriptstyle{j}$};
				\node[anchor=south] at (0,3) (D) {$\scriptstyle{i}$};
				\node[anchor=south] at (1,3) (E) {$\scriptstyle{i}$};
				\node[anchor=south] at (2,3) (F) {$\scriptstyle{j}$};
				\node at (1,1.5) (a) {} ;
				\fill (a) circle(\dw);
				\draw (A.north)--(1,1.5)--(D.south) node[pos=0.8] (a) {} ;
				\fill (a) circle(\dw);
				\draw (B.north)--(0,1.5)--(E.south);
				\draw (C.north)--(F.south);
			\end{tikzpicture}&
			\begin{tikzpicture}[myscale=0.8, scale=0.35,yscale=0.7,line width=\lw]
				\node[anchor=north] at (0,0) (A) {$\scriptstyle{i}$};
				\node[anchor=north] at (1,0) (B) {$\scriptstyle{i}$};
				\node[anchor=north] at (2,0) (C) {$\scriptstyle{j}$};
				\node[anchor=south] at (0,6) (D) {$\scriptstyle{j}$};
				\node[anchor=south] at (1,6) (E) {$\scriptstyle{i}$};
				\node[anchor=south] at (2,6) (F) {$\scriptstyle{i}$};
				\draw (A.north)--(1,1.5)--(0,3)--(0,4.5)--(E.south);
				\draw (B.north)-- node[pos=0.2] (a) {}(0,1.5)--(1,3)--(2,4.5)--(F.south);
				\fill (a) circle (\dw);
				\draw (C.north)--(2,1.5)--(2,3)--(1,4.5)--(D.south);
			\end{tikzpicture}-\begin{tikzpicture}[myscale=0.8, scale=0.35,yscale=0.7,line width=\lw]
				\node[anchor=north] at (0,0) (A) {$\scriptstyle{i}$};
				\node[anchor=north] at (1,0) (B) {$\scriptstyle{i}$};
				\node[anchor=north] at (2,0) (C) {$\scriptstyle{j}$};
				\node[anchor=south] at (0,6) (D) {$\scriptstyle{j}$};
				\node[anchor=south] at (1,6) (E) {$\scriptstyle{i}$};
				\node[anchor=south] at (2,6) (F) {$\scriptstyle{i}$};
				\draw (A.north)--(1,1.5)--(0,3)--(0,4.5)--(E.south) node[pos=0.8] (a) {} ;
				\fill (a) circle(\dw);
				\draw (B.north)--(0,1.5)--(1,3)--(2,4.5)--(F.south);
				\draw (C.north)--(2,1.5)--(2,3)--(1,4.5)--(D.south);
			\end{tikzpicture}&
			-\begin{tikzpicture}[myscale=0.8, scale=0.35,yscale=0.7,line width=\lw]
				\node[anchor=north] at (0,0) (A) {$\scriptstyle{i}$};
				\node[anchor=north] at (1,0) (B) {$\scriptstyle{i}$};
				\node[anchor=north] at (2,0) (C) {$\scriptstyle{j}$};
				\node[anchor=south] at (0,4.5) (F) {$\scriptstyle{i}$};
				\draw (A.north)--(1,1.5)--(0,3)--(F.south);
				\draw (B.north)--node[pos=0.2] (a) {}(0,1.5)--(1,3)..controls +(0.2,0.8) and 	+(-0.2,0.8)..(2,3)--(2,1.5)--(C.north);
				\fill (a) circle (\dw);
			\end{tikzpicture}&
			-\begin{tikzpicture}[myscale=0.8, scale=0.35,yscale=0.7,line width=\lw]
				\node[anchor=north] at (0,0) (A) {$\scriptstyle{i}$};
				\node[anchor=north] at (1,0) (B) {$\scriptstyle{i}$};
				\node[anchor=north] at (2,0) (C) {$\scriptstyle{j}$};
				\node[anchor=south] at (0,4.5) (F) {$\scriptstyle{i}$};
				\draw (A.north)--(1,1.5)--(0,3)--(F.south);
				\draw (B.north)--(0,1.5)--(1,3)..controls +(0.2,0.8) and 	+(-0.2,0.8)..(2,3)--(2,1.5)--(C.north);
			\end{tikzpicture}\\
			-\begin{tikzpicture}[myscale=0.8, scale=0.35,yscale=0.7,line width=\lw]
				\node[anchor=north] at (0,0) (A) {$\scriptstyle{j}$};
				\node[anchor=north] at (1,0) (B) {$\scriptstyle{i}$};
				\node[anchor=north] at (2,0) (C) {$\scriptstyle{i}$};
				\node[anchor=south] at (0,6) (D) {$\scriptstyle{i}$};
				\node[anchor=south] at (1,6) (E) {$\scriptstyle{i}$};
				\node[anchor=south] at (2,6) (F) {$\scriptstyle{j}$};
				\draw (B.north)--(2,1.5)--(1,3)--(0,4.5)--(D.south);
				\draw (C.north)-- node[pos=0.2] (a) {}(1,1.5)--(2,3)--(2,4.5)--(E.south);
				\fill (a) circle (\dw);
				\draw (A.north)--(0,1.5)--(0,3)--(1,4.5)--(F.south);
			\end{tikzpicture}+\begin{tikzpicture}[myscale=0.8, scale=0.35,yscale=0.7,line width=\lw]
				\node[anchor=north] at (0,0) (A) {$\scriptstyle{j}$};
				\node[anchor=north] at (1,0) (B) {$\scriptstyle{i}$};
				\node[anchor=north] at (2,0) (C) {$\scriptstyle{i}$};
				\node[anchor=south] at (0,6) (D) {$\scriptstyle{i}$};
				\node[anchor=south] at (1,6) (E) {$\scriptstyle{i}$};
				\node[anchor=south] at (2,6) (F) {$\scriptstyle{j}$};
				\draw (B.north)--(2,1.5)--(1,3)--(0,4.5)--(D.south) node[pos=0.8] (a) {} ;
				\fill (a) circle(\dw);
				\draw (C.north)--(1,1.5)--(2,3)--(2,4.5)--(E.south);
				\draw (A.north)--(0,1.5)--(0,3)--(1,4.5)--(F.south);
			\end{tikzpicture}&
			-\begin{tikzpicture}[myscale=0.8, scale=0.35,yscale=0.7,line width=\lw]
				\node[anchor=north] at (0,0) (A) {$\scriptstyle{j}$};
				\node[anchor=north] at (1,0) (B) {$\scriptstyle{i}$};
				\node[anchor=north] at (2,0) (C) {$\scriptstyle{i}$};
				\node[anchor=south] at (0,3) (D) {$\scriptstyle{j}$};
				\node[anchor=south] at (1,3) (E) {$\scriptstyle{i}$};
				\node[anchor=south] at (2,3) (F) {$\scriptstyle{i}$};
				\node at (1,1.5) (a) {} ;
				\fill (a) circle(\dw);
				\draw (B.north)--(2,1.5)--(E.south);
				\draw (C.north)-- node[pos=0.2] (a) {}(1,1.5)--(F.south);
				\fill (a) circle (\dw);
				\draw (A.north)--(D.south);
			\end{tikzpicture}+\begin{tikzpicture}[myscale=0.8, scale=0.35,yscale=0.7,line width=\lw]
				\node[anchor=north] at (0,0) (A) {$\scriptstyle{j}$};
				\node[anchor=north] at (1,0) (B) {$\scriptstyle{i}$};
				\node[anchor=north] at (2,0) (C) {$\scriptstyle{i}$};
				\node[anchor=south] at (0,3) (D) {$\scriptstyle{j}$};
				\node[anchor=south] at (1,3) (E) {$\scriptstyle{i}$};
				\node[anchor=south] at (2,3) (F) {$\scriptstyle{i}$};
				\node at (1,1.5) (a) {} ;
				\fill (a) circle(\dw);
				\draw (B.north)--(2,1.5)--(E.south) node[pos=0.8] (a) {} ;
				\fill (a) circle(\dw);
				\draw (C.north)--(1,1.5)--(F.south);
				\draw (A.north)--(D.south);
			\end{tikzpicture}&
			0&0\\
			0&
			\begin{tikzpicture}[myscale=0.8, scale=0.35,yscale=0.7,line width=\lw]
				\node[anchor=north] at (2,1.5) (A) {$\scriptstyle{i}$};
				\node[anchor=south] at (0,6) (D) {$\scriptstyle{j}$};
				\node[anchor=south] at (1,6) (E) {$\scriptstyle{i}$};
				\node[anchor=south] at (2,6) (F) {$\scriptstyle{i}$};
				\draw (D.south)--(0,4.5)--(0,3)..controls +(0.2,-0.8) and 	+(-0.2,-0.8)..(1,3)--(2,4.5)--(E.south);
				\draw (A.north)--(2,3)--(1,4.5)--(F.south);
			\end{tikzpicture}&
			\begin{tikzpicture}[myscale=0.8, scale=0.35,yscale=0.7,line width=\lw]
				\node[anchor=north] at (0,1.5) (A) {$\scriptstyle{i}$};
				\node[anchor=south] at (0,4.5) (F) {$\scriptstyle{i}$};
				\draw (A.north)--(F.south);
			\end{tikzpicture}&0\\
			0&
			\begin{tikzpicture}[myscale=0.8, scale=0.35,yscale=0.7,line width=\lw]
				\node[anchor=north] at (2,1.5) (A) {$\scriptstyle{i}$};
				\node[anchor=south] at (0,6) (D) {$\scriptstyle{j}$};
				\node[anchor=south] at (1,6) (E) {$\scriptstyle{i}$};
				\node[anchor=south] at (2,6) (F) {$\scriptstyle{i}$};
				\draw (D.south)--(0,4.5)--(0,3)..controls +(0.2,-0.8) and 	+(-0.2,-0.8)..(1,3)--(2,4.5)--(E.south) node[pos=0.8] (a) {} ;
				\fill (a) circle(\dw);
				\draw (A.north)--(2,3)--(1,4.5)--(F.south);
			\end{tikzpicture}&0&
			\begin{tikzpicture}[myscale=0.8, scale=0.35,yscale=0.7,line width=\lw]
				\node[anchor=north] at (0,1.5) (A) {$\scriptstyle{i}$};
				\node[anchor=south] at (0,4.5) (F) {$\scriptstyle{i}$};
				\draw (A.north)--(F.south);
			\end{tikzpicture}
		\end{pmatrix}
	\end{align*}
	Applying now \cref{ncdotcrossing,ncinverse,ncdotcrossingii}, one obtains the identity matrix. 
\end{proof}

\bibliography{biblio}
\end{document}